\documentclass[11pt,a4paper]{article}

\usepackage{natbib}

\usepackage{apalike}
\usepackage{comment}
\usepackage{lipsum}
\usepackage{microtype}
\usepackage{subcaption}

\usepackage[font=small,labelfont=bf]{caption}

\usepackage{setspace}
\usepackage{graphicx}
\graphicspath{{./figures}}
\usepackage{bbm}
\usepackage{verbatim}
\usepackage{multicol,latexsym, array}
\usepackage{comment}
\usepackage{textcomp}
\usepackage[ansinew]{inputenc}
\usepackage{amsmath,amssymb, amsthm}
\usepackage[OT1]{fontenc}
\usepackage[english]{babel}
\usepackage{hyperref}

\hypersetup{
   colorlinks,
   linkcolor={blue},
   citecolor={blue},
   urlcolor={blue}
}

\usepackage{enumitem}
\usepackage{MnSymbol} 
\usepackage{empheq}
\usepackage{rotating}
\usepackage{titletoc}
\usepackage{framed}
\usepackage{multirow,bigdelim}
\usepackage{bbm}
\usepackage[noend]{algpseudocode}
\usepackage{algorithm}
\usepackage{pgf}
\usepackage{pgfplots}
\usepackage{tikz}
\pgfplotsset{compat=1.17}
\usepackage{shadethm}
\usepackage{thmtools}
\usepackage{mdframed}
\usepackage{lipsum}
\usepackage{fixmath}  
\usepackage{multicol}
\usepackage{footmisc}

\let\oldXi\Xi 
\renewcommand{\Xi}{\mathrm{\oldXi}}
\let\oldLambda\Lambda 
\renewcommand{\Lambda}{\mathrm{\oldLambda}}
\let\oldPi\Pi 
\renewcommand{\Pi}{\mathrm{\oldPi}}
\let\oldPsi\Psi 
\renewcommand{\Psi}{\mathrm{\oldPsi}}

\usepackage{caption}
\usepackage{xargs}  
\usepackage{xcolor} 

\definecolor{DarkGray}{RGB}{90,90,90}

\usepackage[colorinlistoftodos,prependcaption,textsize=tiny]{todonotes}
\usepackage{cleveref}

\date{\today}

\newenvironment{ackno}%
    {
    \paragraph{Acknowledgements:} 
    }

\usepackage{caption}
\usepackage{mathrsfs}
\usetikzlibrary{arrows}

\usepackage{xpatch}

\makeatletter
\xpatchcmd{\endmdframed}
  {\aftergroup\endmdf@trivlist\color@endgroup}
  {\endmdf@trivlist\color@endgroup\@doendpe}
  {}{}
\makeatother

\usepackage{nicefrac}
\usepackage{dsfont}

\usepackage{bigints}
\usepackage{mathtools}

\usepackage{mathtools}
\usepackage{color}
\usepackage{fancyhdr}
\usepackage{lastpage}

\usepackage{ifthen}

\newcommand{\OT}{\mathrm{OT}}

\newcommand{\KR}{\mathrm{KR}}

\newcommand{\BB}{\mathbb{B}}

\newcommand{\EE}{\mathbb{E}}

\newcommand{\NN}{\mathbb{N}}

\newcommand{\PP}{\mathbb{P}}

\newcommand{\RR}{\mathbb{R}}

\newcommand{\ZZ}{\mathbb{Z}}

\newcommand{\MCsX}{\mathcal{M}(\XC)}

\newcommand{\Poi}{\textup{Poi}}
\newcommand{\mun}{\hat{\mu}_t}
\newcommand{\nun}{\hat{\nu}_t}

\newcommand{\BC}{\mathcal{B}}
\newcommand{\CC}{\mathcal{C}}
\newcommand{\DC}{\mathcal{D}}

\newcommand{\IC}{\mathcal{I}}

\newcommand{\LC}{\mathcal{L}}
\newcommand{\MC}{\mathcal{M}}
\newcommand{\NC}{\mathcal{N}}

\newcommand{\PC}{\mathcal{P}}

\newcommand{\TC}{\mathcal{T}}

\newcommand{\XC}{\mathcal{X}}

\newcommand{\ZC}{\mathcal{Z}}

\newcommand{\norm}[1]{\left\lVert#1\right\rVert}

\newcommand{\Leb}{\mathrm{Leb}}

\newcommand{\id}{\,\mathrm{id}}
\newcommand{\dif}{\, d}  

\newcommand{\supp}{\mathrm{supp}}

\newcommand{\Cov}{\operatorname{Cov}}

\newcommand{\Var}{\mathrm{Var}}

\newcommand{\iid}{\stackrel{\!\mathrm{i.i.d.}}{\sim}\!}

\newcommand{\eps}{\varepsilon}
\renewcommand{\phi}{\varphi}

\newcommand{\Eta}{\mathrm{H}}
\newcommand{\Zeta}{\mathrm{Z}}
\DeclareMathOperator{\pop}{Pop}
\newcommand{\temp}{temp} 
\newcommand{\spat}{spat} 
\newcommand{\eqinlaw}{\stackrel{\mathcal{D}}{=}}
\newcommand{\XCbar}{\bar{\XC}}

\newcommand{\Mbar}{\bar{M}}
\newcommand{\gbar}{\bar{g}}
\newcommand{\nubar}{\bar{\nu}}

\newcommand{\tgz}{\TC_{>0}}
\newcommand{\tgeqz}{\TC_{\geq 0}}

\newcommand{\coloneqq}{:=}
\newcommand{\eqqcolon}{=:}

\newcommand{\U}{\mathrm{D}}

\newcommand{\RNum}[1]{\uppercase\expandafter{\romannumeral #1\relax}}

\newcommand{\diam}{\mathrm{diam}}

\newcommand{\konvD}{\xrightarrow{\;\;w\;\;}}

 \newcommand{\konvP}{\xrightarrow{\;\;p\;\;}}

\makeatletter
\newcommand{\labitem}[2]{%
	\def\@itemlabel{\textbf{#1}}
	\item
	\def\@currentlabel{#1}\label{#2}}
\makeatother

\newcommand{\KL}[1]{\mathrm{KL}(#1)} 
\renewcommand*{\epsilon}{\varepsilon}

\DeclareMathOperator*{\argmin}{\mathrm{argmin}}

\theoremstyle{plain}
\newtheorem{theorem}{Theorem}[section]
\newtheorem{corollary}[theorem]{Corollary}
\newtheorem{lemma}[theorem]{Lemma}
\newtheorem{proposition}[theorem]{Proposition}
\newtheorem*{theorem*}{Theorem}

\theoremstyle{definition}
\newtheorem{definition}[theorem]{Definition}
\newtheorem{remark}[theorem]{Remark}
\newtheorem{example}[theorem]{Example}
\newtheorem*{example*}{Example}

\makeatletter
\def\mysequence#1{\expandafter\@mysequence\csname c@#1\endcsname}
\def\@mysequence#1{%
  \ifcase#1\or (A)\or (B)\or (C)\or (D)\else\@ctrerr\fi}
\makeatother

\makeatletter
\newcommand{\mylabel}[2]{#2\def\@currentlabel{#2}\label{#1}}
\makeatother

\numberwithin{equation}{section}

\newcommand{\footremember}[2]{
	\footnote{#2}
	\newcounter{#1}
	\setcounter{#1}{\value{footnote}}
}

\newcommand{\footrecall}[1]{
	\footnotemark[\value{#1}]
}

\makeatletter

\xdefinecolor{dgreenish}{rgb}{0,0.5,0.2}

\newif\ifexclude
\excludefalse  

\usepackage{calc}
\usepackage{accents}

\usepackage
[
        a4paper,
        left=3cm,
        right=3cm,
        top=3cm,
        bottom=3cm
]
{geometry}

\begin{document}
  \title{Sharp Convergence Rates of Empirical Unbalanced Optimal Transport for Spatio-Temporal Point Processes}
  
\author{Marina Struleva
  \hspace{-0.3em}\footremember{equalcontr}{These authors contributed equally}
  \hspace{-0.7em}\footremember{ims}{\scriptsize
    Institute for Mathematical
		Stochastics, University of G\"ottingen,
		Goldschmidtstra{\ss}e 7, 37077 G\"ottingen, Germany}%
	\\
  \footnotesize{\href{mailto:marina.struleva@uni-goettingen.de}{marina.struleva@uni-goettingen.de}}
  \\[2ex]
	Shayan Hundrieser
  \hspace{-0.3em}\footrecall{equalcontr}
  \hspace{-0.3em}\footrecall{ims}\footremember{dam}{\scriptsize
   Department of Applied Mathematics, University of Twente, Drienerlolaan 5, 7522 NB Enschede, The Netherlands}%
	\\
  \footnotesize{\href{mailto:s.hundrieser@utwente.nl}{s.hundrieser@utwente.nl}}
  \\[2ex]
  Dominic Schuhmacher
  \hspace{-0.3em}\footrecall{ims}%
	\\
  \footnotesize{\href{mailto:schuhmacher@math.uni-goettingen.de}{schuhmacher@math.uni-goettingen.de}}
  \\[2ex]
	Axel Munk
  \hspace{-0.3em}\footrecall{ims}%
  \hspace{-0.05em}\footremember{mbexc}{\scriptsize
    Cluster of Excellence "Multiscale Bioimaging: from Molecular Machines to Networks of Excitable Cells" (MBExC),
    University Medical Center,
    Robert-Koch-Stra{\ss}e 40, 37075 G\"ottingen, Germany}
	\\
  \footnotesize{\href{mailto:munk@math.uni-goettingen.de}{munk@math.uni-goettingen.de}}
}

\pagenumbering{arabic}

\maketitle

\begin{abstract}
\noindent We statistically analyze empirical plug-in estimators for unbalanced optimal transport (UOT) formalisms, focusing on the Kantorovich--Rubinstein distance, between general intensity measures based on observations from spatio-temporal point processes. Specifically, we model the observations by two weakly time-stationary point processes with spatial intensity measures $\mu$ and $\nu$ over the expanding window $(0,t]$ as $t$ increases to infinity, and establish sharp convergence rates of the empirical UOT in terms of the intrinsic dimensions of the measures. We assume a sub-quadratic temporal growth condition of the variance of the process, which allows for a wide range of temporal dependencies. As the growth approaches quadratic, the convergence rate becomes slower. This variance assumption is related to the time-reduced factorial covariance measure, and we exemplify its validity for various point processes, including the Poisson cluster, Hawkes, Neyman--Scott and log-Gaussian Cox processes. Complementary to our upper bounds, we also derive matching lower bounds for various spatio-temporal point processes of interest and establish near minimax rate optimality of the empirical Kantorovich--Rubinstein distance.
\end{abstract}
\vspace{0.5cm}
\vspace{0.5cm}




\pagenumbering{Roman}

\newpage
\tableofcontents

\newpage
\pagenumbering{arabic}
\section{Introduction}
\subsection{From balanced to unbalanced optimal transport} \label{sec:intro_uots}
\emph{Optimal transport (OT)}, also called the Monge-Kantorovich transport, aims to transform a given probability distribution into another given one while minimizing the average transformation cost. Origins of the OT problem date back to the works of \cite{monge} and \cite{kant1942_original}; see also \cite{rachev1998massTheory, rachev1998massApplications, vil03, villani2008optimal, santambrogio2015optimal, bobkov2019one, chewi2024statistical} for comprehensive treatment and different aspects of the topic. 
Applications of this concept span across machine learning \citep{courty2016optimal, arjovsky2017wasserstein, altschuler2017near, dvurechensky2018computational, sommerfeld19FastProb},  computational biology \citep{Schiebinger19,tameling2021Colocalization,wang2021optimal, bunne2023learning} and statistics \citep{Munk98,del1999tests,evans2012phylogenetic, sommerfeld2018, panaretos2019statistical, hallin2021distribution}. 

To formalize the corresponding optimization problem, throughout this paper we consider two probability measures $\mu, \nu \in \PC(\XC)$ on a \emph{complete separable metric space (c.s.m.s.)} $\XC$ and let $c: \XC\times \XC \to [0, \infty)$ be a continuous function, which we refer to as cost function. Then, the \emph{optimal transport cost} between $\mu$ and $\nu$  is defined as  
\begin{equation} \label{eq:def_OT}
    \OT_c(\mu, \nu) \coloneqq \inf_{\pi\in \Pi_{=}(\mu, \nu)} \int_{\XC\times\XC} c(x_1, x_2) \pi(d x_1, \dif x_2), 
\end{equation}
where $\Pi_{=}(\mu, \nu)$ is the collection of Borel measures $\pi$ on $\XC\times \XC$ with marginals $\mu$ and $\nu$, termed balanced transport plans between $\mu$ and $\nu$. We stress that  the subscript ``$=$'' in $\Pi_{=}$ attributes to the fact that the mass is assumed to be preserved in \eqref{eq:def_OT}, i.e., for any two Borel sets $A,B\in\BC(\XC)$ from the Borel $\sigma$-field on $\XC$ it holds 
$$
    \pi(A\times \XC) = \mu(A) \quad \text{and} \quad  \pi(\XC\times B) = \nu(B).
$$
It is well known that an optimizer to \eqref{eq:def_OT} always exists and is termed \emph{optimal transport plan} \citep{villani2008optimal}.
By taking the underlying ground cost as a power of a metric $d:\XC\times\XC \to [0, \infty)$ on $\XC$, i.e., $c(x,y) = d^p(x,y) $ for $p\geq 1$, the OT cost serves to define the \emph{Wasserstein-$p$ metric} \citep{vaserstein1969markov},
\begin{align}\label{eq:WassersteinDistance}
    W_p(\mu, \nu) := \OT_{d^p}^{1/p}(\mu, \nu) = \inf_{\pi\in \Pi_{=}(\mu, \nu)} \int_{\XC\times\XC} d^p(x_1, x_2) \pi(d x_1, \dif x_2),
\end{align}
also known as Kantorovich distance \citep{vershik2006kantorovich}, which is a metric on the space of  probability measures with finite $p$-th moment. In particular, Wasserstein distances are faithful to the geometry of the underlying space and are highly discriminative: according to \cite{savare2022simple}, $W_p^p$ is the largest functional on the space of probability measures which is jointly convex and fulfills $W_p(\delta_x, \delta_y) = d(x,y)$. 

For the definition of the OT cost, it is crucial that the measures $\mu$ and $\nu$ admit identical (finite) total mass (which then can be assumed to be 1); otherwise, the set $\Pi_{=}(\mu, \nu)$ is empty. However, for many practical applications, the assumption of equal masses is often not appropriate, while the concept of mass transfer is still appealing, in principle. To exemplify, we mention image classification \citep{rubner1997earth,pele2008linear, chizat16, lee2019parallel}, where the total intensities of the images might be different. Similar issues arise when processing brain activity scans \citep{gramfort2015fast}, studying radiation patterns of collider events \citep{komiske2019metric,manole2022background},  and tracking dynamics of growing cells \citep{schiebinger2017reconstruction}.

A simple way to extend the OT cost for measures $\mu$ and $\nu$ with different total masses is to normalize the measures to $\tilde \mu = \mu/\mu(\XC)$ and $\tilde \nu = \nu/\nu(\XC)$ and to consider the OT cost thereof. However, this procedure changes the nature of the OT problem. For example, given two point clouds of different sizes, it is often conceptually preferable to match points one-to-one, leaving some points of the larger or both clouds unmatched. This is especially the case if the points represent individual objects (molecules, nuclei of cells, trees, etc.) or events (times of neural spikes, locations of crimes, earthquakes, etc.) and OT-like distances based on such sub-matchings have a long-standing tradition in this context  \citep{victor1997metric, schuhmacher2008consistent, diez2012algorithms, muller2020metrics, sukegawa2024computing}. The renormalization approach, on the other hand, inevitably leads to an OT plan which lacks the structure of a matching and destroys the natural meaning of points as individual entities by redistributing mass between them, see Figure~\ref{fig:UOTmatching}(a) versus (b\,--\,d). 

The situation is similar in the semi-discrete setting, in particular when comparing a point cloud to a measure with density with respect to Lebesgue measure; 
see the introduction of \cite{bourne2018semi} and their concrete example in Figure 7.

The conceptual drawbacks of the renormalization approach have led to various extensions of OT, of which we only provide a selective overview below. 
To keep the presentation concise, this work is mainly concerned with the $(p,C)$-Kantorovich--Rubinstein distance (KRD) (for extensions, see \Cref{sec:other_UOTs}), which was recently introduced by \cite{heinemann2023kantorovich} for parameters $p\geq 1$ and $C>0$ and forms a natural extension of the unbalanced OT formalism by \cite{KantorovichRubinstein1958}. While the work by \cite{heinemann2023kantorovich} only formally introduces the KRD for finite discrete measures, for our purposes, we have to extend their approach to general finite measures on a c.s.m.s.\ $(\XC,d)$. To this end, we denote the space of finite measures on $\XC$ by $\MC(\XC)$, i.e., $m_{\mu}\coloneqq \mu(\XC)< \infty$ for all $\mu\in \MC(\XC)$.

\begin{figure}[!t]
  \begin{subfigure}{100pt}
    \centering
    \includegraphics[width=100pt]{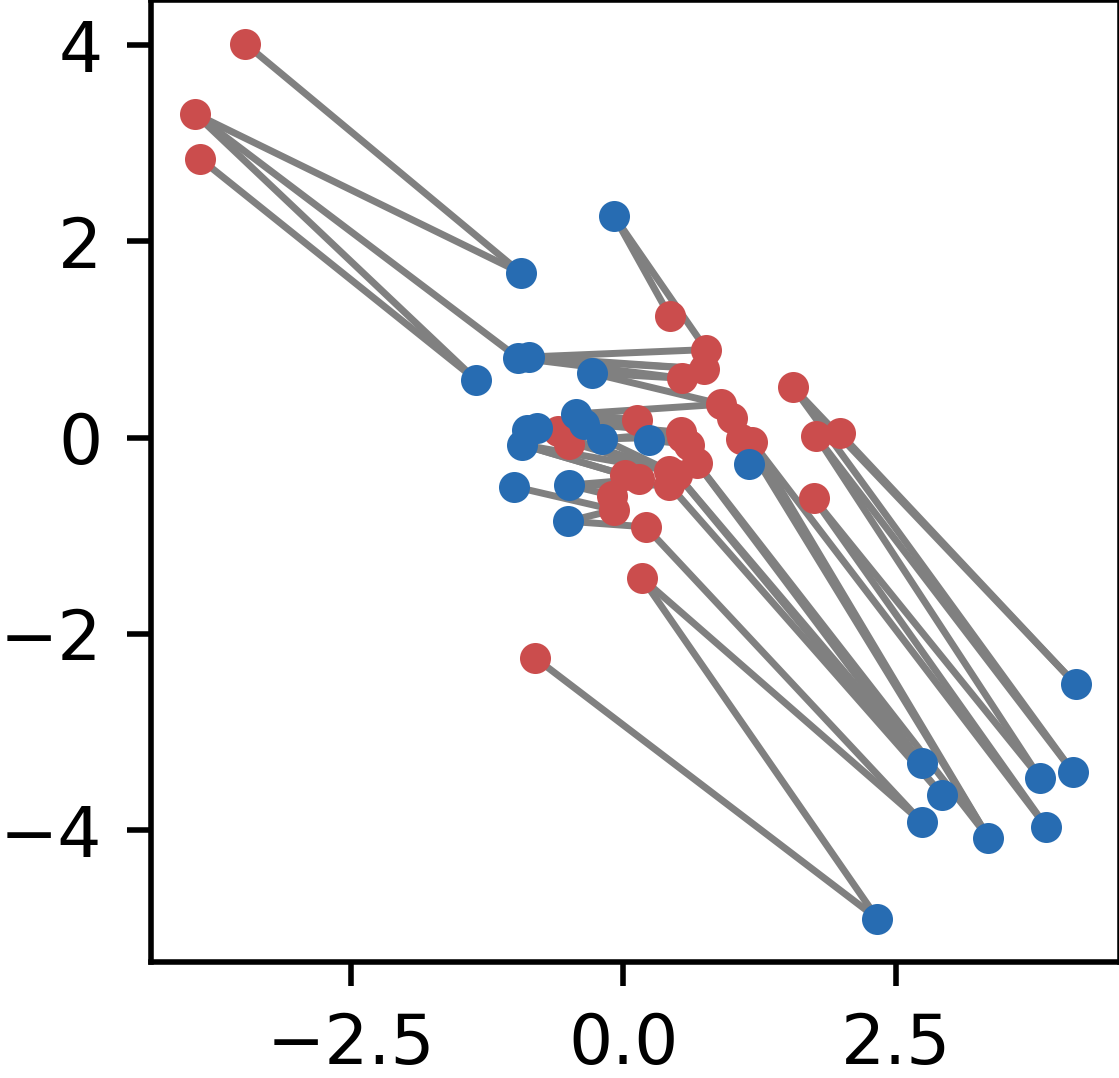}\vspace{0.1cm}
    \caption{Normalized $W_2$-plan}
  \end{subfigure}
	\begin{subfigure}{100pt}
  \centering
  \includegraphics[width=100pt]{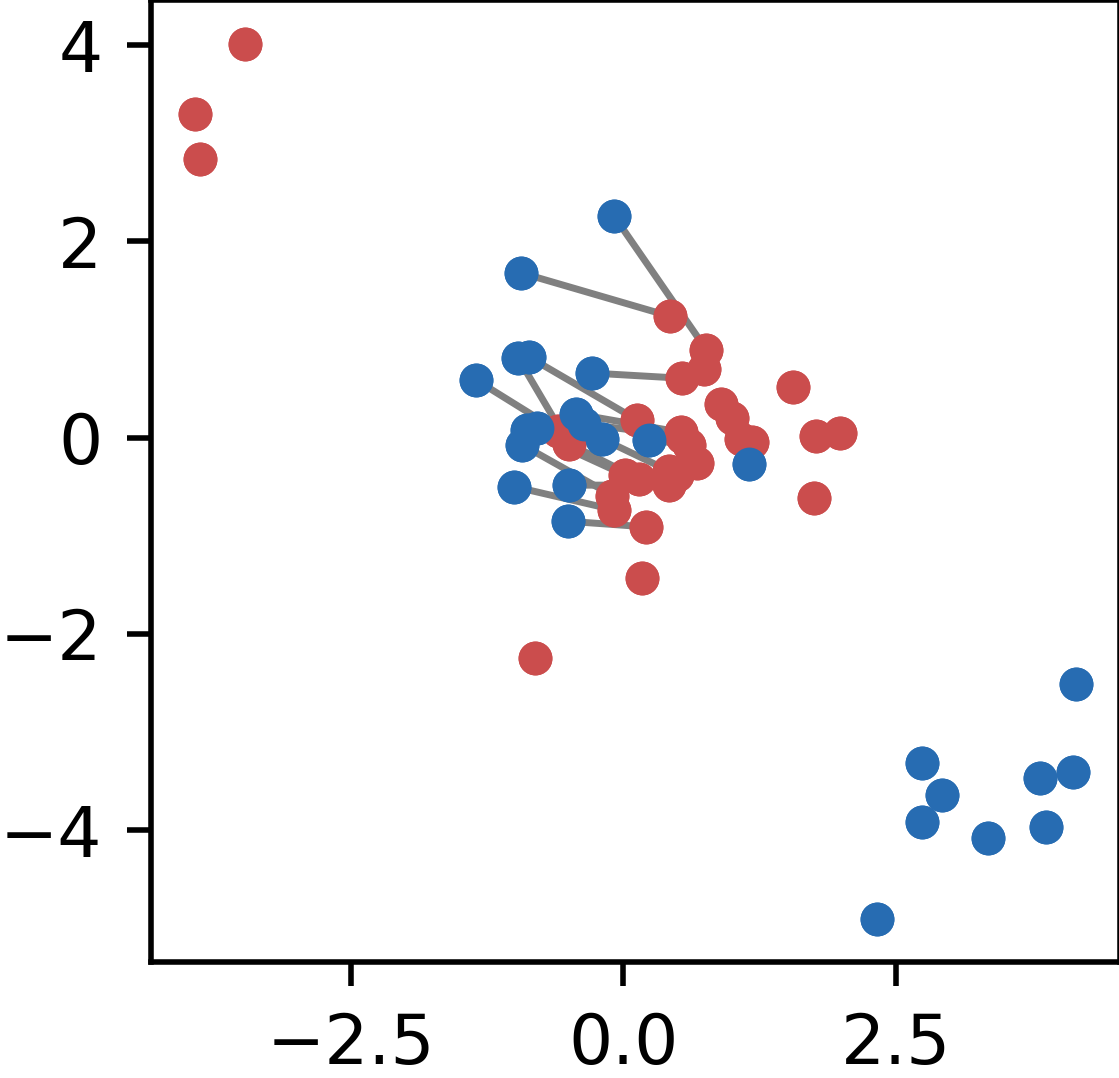}\vspace{0.1cm}
  \caption{$\KR_{2,2}$-plan}
\end{subfigure}
\begin{subfigure}{100pt}
  \centering
  \includegraphics[width=100pt]{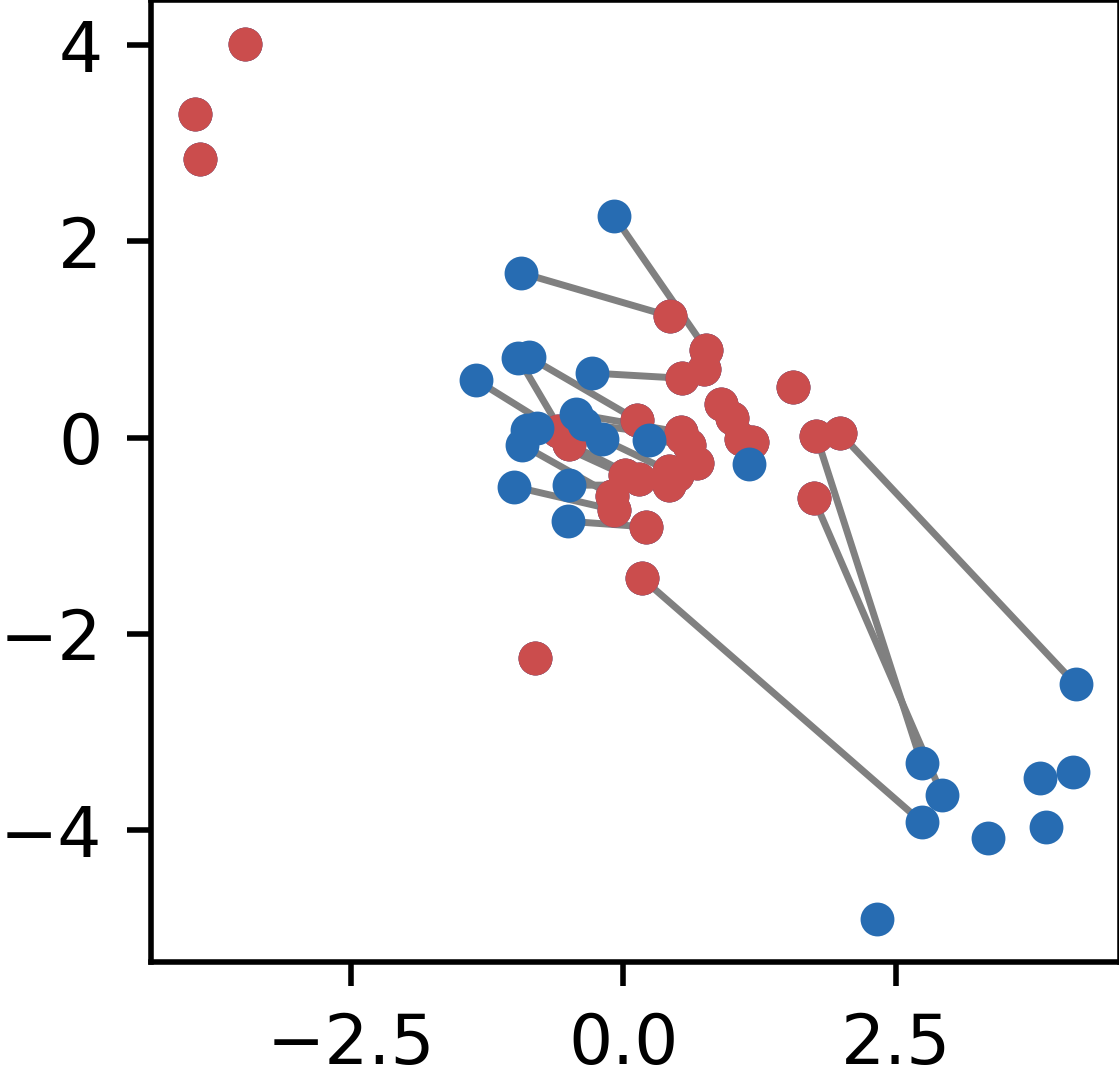}\vspace{0.1cm}
  \caption{$\KR_{2,4}$-plan}
\end{subfigure}
\begin{subfigure}{100pt}
  \centering
  \includegraphics[width=100pt]{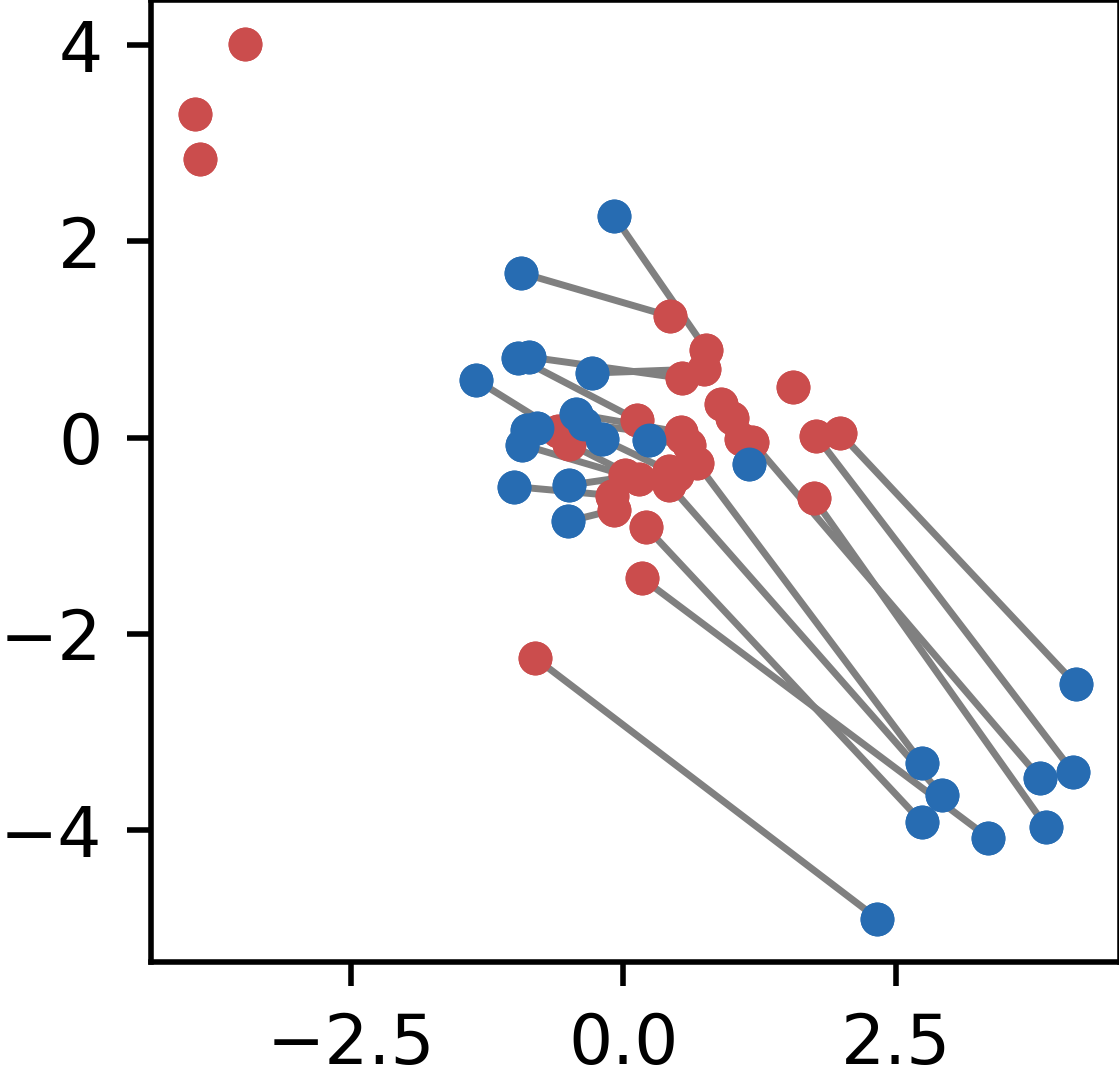}\vspace{0.1cm}
  \caption{$\KR_{2,6}$-plan}
\end{subfigure}
\caption{Optimal plans for balanced and unbalanced optimal transport between two point clouds (32 red and 26 blue points).  The presence of gray lines connecting two points indicates that mass is transported between them. \textbf{(a)}: Optimal transport plan  according to $2$-Wasserstein distance between red and blue points, each weighted with mass $1/32$ and $1/26$, respectively. \textbf{(b\,--\,d)}: Unbalanced optimal transport plan according to $(2,C)$-Kantorovich--Rubinstein plan (\Cref{def:UOT}) between red and blue points, each weighted with unit mass, with mass penalization parameter chosen as $C = 2$ (left), $C= 4$ (middle), and $C = 6$ (right).}
    \label{fig:UOTmatching}
\end{figure}

\begin{definition} \label{def:UOT}
For finite measures $\mu, \nu \in \MC(\XC)$ and parameters $p \geq 1, \ C > 0$ the $(p,C)$-\emph{Kantorovich--Rubinstein-distance} (KRD) between $\mu$ and $\nu$ is defined by 
\begin{equation} \label{eq:def_UOT}
    \KR_{p, C}(\mu, \nu) \coloneqq \left(\inf_{\pi\in \Pi_{\leq}(\mu, \nu)} \int_{\XC\times \XC} d^p(x_1,x_2) \pi(d x_1, \dif x_2) + C^p \left(\frac{m_\mu + m_\nu}{2} - m_\pi \right)\right)^{1/p}, 
\end{equation}
where the subscript ``$\leq$'' in $\Pi_{\leq}$ indicates that the infimum is taken over all \emph{sub-couplings} $\pi$ of $\mu$ and $\nu$, that is, the set
$$
    \Pi_{\leq}(\mu, \nu) \coloneqq \left\{\pi \in  \MC(\XC^2) \;:\; \pi(A \times \XC)\leq \mu(A), \pi(\XC\times B)\leq \nu(B) \text{ for all } A, B \in \BC(\XC)\right\}.
$$ 
\end{definition}
The optimization problem in \eqref{eq:def_UOT} is strongly reminiscent of that in the definition of the Wasserstein distance \eqref{eq:WassersteinDistance} with the key difference that the set of feasible solutions $\Pi_{\leq}(\mu, \nu)$ now additionally contains measures with less mass than $\mu$ and $\nu$, e.g., the zero measure $\pi \equiv 0$. A solution to the optimization problem in \eqref{eq:def_UOT} always exists (\Cref{prop:Prop_d_trunc}) and is called an \emph{unbalanced OT plan}. In addition, the KRD  involves a mass penalization term -- large values of $C$ encourage the unbalanced OT plan to transport as much mass as possible. The parameter $C>0$ also admits an explicit interpretation and quantifies the maximal distance between points  along which mass can be transported, see \Cref{fig:UOTmatching}(b-d) and \Cref{prop:Prop_d_trunc}, which is a consequence of the linearity of the mass penalization term  in $\pi$. This property becomes particularly relevant in the context of statistical data analysis, highlighting that the $(p,C)$-KRD is robust to spatial outliers, i.e., points located far in the tails of the measures under consideration. 
Moreover, for measures of equal mass, the corresponding $(p,C)$-KRD distance coincides for $C$ large enough with the Wasserstein distance (\Cref{prop:interpolate_C}$(iii)$).

Several other approaches to optimal transportation between general measures have been proposed in the literature.
One of the most prominent contributions is the \emph{partial OT formulation} introduced in  \cite{caffarelli2010free} and further studied in \cite{figalli2010optimal}. In this approach, the amount of mass $m \leq \min\{\mu(\XC), \nu(\XC)\}$ to be transported is fixed, and the Wasserstein$-2$ functional is minimized over sub-couplings between $\mu$ and $\nu$ with total mass $m$. 
The approach considered in \cite{piccoli2014generalized, piccoli2016properties} is similar to ours: their \emph{generalized Wasserstein distance} $W_p^{a, b}$ is a linear combination of a ``Wasserstein'' term and a term which penalizes the deviation of the plan from the marginals in $L_1$. %
\cite{figalli2010new} constructed solutions to evolution equations subject to Dirichlet boundary conditions, which differs from the Neumann boundary conditions typically used in OT problems.
Another line of extensions of balanced OT hinges on the \emph{dynamic formulation} of OT, for which we refer to \cite{chizat2018interpolating, chizat2018unbalanced} and the literature review therein. 
\cite{liero2018optimal} introduced unbalanced regularized OT, %
in which the difference of the marginals of the transport plan from the marginal measures is penalized by a convex Csisz\'ar $f$-divergence \citep{csiszar1967information}. The special case of the \emph{Hellinger--Kantorovich distance}, which arises from logarithmic penalization, is studied in~detail. 

Altogether, all these extensions from classical OT come with individual benefits and potential drawbacks, arising from the respective (mathematical or applied) motivations. In this paper, we focus on statistical aspects of the $(p,C)$-Kantorovich--Rubinstein distance \eqref{eq:def_UOT} due to its intuitive interpretation and documented potential for data analysis (see, e.g., \citealp{naas2024multimatch}). 
We stress, however, that our approach is not limited to this case, and we also derive results for alternative notions of unbalanced OT in \Cref{sec:other_UOTs}.
\subsection{Spatio-temporal point processes}
\label{ssec:STpp_intro}
Despite the popularity of various notions of unbalanced OT in practice, there is still not much known about their statistical behavior.
A key reason is potentially the lack of a widely accepted sampling model from general measures. 
To the best of our knowledge, only two works address this gap. \cite{sejourne2019sinkhorn} establish convergence rates for plug-in estimators of the unbalanced entropy-regularized OT cost under a model where empirical measures from i.i.d.\ samples of normalized distributions $\tilde \mu = \mu /\mu(\XC)$ and $\tilde \nu = \nu /\nu(\XC)$ are rescaled by the true masses. This estimator assumes knowledge of total masses and is primarily motivated by computational considerations. \cite{heinemann2022unbalanced} analyze $(p,C)$-KRD quantities under three statistical models -- including the i.i.d. model with unknown total mass, a Poisson regression and a Bernoulli model relevant to microscopy and imaging.  Their work is limited to the setting of finitely supported measures, but serves as a major motivation for the present approach.

In this paper, we consider data-generating approaches that go well beyond the usual independent and identically distributed (i.i.d.) assumption and are well-posed for generic non-discrete domains. For a given time frame $(0,t]$, we assume that we observe a set of a random number of data points $X_1, \dots, X_{N_t}$ with associated timestamps $T_1, \dots, T_{N_t}$ in $(0,t]$. In order to treat both continuous and discrete time in a unified approach, we write $\TC$ for the full time axis, which is either $\RR$ or $\ZZ$. We use generalized interval notation, setting, e.g., $(a,b] = \{ t' \in \TC :  a < t' \leq b \}$ and $[a,b] = \{ t' \in \TC :  a \leq  t' \leq b \}$. 
Our data is then modeled as realizations of a \emph{spatio-temporal (ST) point process} on $\XC \times \tgz$, where $\tgz = (0,\infty) = \{t' \in \TC :  t' > 0\}$ (see \Cref{fig:ABC}(a)-(b) for illustration). Such a point process may be described as a sum of Dirac measures,
\begin{align}\label{eq:STPP}
    \Xi = \sum_{i = 1}^{\infty} \delta_{(X_i,T_i)},
\end{align}
where each pair $(X_i, T_i)\in \XC \times \tgz$ represents the spatial location and temporal occurrence of a (random) event, and we may require $T_i\leq T_j$ for all $i\leq j\in \NN$. The number of observed points up to time $t$ is then given by $N_t \coloneqq \sup\{n \in \NN_0 :  T_n \leq t\}$ with $T_0\coloneq 0$.

This approach interprets data points as events characterized by both spatial locations and temporal occurrences, offering a more flexible and realistic description in various contexts, as it allows for general domains and non-trivial dependencies between data points as we do not assume independence of $X_{n+1}$ nor $T_{n+1}-T_{n}$ from the pairs $(X_i,T_i)$ for $1 \leq i \leq n$.
For instance, observed data points might represent occurrences of ecological events \citep{ogata1998earthquake,peng2005space}, or could describe disease incidents spread across geographical regions and recorded over time \citep{meyer2012space} or brain activity \citep{tagliazucchi2013breakdown}. 
We further emphasize that our setting encompasses the standard scenario of i.i.d.\ sampling on $\XC$ as a special case, where data points are (artificially) interpreted to appear sequentially at time steps in $\NN$ (so that $N_t = t$ for all $t \in \tgz = \NN$) without dependencies (see \Cref{ex:iid_sampling} and \Cref{sec:Bimomial}).

To guarantee that increasing $t$ leads to a more accurate representation of the underlying spatial population measure
(analogous to the identical distribution assumption in the classical i.i.d.\ sampling model), we work under the following notion of stationarity.

\begin{figure}[!t] %
    \centering
    \begin{subfigure}[b]{0.62\textwidth}
        \centering
        \begin{subfigure}[b]{\textwidth}
            \centering
         \includegraphics[width=\textwidth]{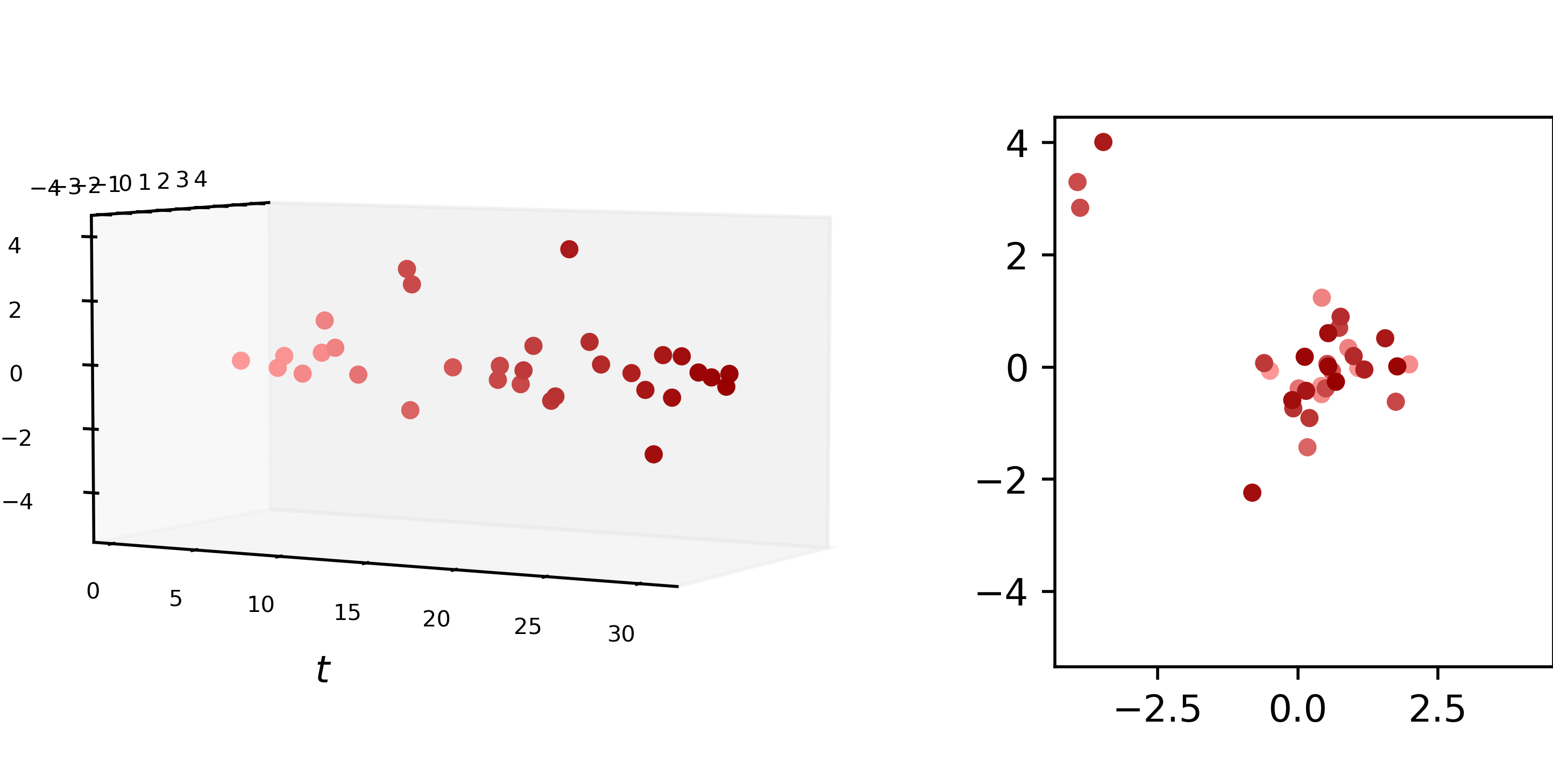}
            \caption{An ST point process and its spatial empirical measure}
            \label{fig:A}
        \end{subfigure}

        \begin{subfigure}[b]{\textwidth}
            \centering
            \includegraphics[width=\textwidth]{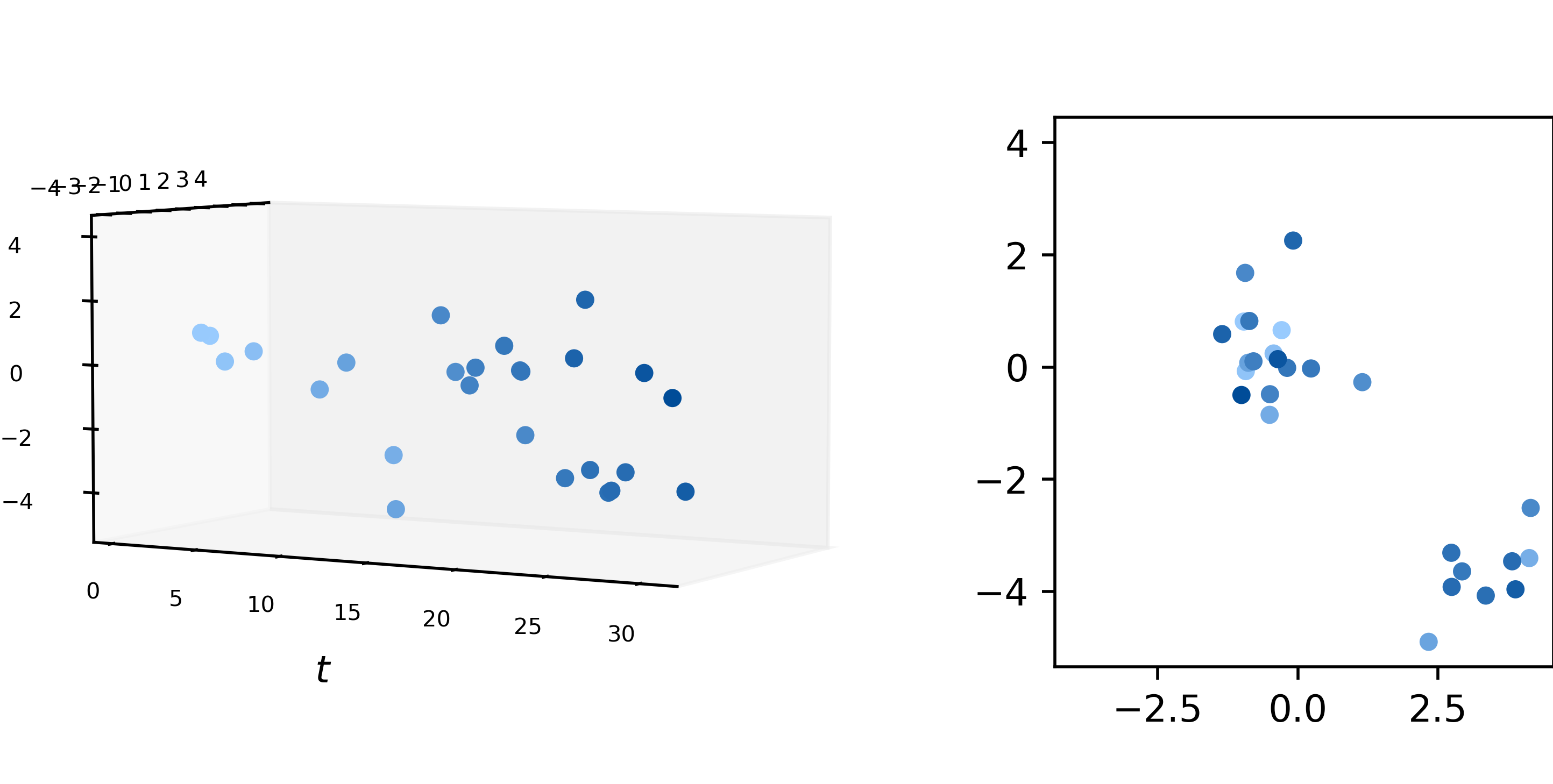}
             \caption{Another ST point process and its spatial empirical measure}
            \label{fig:B}
        \end{subfigure}
    \end{subfigure}
    \hfill
    \begin{subfigure}[b]{0.35\textwidth}
        \centering
        \raisebox{0.23\textheight}{
            \begin{minipage}{\textwidth}
                \centering
                \includegraphics[width=0.80\textwidth]{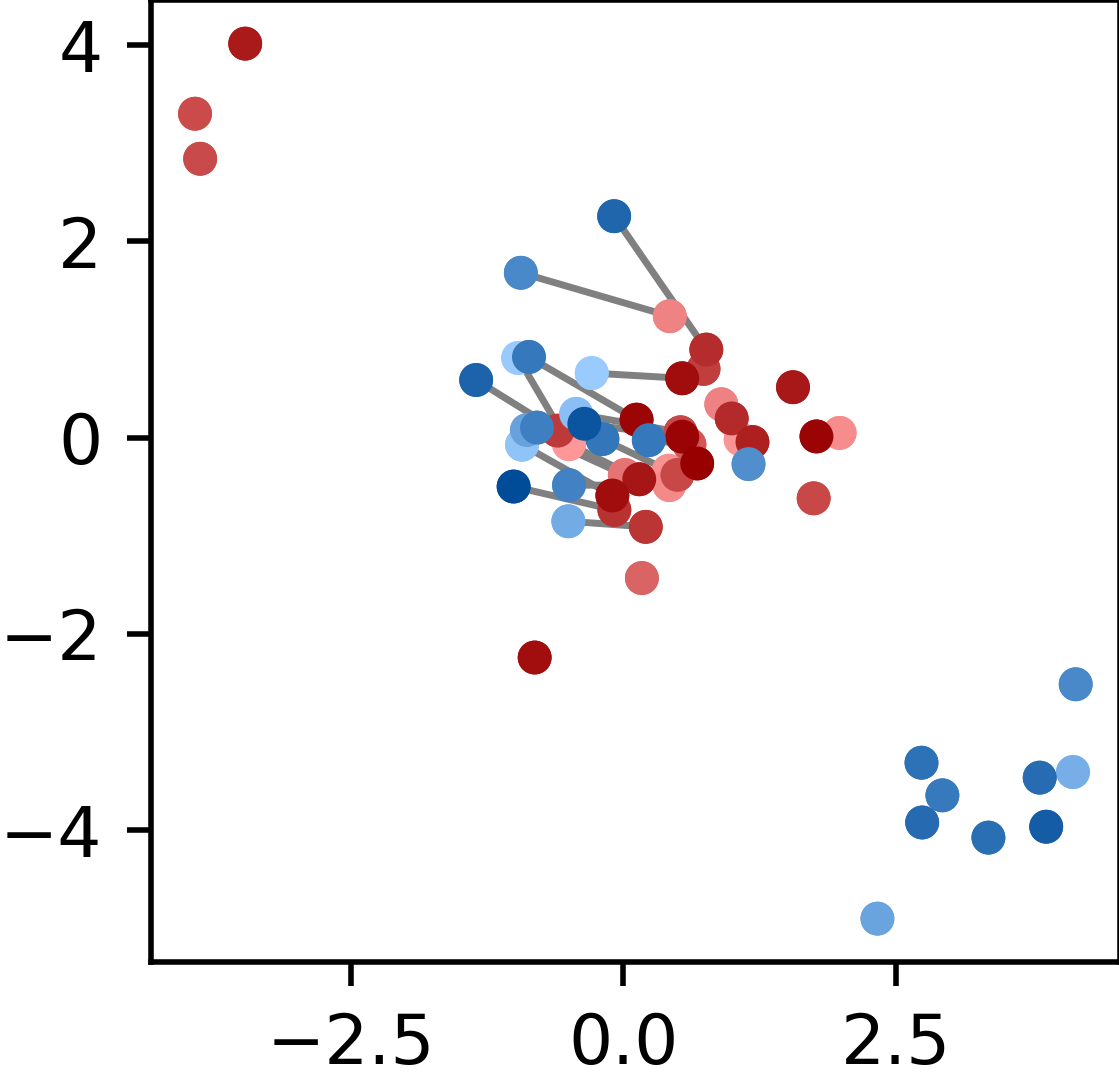}
                \caption{Unbalanced optimal transport\\between the spatial empirical measures}
                \label{fig:C}
            \end{minipage}
        }
    \end{subfigure}
    \caption{
    Estimation of KR distances, exemplified for two ST Poisson processes.
    \textbf{(a)}: Left: a realization of a Poisson point process $\Xi$ on $\RR^2 \times (0,\infty)$ with spatial intensity measure $2\NC((-0.5, 0)^\intercal, \Sigma)+\NC((-3.5, 3.5)^\intercal, \Sigma), \ \Sigma = \mathrm{diag} (0.5, 0.5)$. Right: $\Xi(\cdot \times (0, 50])$ for the same realization. The color shades indicate (from light to dark red/blue) the observation times of points for better visuals.  \textbf{(b)}: A realization of a Poisson point process $\Eta$ on $\RR^2 \times (0,\infty)$ with spatial intensity measure $3\NC((0.5, 0)^\intercal, \Sigma)+\NC((3.5, -3.5)^\intercal, \Sigma)$ using the same $\Sigma$. \textbf{(c)}: Kantorovich--Rubinstein transport between $\Xi(\cdot \times (0, 50])$ and $\Eta(\cdot \times (0, 50])$
    with $(p,C)=(2, 2)$. The presence of gray lines connecting two points indicates that mass is being transported between those points. %
    } 
    \label{fig:ABC}
\end{figure}

\begin{definition} \label{def:ST_stat_L2_process_def}
  An ST point process $\Xi$ on\footnote{\label{fn:full_time_axis}Note that this definition requires, in fact, that the process is defined on the whole of $\XC \times \TC$. We will therefore think of $\Xi$ as the restriction to $\XC \times \tgz$ of a process defined on $\XC \times \TC$ (which we also denote by $\Xi$).} $\XC \times \tgz$ is called  \emph{second-order time-stationary}%
, or \emph{weakly time-stationary},  if the first two moment measures of $\Xi$ exist and are invariant under time-shifts, i.e., $\EE[\Xi \chi_t^{-1}(A)] = \EE[\Xi(A)] < \infty$ and $\EE[\Xi \chi_t^{-1}(A) \, \Xi \chi_t^{-1}(B)] = \EE[\Xi(A) \, \Xi(B)] < \infty$ for all bounded $A,B \in \BC(\XC \times \TC)$ where $\chi_t \colon \XC \times \TC \to \XC \times \TC$, $(x,s) \mapsto (x,s+t)$.
\end{definition}
Weak time-stationarity implies the existence of a (finite) measure $\mu\in \MCsX$ such that $\EE[\Xi(A\times (0,t])] = t \mu(A)$ for all $A\in \BC(\XC), \, t \in \tgz$ (\Cref{prop:time-reduction}), which thereby is to be understood as the underlying \emph{(spatial) population measure}, also referred to as the \emph{(spatial) intensity measure} of $\Xi$. It~is therefore natural to introduce the \emph{(spatial) empirical measure}
\begin{align}\label{eq:emp_measure}
    \hat \mu_t \coloneqq \frac{1}{t}\Xi(\cdot \times (0,t]) = \frac{1}{t}\sum_{i = 1}^{N_t} \delta_{X_i} \in \MCsX, 
\end{align}
which serves as an estimator for $\mu$. Estimation of the unknown (spatial) intensity measure $\mu$ is one of the fundamental tasks in point process statistics \citep{diggle2013statistical, baddeley2016spatial}. Models range from very explicit log-linear parametric ones in terms of known spatial covariates $\XC \to \RR$ \citep{baddeley2000practical} %
to fully nonparametric ones, e.g., based on kernel smoothing \citep{ellis1991density, vanlieshout2012estimation} or reproducing kernel Hilbert spaces \citep{flaxman2017poisson}. Judicious intensity estimation is crucial for distinguishing between spatial inhomogeneity and correlation of points. It is also the basis of more sophisticated approaches for estimating a (hidden) component of the intensity, such as the background seismicity in the very popular ST epidemic-type aftershock sequence model (ETAS) for describing earthquake occurrences \citep{ogata1998earthquake, ogata2006space}

In our context of weakly time-stationary ST point processes, a natural framework for investigating the asymptotic behavior of spatial intensity estimators (per unit time, say) is given by observing the point process on $\XC \times (0,t]$ and letting $t \to \infty$. Similar asymptotics have been considered for various concrete intensity estimators \citep{ellis1991density, schoenberg2005consistent, fuentes2016consistent, macdonald2025bandwidth}. %

\subsection{Main results}

In the present paper, we provide quantitative error bounds for the expected KR distance between the empirical measure $\hat{\mu}_t$ defined in \eqref{eq:emp_measure} and the underlying spatial intensity measure $\mu$. This also allows us to derive  statistical error bounds for the one- and two-sample plug-in estimators $\KR_{p,C}(\hat \mu_t, \nu)$ and $\KR_{p,C}(\hat \mu_t, \hat \nu_s)$ of $\KR_{p,C}(\mu,\nu)$ based on the empirical measures $\hat \mu_t $ and $ \hat \nu_s$ stemming from weakly stationary ST point processes with respective population measures $\mu$ and $\nu$ (see \Cref{fig:ABC}(c)). To derive such bounds, we introduce a condition quantifying the temporal growth of the sum of variances of an ST point process $\Xi$ on $\XC\times \TC_{>0}$ over partitions of the domain in terms of a parameter $\beta \in [0,1)$.

\begin{description}
	\labitem{(PVar$_\beta$)}{ass:Xi_mixing_Var_beta}
  There exist constants $\kappa\geq 0$ and $t_0\in \TC_{>0}$ such that for any partition $(A_i)_{i=1,\ldots,N}$ of $\XC$ into measurable sets and any $t \in \tgz$, $t \geq t_0$, 
\begin{align*}
  \sum_{i = 1}^N \Var(\Xi(A_i \times (0,t])) \leq  \kappa t^{1+\beta}.
 \end{align*}
\end{description}
To exemplify our key insights on the statistical error of the empirical $(p,C)$-KRD, let us focus  on the setting where $\XC= [0,1]^d$ is the $d$-dimensional Euclidean unit cube, and $\Xi$ is a second-order time-stationary ST point process with population measure $\mu \in \MCsX$ and empirical measure $\hat \mu_t$ as defined in \eqref{eq:emp_measure}. If $\Xi$ fulfills Assumption \ref{ass:Xi_mixing_Var_beta} for $\beta\in [0,1)$ and $t_0\in \TC_{>0}$, then we show in \Cref{th:est_mu_inUOT_beta} for $t\in \TC_{>0}$, $t\geq t_0$ and all $C>0$ that 
  \begin{align}  \label{eq:est_mu_inUOT_beta_intro}
	&\EE[\KR_{p,C}(\mun, \mu)] \lesssim_{d,p, \kappa, m_{\mu}} C \wedge  \left(
   C  t^{(-1+\beta)/2p} +  
  \begin{cases}
        C^{1-d/2p} t^{(-1+\beta)/2p} & \text{if } d < 2p, \\
        t^{(-1+\beta)/2p} \log(2+C^{p}t^{(1-\beta)/2})^{1/p} & \text{if }d = 2p, \\
        t^{(-1+\beta)/d} & \text{if } d > 2p.
    \end{cases} \right)
  \end{align}
Notably, these results also apply when $\mu$ is concentrated on a submanifold of dimension $d$ embedded in a high-dimensional Euclidean space.  Moreover, given another second-order time-stationary point process which also fulfills \ref{ass:Xi_mixing_Var_beta} and has spatial empirical and population measures $\hat \nu_t$ and $\nu$, respectively, we obtain from the triangle inequality,
\begin{align*}
 \EE[\left|\KR_{p,C}(\mun, \nun) - \KR_{p,C}(\mu, \nu)\right|] \leq \EE[\KR_{p,C}(\mun, \mu)] + \EE[\KR_{p,C}(\nun, \nu)],
\end{align*}
a similar upper bound for the empirical $(p,C)$-KRD.

Let us shortly discuss the key insights from our error bound in \eqref{eq:est_mu_inUOT_beta_intro}. The first term, $C$, arises from the fact that the $(p,C)$-KRD is suitably bounded in terms of the mass penalization parameter and the total mass of $\mu$ (\Cref{lem:trivialUpperBound}). To achieve a smaller statistical error, the time parameter $t$ needs to be sufficiently large. Specifically, the second term arises from the mean absolute deviation between $m_{\mun}$ and $m_{\mu}$, and achieves for $\beta =0$ a parametric convergence rate in $t$. For $d\geq 2p$ or $\beta>0$ the convergence rate in $t$ is strictly slower than parametric. 
We also note that for $d\leq2p$, choosing the mass penalization parameter  $C$ smaller leads to an improved convergence rate in $t$, whereas for $d> 2p$ the mass penalization parameter does not affect the dimension-dependent rate in $t$ in \eqref{eq:est_mu_inUOT_beta_intro}.

As an extension of our analysis for the empirical KRD, we derive in \Cref{sec:other_UOTs} similar statistical convergence results for empirical plug-in estimators of other unbalanced OT quantities from \Cref{sec:intro_uots}. Specifically, we consider divergences $\U: \MCsX \times \MCsX \to \RR_+$ satisfying certain metric properties (see \Cref{def:niceUOTformalism} for the details), which encompass, among others, the Gaussian--Hellinger and the Hellinger--Kantorovich distance introduced by \cite{liero2018optimal}. For $\XC = [0,1]^d$ and $\Xi$ as above, we establish for $t\in \TC_{>0}, \, t \geq t_0$ that 
\begin{equation*}
	\EE[\U(\mun, \mu)]\lesssim t^{(-1+\beta)/2p} + t^{(-1+\beta)/d}(1 + \mathds{1}(d = 2p) \log(2+t)).
\end{equation*}

One can see that the rate in $t$ (and the dependency in $\beta$) captured in \eqref{eq:est_mu_inUOT_beta_intro} is subject to the data-generating mechanism and not our OT formulation. 
In fact, under the classical i.i.d.\ sampling, Assumption \ref{ass:Xi_mixing_Var_beta} can be confirmed for $\beta = 0$, and our insights assert that the same rates as for the Wasserstein-$p$ distance  (see \eqref{eq:WassersteinRates} below) manifest by choosing $C\geq \diam(\supp(\mu+\nu))$ (see \Cref{th:est_mu_inUOT_beta} and \Cref{prop:interpolate_C}$(iii)$). 
Our findings emphasize, however, that similar convergence rates also manifest for much more complicated statistical models, as long as \ref{ass:Xi_mixing_Var_beta} is met for $\beta$ close to zero. 

To put Assumption \ref{ass:Xi_mixing_Var_beta} in context, we confirm in \Cref{sec:ST_proc} that it is satisfied by a wide range of ST point processes. Perhaps the most prominent case with $\beta = 0$ (and $\kappa = m_{\mu}$) is the ST (time-stationary, spatially inhomogeneous) \emph{Poisson point process}, see \cite{daley2003introduction,daley2007introduction} or \Cref{sec:pp_examples}. We also discuss various ST processes with non-trivial spatio-temporal dependencies, including \emph{Poisson cluster processes}, \emph{Hawkes processes}, and \emph{Log-Gaussian Cox processes}. 
Meanwhile, in case $\beta >0$, the convergence rate in \eqref{eq:est_mu_inUOT_beta_intro} is strictly slower, and can be arbitrarily slow as $\beta$ tends to $1$. Indeed, for the time-stationary point process $\Xi = \sum_{i = 1}^\infty \delta_{(X, i)}$ supported at a single element $X\sim \textup{Unif}[0,1]^d$ and $\TC = \ZZ$  it follows that $\Var(\Xi([0,1/2]^d \times (0,t])) \asymp t^2$ and $\KR_{p,C}(\hat \mu_t,\mu) \asymp C$ independently of $t$, which shows that for $\beta$ close to $1$ slow rates in \eqref{eq:est_mu_inUOT_beta_intro} are inevitable.
In fact, complementary to the upper bound in \eqref{eq:est_mu_inUOT_beta_intro}, we establish in \Cref{sec:lowerbounds}  for $\beta =0$ matching pointwise lower bonds in $C$ and $t$, jointly with minimax rate optimality, and for $\beta \in (0,1)$ provide an example with a matching rate in $t$ under $d< 2p$.

\subsection{Applications and discussion}
The central objective of this work is to establish rigorous statistical guarantees for the spatial empirical measure as an estimator of its population counterpart and, by extension, for the empirical KRD as an estimator of the population KRD. The KRD constitutes a conceptually appealing functional for comparing finite measures, as it naturally incorporates the geometry of the underlying spatial domain. By computing the KRD for varying choices of $C$, the KRD enables a  comparison between two measures at multiple scales of resolution. This property is a direct consequence of the proposed mass penalization framework, which ensures that no mass is matched across distances exceeding $C>0$. A further implication is that the KRD exhibits robustness to spatial outliers.

Our statistical results confirm that the empirical plug-in estimators for the KRD serve as good estimators if Assumption~\ref{ass:Xi_mixing_Var_beta} is met for $\beta\in [0, 1)$ close to zero and if the underlying spatial population measure is of small (intrinsic) dimension. The latter assumption is naturally satisfied in practical imaging contexts, particularly in two and three dimensions \citep{komiske2019metric,tameling2021Colocalization,manole2022background}.  For the former assumption, we develop in Subsection \ref{sec:pp_assumptions} a relation to \emph{time-reduced factorial covariance measures} of point processes. 
Factorial covariance measures are a well-studied and popular tool in spatial statistics, and they can be calculated for many commonly used point process models.
For example, we may infer that Assumption~\ref{ass:Xi_mixing_Var_beta} is met for $\beta=0$ for a Poisson cluster process (\Cref{sec:cluster_ppp}) if each point of a parent Poisson process induces at most an $L^2$ random number of subsequent points over infinite time. 
Meanwhile,  \ref{ass:Xi_mixing_Var_beta} is met for $\beta \in (0,1)$ if the second moment of the number of subsequent points within time $(0,t]$ grows sublinearly in $t$ and, moreover, the growth of the sum of partitioned variances $\sum_{i = 1}^N \Var(\Xi(A_i \times (0,t])$ can at most be of order $t^2$ due to the assumption of weak time-stationarity of the $L^2$ point process. We formalize these insights in Subsection~\ref{sec:pp_assumptions}. 

As a direct consequence of our analysis of the KRD based on the (spatial) empirical measure $\hat \mu_t$, we note that similar statistical error bounds can be confirmed for other measure estimators $\hat \mu_t'$ as long as they are close enough in expected KRD to $\hat \mu_t$. The latter is straightforward to see, e.g., for a kernel intensity estimator. In particular,  in low-dimensional settings (or low degrees of smoothness), previously derived minimax lower bounds for estimation of measures in Wasserstein distance align with the rates in \eqref{eq:WassersteinRates} when the density is not bounded away from zero \citep{niles2022minimax,divol2022measure}. We note, however, that for moderate dimensions and high degrees of smoothness, faster rates can be achieved with respect to the Wasserstein distance, which indicates a similar phenomenon for the KRD.  We do not pursue this here since the corresponding techniques likely require completely different tools.

On a more general note, we would like to stress that we see our work as a first step towards the development of tools for statistically sound data analysis based on the KRD and other notions of unbalanced transport. Our convergence results highlight potential use of the KRD in statistical goodness-of-fit testing on the underlying (spatial) intensity measure. Herein, the empirical KRD would serve as a test statistic to evaluate the adequacy of model assumptions or to detect deviations from a hypothesized distribution. %
 This would also enable the formulation of refined confidence intervals for the KRD. A full development of this direction is left for future research, as it is likely to require distributional limits of $\KR_{p,C}(\mun, \mu)$ and, in turn, stronger assumptions than those imposed by \ref{ass:Xi_mixing_Var_beta}.  

\subsection{Related work on statistical optimal transport}\label{sec:ot_stat_overview}

Most contributions in the statistical OT literature operate under the assumption that the respective data points are independent and identically distributed, say, $X_1, \dots, X_n \sim \mu$ and $Y_1, \dots, Y_n \sim \nu$, where $\mu$ and $\nu$ are (unknown) probability measures with empirical measures 
\begin{align*}
  \hat \mu_n = \frac{1}{n}\sum_{i=1}^n \delta_{X_i} \quad \text{ and } \quad \hat \nu_m = \frac{1}{n}\sum_{j=1}^n \delta_{Y_j}.
\end{align*}
Within this  statistical framework, a rich body of literature has been developed to analyze convergence rates of $W_p(\hat \mu_n, \hat \nu_n)$ towards $W_p(\mu, \nu)$. Starting with the seminal work of \cite{dudley1969}, results for $p \geq 1$ have been established in various forms: we mention high-probability bounds \citep{ajtai1984optimal, talagrand1994matching}, almost sure limit theorems \citep{yukich86, barthe2013combinatorial, ambrosio2019pde}, and expectation-based results \citep{dereich2013constructive, boissard2014, fournier2015rate, singh2018minimax, bobkov2019one, sommerfeld19FastProb, weed2019sharp, chizat2020, Manole21,hundrieser2022empirical}; for a comprehensive treatment we refer to \cite{chewi2024statistical}.

One of the key insights of these works, is that the convergence rates fundamentally depend on the intrinsic dimension of the support of the measures and on certain moment conditions. For instance, in case of a probability measure $\mu$ on $[0,1]^d$, it holds %
\begin{align}\label{eq:WassersteinRates}
\EE[W_p(\hat \mu_n, \mu)]
\lesssim \begin{cases}
    n^{-1/2p}& \text{ if } d < 2p,\\
    n^{-1/2p}\log(1+n)^{1/p} & \text{ if } d = 2p,\\
    n^{-1/d} & \text{ if } d > 2p,
    \end{cases}
\end{align}
and by  triangle inequality  the same bound also holds for $\EE[\left|W_p(\hat \mu_n, \hat \nu_n)  - W_p(\mu, \nu)\right|]$. These convergence rates are known to be minimax rate optimal up to logarithmic terms \citep{singh2018minimax, niles2019estimation}.
Our bound on the empirical KRD \eqref{eq:est_mu_inUOT_beta_intro} is strongly evocative of this result and, in fact, implies \eqref{eq:WassersteinRates} since i.i.d. sampling fulfills \ref{ass:Xi_mixing_Var_beta} for $\beta = 0$ (\Cref{th:est_mu_inUOT_beta} and \Cref{prop:interpolate_C}$(iii)$). 

Under additional structural assumptions, such as connected support with positive density, the convergence rate for empirical estimators improves to $n^{-1/(2\vee d)}$ with an additional logarithmic term for $d = 2$ \citep{ajtai1984optimal,talagrand1992matching,bobkov2019one,ledoux2019OptimalMatchingI}. Another set of articles has revealed that faster convergence rates also manifest for the empirical Wasserstein distance under additional separation constraints on the population measures \citep{chizat2020,Manole21,hundrieser2022empirical, staudt2025convergence}. While it is known that similar phenomena emerge beyond classical i.i.d.\ sampling under certain mixing conditions \citep{deb2024tradeoff}, see also \Cref{rem:mu_in_krd}$(v)$, a sharp characterization of necessary conditions is out of scope of the current paper.

We finally mention that the proof of \eqref{eq:est_mu_inUOT_beta_intro} relies on a dyadic partition argument, a well-known tool applied in the OT context \citep{dereich2013constructive, fournier2015rate, weed2019sharp,sommerfeld19FastProb,heinemann2022unbalanced}. Our two core technical contributions in this context lie in $(i)$ adapting this existing technique to the setting of ST point processes which fulfill the weak but sufficient Assumption \ref{ass:Xi_mixing_Var_beta}, and $(ii)$ achieving a sharp dependence on the mass penalization parameter $C$.

\subsection{Outline} The manuscript is structured as follows. 
In \Cref{sec:UOT} we derive fundamental properties of the $(p,C)$-KRD defined in \eqref{eq:def_OT}. %
Section \ref{sec:main} contains our main results on the statistical performance of the empirical KRD and other notions of unbalanced OT in the framework of ST processes. 
Section \ref{sec:ST_proc} contains some basic theory on ST point processes, discusses our variance and covariance growth conditions, and discusses point process examples satisfying them. 
Section \ref{sec:lowerbounds} derives pointwise and  minimax lower bounds that match our upper bounds. 
Almost all proofs, except for our main result in \Cref{sec:main}, are provided in separate appendices to maintain a clean presentation. 
Appendices \ref{sec:proofs2}--\ref{sec:proofs_LowerBounds} contain the proofs for Sections \ref{sec:UOT}--\ref{sec:lowerbounds}, respectively, and, in \Cref{sec:aux_results}, additional technical results can be found.

\subsection{Notation}

Throughout this work, we use the following notational conventions. The sets of integers and non-negative integers are denoted by $\ZZ$ and $\NN_0$, respectively, while $\RR$ and $\RR_+$ denote the real numbers and non-negative real numbers. For $\TC\in\{\RR, \ZZ\}$ we write $\tgz = \TC \cap (0, \infty)$ and $\tgeqz = \TC \cap [0, \infty)$.
The reference measure on $\mathcal{T}$ is written as $\ell$, which equals the Lebesgue measure $\Leb_{\RR}$ when $\mathcal{T} = \RR$, and the counting measure $\sum_{z \in \ZZ} \delta_z$ when $\mathcal{T} = \ZZ$; analog conventions apply for the reference measures on $\tgz$ and $\tgeqz$.  The Borel $\sigma$-field on a topological space $\mathcal{X}$ is written as $\mathcal{B}(\mathcal{X})$, and we denote by $\mathcal{M}(\mathcal{X})$ the collection of non-negative, finite Borel measures on $\mathcal{X}$. Further, we write $\PC(\XC)$ to denote the Borel probability measures on $\XC$. The Borel space of locally finite point configurations on $\XC \times \TC$ is denoted by $(\mathfrak{N}, \, \NC) = (\mathfrak{N}(\XC \times \TC), \, \NC(\XC \times \TC))$. If $\Xi$ is a Poisson point process with intensity $\Lambda$, we denote it by $\Xi \sim \pop(\Lambda)$.
The minimum and maximum of two numbers $a$ and $b$ are denoted by $a \wedge b$ and $a \vee b$, respectively. 
For non-negative functions $f, g\colon \TC_{>0}\to \RR_{+}$, we write $f(t) \lesssim_\alpha g(t)$ as $t \to \infty$ to mean that there exists a constant $c > 0$, possibly depending on a parameter $\alpha$, such that $f(t) \leq c \, g(t)$ for $t$ large enough, and we write $f(t) \asymp_{\alpha} g(t)$ if both $f(t) \lesssim_{\alpha} g(t)$ and $g(t) \lesssim_{\alpha} f(t)$ hold. 
For $\mu, \nu \in \mathcal{M}(\mathcal{X})$, the inequality $\mu \leq \nu$ indicates that $\mu(A) \leq \nu(A)$ for every Borel set $A \subseteq \mathcal{X}$. %
The total mass of a measure $\mu$ is denoted by $m_\mu = \mu(\mathcal{X})$.  %

\section{Foundations of Kantorovich--Rubinstein distance} \label{sec:UOT}
This section details various structural properties  of our formulation of the unbalanced OT problem defined in \eqref{eq:def_UOT}. We first provide an interpretation of the mass penalization parameter $C$ in the definition of the KRD in terms of unbalanced OT plans. We then develop a link between the KRD and a lifted OT problem, which enables efficient computation. As a byproduct, we also derive how the KRD behaves under changes of the model parameters and inputs, and characterize metric properties.  
Since most statements are already known for finite metric spaces, we defer most proofs to the appendix except for that of Proposition $2.1(ii)-(iii)$, which is provided in Appendix \ref{sec:proofs2}.

\subsection{Basic structural facts on Kantorovich--Rubinstein distance}

Our main insight in this subsection provides a sharp characterization between unbalanced OT plans and the mass penalization parameter. 

\begin{proposition}[Structure of KRD]\label{prop:Prop_d_trunc}
Let $(\XC, d)$ be a c.s.m.s.\ and let $p \geq 1$ and $C>0$. Then for measures $\mu, \nu \in \MCsX$ the following assertions hold. 
\begin{enumerate}[label=$(\roman*)$]
	\item There exists a feasible transport plan $\pi\in \Pi_{\leq}(\mu, \nu)$ which solves \eqref{eq:def_UOT}. The collection of all such optimizers is called the collection of \emph{unbalanced OT plans}.
	\item The $(p,C)$-KRD between $\mu$ and $\nu$ is characterized by 
	\begin{equation*}
    \KR_{p, C}(\mu, \nu) = \left(\inf_{\pi\in \Pi_{\leq}(\mu, \nu)} \int_{\XC\times \XC} d^p(x_1,x_2) \wedge C^p \pi(d x_1, \dif x_2) + C^p \left(\frac{m_\mu + m_\nu}{2} - m_\pi \right)\right)^{1/p}.
  	\end{equation*}
	\item Every unbalanced OT plan $\pi$ solving \eqref{eq:def_UOT} fulfills $\int \mathds{1}(d(x_1,x_2)>C)\pi(d x_1, \dif x_2) = 0$.
\end{enumerate}
\end{proposition}
\Cref{prop:Prop_d_trunc}$(iii)$ details a precise interpretation of the mass penalization parameter $C>0$: it controls the maximal distance along which mass is transported. This property of the $(p,C)$-KRD is conceptually appealing and useful for data analysis, as it allows the practitioner to explicitly model distances between points when no transportation of mass is allowed. Specifically, if $\mu$ and $\nu$ are finitely supported and each atom has unit mass, the solution to the unbalanced OT problem in \eqref{eq:def_UOT} amounts to a variable selection problem, namely to find the matching between pairs of (potentially strict) subcollections of points associated with the total cheapest transportation cost. 

\subsection[Linkage between Kantorovich--Rubinstein distance and optimal transport]{Linkage between Kantorovich--Rubinstein distance and optimal\\ transport}
Inspired by \cite{caffarelli2010free, guittet2002extended} and more recently \cite{heinemann2023kantorovich}, we relate in this section the $(p,C)$-KRD in \eqref{eq:def_UOT} to an equivalent balanced OT problem. This linkage is a key feature of our formulation, which helps establish certain topological and analytical properties of the $(p,C)$-KRD. Further, it enables direct porting of algorithms for balanced OT under general cost functions to computing the KRD. 

To formalize the lifted problem, consider an external point $\mathfrak{d} \notin \XC$ and let $\tilde{\XC} := \XC \,\dot\cup\,\{\mathfrak{d}\}$ be  the augmented space equipped with distance 
\begin{equation} \label{eq:dhat_def}
    \tilde{d}_{p,C}\left(x_1, x_2\right) := \begin{cases}
        d\left(x_1, x_2\right) \wedge C & \text{ if } (x_1, x_2) \in \XC\times\XC, \\
        C/2^{1/p} & \text{ if } \mathfrak{d} \in \{x_1, x_2\}, \ x_1 \neq x_2, \\
        0 & \text{ if } x_1 = x_2 =\mathfrak{d}. \end{cases}
\end{equation}
By \citet[Lemma A.1]{muller2020metrics} $\tilde{d}_{p,C}$ defines a metric on $\tilde \XC$.
Further, to extend the measures $\mu$ and $\nu$ to $\tilde \XC$, we define for a fixed $K \geq m_\mu \vee m_\nu$ the measures $\tilde\mu_K, \tilde\nu_K \in \MC(\tilde\XC)$, 
\begin{equation} \label{eq:tilda_measures_def}
    \tilde{\mu}_K \coloneqq \mu +(K - m_\mu)\delta_{\mathfrak{d}}, \quad \tilde{\nu}_K \coloneqq \nu + (K - m_\nu)\delta_{\mathfrak{d}}.
\end{equation}
This way, $\tilde \mu_K$ and $\tilde \nu_K$ are finite measures with identical total mass,  $m_{\tilde{\mu}_K}=m_{\tilde{\nu}_K}=K$. 
The following result links the unbalanced OT problem to a balanced OT problem with (balanced) augmented measures. 
\begin{proposition}[Equivalence of KRD and lifted optimal transport] \label{prop:linkOT}
Let $p \geq 1$, $C>0$, and consider measures  $\mu, \nu \in \MCsX$. %
Then, the following assertions hold.
\begin{enumerate}[label=$(\roman*)$]
    \item The $p$-th power of $(p,C)$-KRD between $\mu$ and $\nu$ coincides with the balanced OT cost between the augmented measures $\tilde \mu_K, \tilde \nu_K$ for $K\geq m_{\mu}\vee m_{\nu}$ and cost function $\tilde{d}^p_C$, i.e.,  \begin{equation} \label{eq:UOT=OT}
        \KR_{p, C}^p(\mu, \nu) = \OT_{\tilde{d}^p_{p,C}}(\tilde{\mu}_K, \tilde{\nu}_K).
    \end{equation}
    \item Fix $K= m_{\mu}\vee m_{\nu}$. Then, for every unbalanced OT plan $\pi$ between $\mu$ and $\nu$ with parameters $(p,C)$, there is an OT plan $\tilde \pi$ between $\tilde \mu_K$ and $\tilde \nu_K$ for cost function $\tilde d_{p,C}^p$ such that $$\pi(\cdot\cap \mathcal{D}_C) = \tilde \pi(\cdot\cap \mathcal{D}_C) \quad \text{ for } \quad \mathcal{D}_C \coloneqq \{(x,x')\in \XC^2 \colon d(x,x')< C\}.$$
    Conversely, for every OT plan $\tilde \pi$ between $\tilde \mu_K$ and $\tilde \nu_K$ for the cost $\tilde d_{p,C}^p$, there is an unbalanced OT plan $\pi$ between measures $\mu$ and $ \nu$ with parameters $(p,C)$ such that the above equality holds. 
\end{enumerate}
\end{proposition}

The above proposition highlights that essentially all available (exact and approximate) methods for computing balanced OT for general cost functions, e.g., the auction algorithm \citep{bertsekas1989}, transportation simplex, \citep{luenberger2008linear}, network flow algorithms \citep{Bonneel2011}, or interior methods such as the Sinkhorn algorithm \citep{cuturi13,cuturi19}, can be immediately employed to compute the KRD.  In particular, they also enable computing the associated UOT plan (for details, see \cite{heinemann2023kantorovich}). Another notable aspect of \Cref{prop:linkOT}$(i)$ is that since the left-hand side of \eqref{eq:UOT=OT} is independent of $K\geq m_\mu \vee m_\nu$, so is the right-hand side. %

\subsection{Effects of parameters on Kantorovich--Rubinstein distance}
Based on \Cref{prop:linkOT}, several properties of the $(p,C)$-KRD follow. %
\begin{proposition}[Role of parameters]\label{prop:kr_properties} For $\mu, \nu \in \MCsX$ the following properties hold. 
    \begin{enumerate}[label=$(\roman*)$]
        \item For $0 < C_1 \leq C_2$ it follows that $\KR_{p,C_1}(\mu, \nu) \leq \KR_{p,C_2}(\mu, \nu)$.
        \item  For $0 < p \leq q$ it follows that $\KR_{p,C}(\mu, \nu) \leq \KR_{q,C}(\mu, \nu)$.
        \item  For $a \geq 0$ it follows that $\KR_{p,C}(a\mu, a\nu) = a^{1/p} \KR_{p,C}(\mu, \nu)$ and moreover it holds that $\KR_{p,C}(\mu, a \mu) = \frac{C^p}{2}|m_\mu - am_\mu|$.
        \item It holds that $\frac{C^p}{2}|m_\mu - m_\nu| \leq \KR_{p,C}^p(\mu, \nu) \leq \frac{C^p}{2}(m_\mu + m_\nu)$.
        \item In particular, if $m_\mu \wedge m_\nu=0$, then $\KR_{p,C}^p(\mu, \nu) = \frac{C^p}{2}(m_\mu +m_\nu)$. 
    \end{enumerate}
\end{proposition}

Extending upon Assertions $(i)$ and $(iv)$ from \Cref{prop:kr_properties}, in the following proposition we analyze the two regimes where the mass penalization parameter $C$ is small or large relative to the support of the measures. 

\begin{proposition}[Effect of $C$] \label{prop:interpolate_C}
For $\mu, \nu \in \MCsX$ and $p \geq 1$, the following assertions hold. 
\begin{enumerate}[label=$(\roman*)$]
    \item If $0<C \leq \inf\limits_{x_1\in \supp(\mu)} \inf\limits_{x_2 \in  \supp(\nu)\backslash\{x_1\}}  d(x_1, x_2)$, then $\KR_{p,C}^p(\mu, \nu) = \frac{C^p}{2}\mathrm{TV}(\mu, \nu)$, 
    where $\mathrm{TV}(\mu, \nu) \coloneqq \sup\limits_{B\in \BC(\XC)}|\mu(B) - \nu(B)|$ 
   is the total variation distance of $\mu$ and~$\nu$.%
    \item 
    If $0<C\leq \inf\limits_{\substack{x_1 \in \supp(\mu), x_2 \in \supp(\nu)}}d(x_1,x_2)$, then $\KR_{p,C}^p(\mu, \nu) = \frac{C^p}{2} (m_\mu + m_\nu).$
    \item If $C \geq \sup\limits_{\substack{x_1\in \supp(\mu), x_2\in \supp(\nu)}}d(x_1, x_2)$ and $m_\mu = m_\nu$, then 
    $
        \KR_{p, C}^p(\mu, \nu)=W_p^p(\mu, \nu).
    $
\end{enumerate}
\end{proposition}

\subsection{Metric properties}
The relation between the $(p,C)$-KRD and the augmented OT \eqref{eq:UOT=OT} also enables us to derive certain topological properties of $(p,C)$-KRD. 
\begin{proposition}\label{prop:metricKRD}
Let $(\XC, d)$ be a c.s.m.s. For $p\geq 1$ and $C>0$ the $(p,C)$-KRD defines a metric on the space of finite measures $\MC(\XC)$.
\end{proposition}

Given that the $(p,C)$-KRD is a metric, the following result asserts that it metrizes weak convergence of measures on the space  $\MCsX$ of finite measures on $\XC$. 

\begin{proposition}[Metrizability and continuity] \label{prop:weak_conv_KRD}
Let $(\XC, d)$ be a c.s.m.s.\ with $\mu, \nu \in \MC(\XC)$ and sequences $(\mu_n)_{n \in \NN}, (\nu_n)_{n \in \NN}\subseteq \MCsX$. 
\begin{enumerate}[label=$(\roman*)$]
    \item It holds $\KR_{p, C}\left(\mu_n, \mu\right) \to 0$ if and only if the sequence $\left(\mu_n\right)_{n \in \NN} \subset \MCsX$ weakly converges for $n \to \infty$ to $\mu$. In particular, the sequence $\left(\mu_n\right)_{n \in \NN}$ is tight.
    \item If for $n \to \infty$ the sequences $\left(\mu_n\right)_{n \in \NN},\left(\nu_n\right)_{n \in \NN}$ weakly converge to $\mu, \nu$ respectively, then $\KR_{p, C}\left(\mu_n, \nu_n\right) \to \KR_{p, C}(\mu, \nu)$.
\end{enumerate}
        
\end{proposition}

\section{Upper bounds on performance of empirical unbalanced optimal transport} \label{sec:main}
In this section, we derive our main results on the statistical error of the empirical KRD (\Cref{sec:mu_n_inUOT}) based on empirical measures generated from weakly stationary ST point processes (\Cref{def:ST_stat_L2_process_def}).  We note that for such a process $\Xi$ there always exists a natural (spatial) population measure $\mu\in\MCsX$ which fulfills $\mu(\cdot) = \EE[t^{-1}\Xi(\cdot \times (0,t])]$ for all $A\subseteq \BC(\XC)$ (\Cref{prop:time-reduction}) and an empirical measure $\hat \mu_t(\cdot) = t^{-1}\Xi(\cdot \times (0,t])$. 
Extensions to empirical unbalanced OT are detailed in \Cref{sec:UOT}, and we provide the proof for our main result for the empirical KRD  in \Cref{sec:proof_main_result}.

\subsection[Statistical error bounds for empirical Kantorovich--Rubinstein distance]{Statistical error bounds for empirical Kantorovich--Rubinstein\\ distance} \label{sec:mu_n_inUOT}

Throughout this subsection, we consider general parameters $p\geq 1$, $C>0$ for the KRD, which are chosen by the practitioner and thus known. To streamline exposition, we first focus on the statistical bounds for $\KR_{p,C}(\hat \mu_t, \mu)$, i.e., where the empirical measure is compared with the population measure, and later treat the general two-sample case.  

As a first step, we show  a generic upper bound that solely depends on the mass penalization parameter and the total mass of the underlying intensity measure. %
\begin{lemma}\label{lem:trivialUpperBound}
    Let $(\XC,d)$ be a metric space and let $\Xi$ be a weakly stationary ST point process on $\XC \times \tgz$  with population measure $\mu\in \MCsX$ and empirical measure $\hat \mu_t$. Then, for all $t\in \tgz$, 
    \begin{equation} \label{eq:trivialUpperBound}
    	\begin{aligned}
    		&\EE[\KR_{p, C}(\mun, \mu)] \leq 2^{1/p}C \, m_\mu^{1/p}.
    	\end{aligned}
    \end{equation}
\end{lemma}

The statistical bound in \eqref{eq:trivialUpperBound} is to be interpreted as a benchmark bound, which always holds and confirms that for small mass penalization parameter $C$ or small total mass $m_{\mu}$, the statistical error of the empirical $(p,C)$-KRD is small because the KRD is already small. All subsequent results will build on this foundation by demonstrating that the quantity on the left-hand side vanishes as the sample size $t$ increases. As a preliminary result of this type, we confirm consistency of the empirical KRD under a qualitative analog to Assumption \ref{ass:Xi_mixing_Var_beta}.

\begin{lemma}\label{prop:consistency}
	Let $(\XC,d)$ be a c.s.m.s.\ and let $\Xi$ be a weakly stationary ST point process on $\XC\times \tgz$ with population measure $\mu\in \MCsX$ and empirical measure $\hat \mu_t$. Assume that for every open set  $A\subseteq \XC$ it holds $\Var(\Xi(A\times (0,t])) = o(t^2)$ for $t\to \infty$. Then,  it follows that $$\lim\limits_{t\to \infty} \EE\left[\KR_{p,C}(\mun, \mu)\right]=0.$$
\end{lemma}

Our main result in this section is a quantitative version of \Cref{prop:consistency}. To this end, we employ \ref{ass:Xi_mixing_Var_beta} and introduce an additional dimensionality constraint on the domain, which we formalize in terms of covering numbers as part of Assumption \ref{ass:XC_covering_cond}. Note that this assumption implies the total boundedness of the space $\XC$.

\begin{theorem}[Estimation of measure in KRD] \label{th:est_mu_inUOT_beta}
Let $(\XC, d)$ be a c.s.m.s.\ for which the following assumption is satisfied. 
\begin{description}
	\labitem{(Dim$_{\alpha}$)}{ass:XC_covering_cond} There are constants $A>0, \alpha \geq 0$ such that the covering number\footnote{The covering number $\mathcal{N}(\varepsilon, \mathcal{X}, d)$  of $\XC$ for precision $\epsilon>0$ is defined as the smallest integer $N \in \NN$ such that there exists a set $\{s_1, \dots, s_N\} \subset \XC$ such that  for every $x \in \XC$ there exists some index $i$ with $d(x, s_i) \leq \varepsilon$.} of $\XC$ is bounded~by
$$
\mathcal{N}(\varepsilon, \mathcal{X}, d) \leq A \varepsilon^{-\alpha} \quad \text { for all } \varepsilon \in(0, C].
$$
\end{description}
Let $\Xi$ be a weakly stationary ST point process on $\XC \times \tgz$ with population measure $\mu \in \MCsX$ and empirical measure $\hat \mu_t$. Assume $\Xi$ satisfies \ref{ass:Xi_mixing_Var_beta} for $\beta \in [0,1)$, $\kappa >0$, and $t_0 \in \tgz$. 
Then, it follows for all $t \geq t_0$ that %
\begin{equation} 
\begin{aligned}
	&\;\;\EE[\KR_{p,C}(\mun, \mu)] \\[-0.2cm]&\lesssim_{p, \alpha} %
  \, C \kappa^{1/2p} t^{(-1+\beta)/2p}
  +  A^{1/(2p \vee \alpha)} \begin{cases}
        C^{1-\alpha/2p} \kappa^{1/2p} t^{(-1+\beta)/2p} & \text{if } \alpha < 2p, \\
        \kappa^{1/2p} t^{(-1+\beta)/2p} \log(2+C^pt^{(1-\beta)/2}(A\kappa)^{-1/2}m_{\mu}) & \text{if } \alpha = 2p, \\
        m_{\mu}^{-2/\alpha+1/p} \kappa^{1/\alpha} t^{(-1+\beta)/\alpha} & \text{if } \alpha > 2p.
    \end{cases}%
    \label{eq:est_mu_inUOT_beta}
    \end{aligned}
\end{equation}
\end{theorem}

\begin{corollary}\label{th:est_mu_inUOT}
  Let $(\XC,d)$ be a bounded c.s.m.s.\ and assume the setting of \Cref{th:est_mu_inUOT_beta}, including \ref{ass:XC_covering_cond} and \ref{ass:Xi_mixing_Var_beta} with $\beta = 0$. Then it follows for all $t \geq t_0$ that 
  \begin{equation} 
  \begin{split} 
    &\;\;\EE[\KR_{p,C}(\mun, \mu)] \\[-0.4cm]
    & \lesssim_{p, \alpha} %
   C \kappa^{1/2p} t^{-1/2p} 
  + A^{1/(2p \vee \alpha)} \begin{cases}
        C^{1-\alpha/2p} \kappa^{1/2p} t^{-1/2p} & \text{if } \alpha < 2p, \\
       \kappa^{1/2p}t^{-1/2p} \log(2+C^pt^{1/2}(A\kappa)^{-1/2}m_{\mu}) & \text{if } \alpha = 2p, \\
        m_{\mu}^{-2/\alpha  + 1/p}\kappa^{1/\alpha}t^{-1/\alpha}(t) & \text{if } \alpha > 2p.
    \end{cases}%
    \end{split}\label{eq:est_mu_inUOT}
\end{equation}
\end{corollary}

We further note that in case two processes $\Xi$ and $\Eta$ are observed over potentially different time frames, a bound for the statistical error of the plug-in estimator follows from \Cref{th:est_mu_inUOT_beta} via the triangle inequality. 
\begin{corollary}\label{thm:UOT_beta_distinct}
  Let $(\XC, d)$ be a bounded c.s.m.s.\ which fulfill \Cref{ass:XC_covering_cond}. Let $\Xi, \, \Eta$ be two weakly stationary ST point processes on $\XC\times \TC_{>0}$ with population measures $\mu, \, \nu \in \MCsX$ and empirical measures $\hat \mu_t, \, \hat \nu_s$. Assume $\Xi$ and $\Eta$ satisfy \ref{ass:Xi_mixing_Var_beta} for $\beta>0$, $\kappa>0$, and $t_0 \in \TC_0$. Then, it follows for all $t\wedge s\geq t_0$ that 
	\begin{align*}
        &\EE\left[\left| \KR_{p,C}(\hat \mu_t, \hat \nu_s)- \KR_{p, C}(\mu, \nu) \right|\right] \lesssim_{p, \alpha}   C \kappa^{1/2p} (t \wedge s)^{(-1+\beta)/2p} \\
        &\quad  +  A^{1/(2p \vee \alpha)} \begin{cases}
        C^{1-\alpha/2p} \kappa^{1/2p} (t \wedge s)^{(-1+\beta)/2p} & \text{if }\alpha < 2p, \\
        \kappa^{1/2p}(t \wedge s)^{(-1+\beta)/2p} \log^{1/p}(2+C^p(t\wedge s)^{(1-\beta)/2}(A\kappa)^{-1/2}m_{\mu}) & \text{if } \alpha=2p, \\
        m_{\mu+\nu}^{-2/\alpha+1/p} \kappa^{1/\alpha}(t \wedge s)^{(-1+\beta)/\alpha} & \text{if } \alpha > 2p.
    \end{cases} 
\end{align*}
\end{corollary}

Let us now illustrate some explicit implications of \Cref{th:est_mu_inUOT_beta} for the ST Poisson point process, where Assumption \ref{ass:Xi_mixing_Var_beta} is met for $\beta = 0$, and describe a few settings where covering number bounds for \ref{ass:XC_covering_cond} are readily available.

\begin{example}[Classical i.i.d.\ sampling] \label{ex:iid_sampling}
  Our discussion from \Cref{sec:Bimomial} confirms that our bounds imply sharp rates for the empirical KRD when $\mu$ and $\nu$ are probability measures and i.i.d.\ sampling is considered. In particular, for fixed $C\geq \diam(\XC)$, the KRD coincides with the Wasserstein distance, and we obtain the same convergence rates as for the empirical Wasserstein distance from \eqref{eq:WassersteinRates}. 
\end{example}

\begin{example}[Poisson point process]
    If $\Xi$ is a time-homogeneous (and space-
  inhomogeneous) Poisson point process on $\XC\times (0, \infty)$ with intensity measure $\mu \otimes \Leb_{\RR}$ for $\mu\in \MCsX$ it follows according to \Cref{sec:example_PoissonPP} that Assumption \ref{ass:Xi_mixing_Var_beta} is met for $\beta = 0$ and $\kappa = m_{\mu}$. For this case, \Cref{th:est_mu_inUOT}  asserts under \ref{ass:XC_covering_cond} for $\alpha\geq 0$ and $A>0$ that
\begin{align*} 
	&\EE[\KR_{p,C}(\mun, \mu)] \\
  &\lesssim_{A, p, \alpha} %
  (Cm_{\mu}^{1/p})\wedge \left(C m_\mu^{1/2p} t^{-1/2p} +  m_{\mu}^{1/p} \begin{cases}
        C^{1-\alpha/2p} (m_\mu t)^{-1/2p} & \text{if } \alpha < 2p, \\
        (m_\mu t)^{-1/2p} \log^{1/p}(2+C^p(m_\mu t)^{1/2}) & \text{if } \alpha = 2p, \\
        (m_\mu t)^{-1/\alpha} & \text{if } \alpha > 2p.
    \end{cases} \right)
\end{align*}
This bound is comparable to that from \eqref{eq:WassersteinRates} for the empirical Wasserstein distance based on $n$ i.i.d.\ random variables from the probability measure $\mu/m_{\mu}\in \PC(\XC)$ with the identification that the sample size $n$ is proportional to $m_{\mu} t$.  This is not surprising as the expected number of realizations of $\Xi$ within the time frame $(0,t]$ scales with order $m_{\mu}t$. At the same time, under a larger mass $m_{\mu}$, the $(p,C)$-KR distance scales proportionally to the $p$-th root of the mass of the underlying measures. Consequently, while an increase in mass leads to a greater number of observations, the overall error increases as the $(p,C)$-KRD is inflated. %
\end{example}

\begin{example}[Assumption \ref{ass:XC_covering_cond} for different $\XC$] 
\leavevmode 

\begin{enumerate}[label=$(\roman*)$]
  \item %
  Let $R > 0$ be the radius of the centered ball $\XC = \BB(0, R) \coloneqq \{x \in \RR^d: \|x\|_2 \leq R \}$. It is well-known (see, e.g., \citet[Proposition 4.2.12]{vershynin2018high}) that (for $R$ large enough)
  \begin{equation*}
      \NC(\varepsilon, \BB(0, R), \|\cdot\|_2) \leq \frac{\Leb(\BB(0, R + \varepsilon/2))}{\Leb(\BB(0, \varepsilon/2))} = \left( 1 + \frac{2R}{\varepsilon}\right)^d \overset{R \geq \varepsilon/2}{\leq} (4R)^d \varepsilon^{-d},
  \end{equation*}
  so \ref{ass:XC_covering_cond} holds with $\alpha = d$ and $A = (4R)^{d}$.  Repeating the calculations in the proof of \Cref{th:est_mu_inUOT}, we can get an explicit bound in terms of $R$. We obtain 
  \begin{align*}
      &\EE[\KR_{p, C}(\mun, \mu)] \\
      &\lesssim_{d,p, \kappa,m_\mu}
  \left(C m_\mu^{1/p}\right) \wedge \left(C t^{-1/2p} + R \cdot\begin{cases}
      C^{1-d/2p}t^{-1/2p} & \text{if } d < 2p, \\
      t^{-1/2p} \log(2+C^pt^{1/2}) & \text{if } d = 2p, \\
      t^{-1/d} & \text{if } d > 2p.
  \end{cases} \right)
  \end{align*}
    \item Let $\XC = \{x_1, \dots, x_{N}\}$ be a finite metric space with $N \in \NN$ elements. Without any additional assumptions, in this case we have
    $$
        \NC(\varepsilon, \XC, d) \leq N, %
    $$
    meaning that \ref{ass:XC_covering_cond} holds with $A = N$ and $\alpha = 0$. This implies 
    \begin{align}\label{eq:discreteBounds}
        \EE[\KR_{p, C}(\mun, \mu)] \lesssim_{d,p,m_\mu}  C \wedge \left( C N^{1/2p}m_{\mu}^{1/2p} t^{(-1+\beta)/2p}\right) . %
    \end{align}
    If we further impose that $\XC$ is contained in a ball $\BB(0,1)$ in $\RR^d$, further refinements of the rate in terms of $N$ are possible \citep{sommerfeld19FastProb, heinemann2022unbalanced}. To be precise, in this case we have $
        \NC(\varepsilon, \XC, d) \leq \min(4^d\epsilon^{-d}, N)
    $ and from \citet[Section 2.4]{heinemann2022unbalanced} it follows that the bound \eqref{eq:discreteBounds} tightens to
    \begin{align}\label{eq:discreteRatesWasserstein}
    	C \wedge \left( C N^{\max\{1/2p-1/d, 0\}}m_{\mu}^{1/2p} t^{(-1+\beta)/2p}\right),
    \end{align}
    where we suppress a logarithmic term in $N$ for $d = 2p$.
Moreover, the bound given in \eqref{eq:discreteRatesWasserstein} aligns with the known bounds on the accuracy with which an empirical measure approximates a finitely supported population measure in Wasserstein-$p$ distance, provided we identify $t$ as the sample size $n \in \NN$ representing the number of i.i.d.\ observations \citep{sommerfeld19FastProb}.
\end{enumerate}
\end{example}

Let us now discuss some more insights from \Cref{th:est_mu_inUOT_beta}. 

\begin{remark} \label{rem:mu_in_krd}
\begin{enumerate}[label=$(\roman*)$] 
  \item (On the conditions) \Cref{th:est_mu_inUOT_beta} effectively decouples the statistical bounds for the empirical KRD between the spatial empirical measure from an ST point process and the population counterpart into a geometric and a stochastic component. The first condition requires that the domain $\XC$ needs to be of small intrinsic dimension, i.e., \ref{ass:XC_covering_cond} is met for small $\alpha>0$. Alternatively, this condition can be understood as imposing that $\mu$ is defined on a generic domain but is concentrated on a domain of intrinsic dimension $\alpha$.  The second condition boils down to requiring that the sum of variances over partitions of the ST point process scales close to linear in $t$, i.e., if \ref{ass:Xi_mixing_Var_beta} is met for $\beta\geq 0$ close to zero. Later in \Cref{sec:ST_proc} we confirm that this property can be understood as demanding that every point induces only a bounded $(\beta=0)$ or slowly growing $(\beta>0)$ number of other points over the subsequent time. 
  
  \item (Decomposition of the bound) %
  The first term in \eqref{eq:est_mu_inUOT_beta}, which is independent of $\alpha > 0$, arises from the necessity to estimate the total mass of the spatial population measure based on the  ST point process. It scales linearly in $C$.  %
    The second term in \eqref{eq:est_mu_inUOT_beta} aligns with the convergence behavior of the empirical Wasserstein distance between probability measures based on i.i.d.\ samples in terms of the sample size. 
    Further, for the regime $\alpha \in [0, 2p)$ we observe an interpolation phenomenon $C^{1-\alpha/2p}$ in the mass penalization parameter, leading to a linear dependency $C$ for $\alpha = 0$ whereas for $\alpha \searrow 2p$ the dependency in $C$ becomes logarithmic. 
    Moreover, under $\alpha \geq 2p$, the dependency in $t$ deteriorates exponentially in $\alpha$ and is independent of $C$. 
	\item (Curse of dimensionality) Complementary to our upper bounds in \Cref{sec:lowerbounds} we establish for  certain ST point processes fulfilling \ref{ass:Xi_mixing_Var_beta} with $\beta = 0$ matching pointwise lower bounds and minimax rate optimality of the empirical measure in estimating its population counterpart with respect to $m_\mu$, $t$ and $C$. Only for $\alpha = 2p$ we observe a discrepancy of a logarithmic order between the lower and upper bound, and the sharp rate of convergence remains open. Overall, our insights show that estimation of measures with respect to the $(p,C)$-KRD suffers from the curse of dimensionality, i.e., that slow convergence rates are inevitable for large $\alpha$. Moreover, we also provide an example for an ST point process fulfilling \ref{ass:Xi_mixing_Var_beta} with $\beta \in (0,1)$ for which we confirm a matching pointwise lower bound in $t$ for $\alpha< 2p$, emphasizing the necessity of Assumption \ref{ass:Xi_mixing_Var_beta}. 
    \item (Local Wasserstein dimension) Assumption \ref{ass:XC_covering_cond} can be slightly relaxed to demanding that it is satisfied only for $\varepsilon \in (\gamma t^{(-1+\beta)/(2\wedge \alpha)},C]$ with $\gamma>0$, which implies an analogous bound which additionally also depends on $\gamma$. This reflects the multiscale behavior of the empirical plug-in estimator previously observed for empirical OT \citep{weed2019sharp}, meaning that if the support admits (up to some $t$-dependent scale) similar covering number bounds as a low-dimensional domain, then small error bounds are to be expected.  %
	\item (Lower complexity adaptation) The recent  work \cite{hundrieser2022empirical} asserts for the empirical OT cost that the convergence rate is essentially driven by the minimum intrinsic dimension among the measures, a phenomenon coined \emph{lower complexity adaptation (LCA)}. This conclusion is not possible by invoking the triangle inequality, and the proof is based on controlling certain empirical processes over suitable function classes. Hence, the validity of LCA for the empirical $(p,C)$-KRD does not directly follow from \Cref{th:est_mu_inUOT_beta}, and we leave it as an open question to assess if it can be confirmed under a similar condition as \ref{ass:Xi_mixing_Var_beta}. We note, however, that \citet[Remark 6.5]{deb2024tradeoff} establish LCA of the empirical Wasserstein distance for non-i.i.d.\ data under mixing assumptions that appear stronger than our variance-growth requirement. %
\end{enumerate}	
\end{remark}

\subsection{Implications to other unbalanced optimal transport divergences} \label{sec:other_UOTs}

Our statistical bounds for the empirical KRD yield similar convergence statements for other unbalanced OT based divergences. In the following, we briefly discuss these implications. To this end, we first introduce the setting under which we will be operating and discuss it afterward.

\begin{definition}\label{def:niceUOTformalism}
Let $(\XC, d)$ be a metric space. We consider a divergence $\U : \MCsX \times \MCsX \to [0, \infty)$ to be \emph{$p$-dominated}, if for any $\mu, \nu, \tau \in \MCsX$ and any $a \geq 0$  there exists $p \geq 1$ such that %
    \begin{align}
         &\U(\mu, \nu) \lesssim \U(\mu, \tau) + \U(\tau, \nu) , \label{ass:U1} \\
        &\U\left(\mu,\nu\right) \lesssim W_p\left(\mu, \nu\right) \quad \text{if } m_{\mu} = m_{\nu}, \label{ass:U2}\\
        &\U\left(\mu,a\mu\right) \lesssim |1-a|^{1/p} m_\mu^{1/p} \label{ass:U3},
    \end{align}
    where the hidden constants in the three inequalities do not depend on the measures $\mu$, $ \nu$ or $\tau$.
\end{definition}

\begin{example} Many unbalanced OT formalisms (recall introduction) satisfy above conditions, some of which we now recall and discuss. %
    \begin{enumerate} [label=$(\roman*)$]
        \item The first example we provide is a variant of the \emph{earth mover's distance} (EMD) used by  \cite{komiske2019metric} as a metric on the space of collider events. In this context of high energy physics, two collider events are modelled as measures, say $\mu$ and $\nu$, supported on finite sets $I$ and $J$ and with possibly different total energies $m_\mu = \sum_{i \in I} \mu(i)$ and $m_\nu = \sum_{j \in J} \nu(j)$, respectively. Here $\mu(i)$ has the meaning of the energy of $\mu$ at the particle $i$. We introduce
        \begin{equation}
            \begin{aligned}
            & \quad \quad \quad \quad \operatorname{EMD}\left(\mu, \nu\right)=\min _{\pi(i, j) \geq 0 \, \forall i \in I, j \in J} \sum_{i,j} \pi(i, j) \frac{\theta_{i j}}{R}+\left|m_\mu - m_\nu\right|, \\
            &\text{subject to}\quad \sum_{j} \pi(i, j) \leq \mu(i), \quad \sum_{i } \pi(i, j) \leq \nu(j), \quad \sum_{i,j } \pi(i,j)= m_\mu \wedge m_\nu ,
            \end{aligned}
        \end{equation}
        where $\theta_{i j}$ is the angular distance
        between particles and $R$ is a parameter that controls the relative importance of the two minimized terms. Since the $\operatorname{EMD}$ metrizes the space of collider events, \eqref{ass:U1} holds. To see \eqref{ass:U2},  we note that if the total energies of $\mu$ and $\nu$ coincide, we have
        $$
            \operatorname{EMD}\left(\mu, \nu\right) \leq \min _{\sum_{j} f_{ij}=\mu_i, \sum_i \pi(i,j)=\nu_j } \sum_{i,j } \pi(i,j)\frac{\theta_{i j}}{R} = \frac{1}{R} W_1\left(\mu, \nu\right),
        $$
        where the 1-Wasserstein distance is with respect to the angular distance $\theta$.
        Finally, if $\nu = a \mu$, then $\pi(i,j )\equiv 0$ is a feasible plan, and therefore  
        $$
            \operatorname{EMD}\left(\mu, a  \mu\right) \leq \min _{\pi(i,j)\equiv 0} \left|m_\mu-a m_\mu\right| = |1-a| m_\mu.
        $$
        \item The $\KR^p_{p,C}$ functional introduced in the current article is also feasible. Indeed, by \Cref{prop:metricKRD},  \eqref{ass:U1} holds.
        For \eqref{ass:U2} observe that, if $m_\mu = m_\nu$, we can estimate
        $$
            \KR^p_{p,C}(\mu, \nu) \leq \inf_{\pi \in \Pi_{=}(\mu, \nu)} \int_{\XC^2} d^p(x_1, x_2) \pi (d x_1, \dif x_2) + \frac{C^p}{2}(2m_\mu - 2m_\pi) =  W_p^p\left(\mu, \nu\right).
        $$
        Finally, for $\nu = a \mu$, the optimal plan is to not transport anything and delete the excess mass of $|m_\mu - m_{a\mu}| = |1-a|m_\mu$ at rate $C^p/2$, hence $\KR^p_{p,C}(\mu, a\mu) \leq \frac{C^p}{2}|1-a|m_\mu$ and \eqref{ass:U3} follows.
        
        We highlight that our approach to proving \Cref{th:est_mu_inUOT_beta} does not quite follow the general strategy outlined below in \Cref{th:est_mu_in_U} (although essentially employs the same idea). This is due to the necessity for extra attention to be paid to the dependency on the parameter $C$.
        \item More generally, let us show how the conditions \eqref{ass:U2} and \eqref{ass:U3} look like in the generic \emph{entropy-transport }(ET) problem class introduced by \cite{liero2018optimal}. 
        They define for $\mu_1,  \mu_2 \in \MCsX$
        \begin{equation} \label{eq:liero_primal_ex}
    		\mathrm{ET}(\mu, \nu) = \inf_{\pi \in \MC(\XC^2)} \int_{\XC^2} c(x_1, x_2) \pi (d x_1, \dif x_2) + \DC_1\left( \pi_1 \mid \mu_1\right) + \DC_2\left(\pi_2 \mid \mu_2\right),
		\end{equation}
where $c$ is an l.c.s. cost function, and the  relative entropy $\DC_i$ of the $i$-th marginal  $\pi_i$ of $\pi$ to $\mu_i$ is defined by
$$
	\DC\left(\pi_i \mid \mu_i \right)\coloneqq\int_{\XC_1} \varphi \left(\frac{\dif \pi_i}{\dif \mu_i}\right) \dif \mu_i + \left( \varphi \right)_{\infty}^{\prime} \pi_i^{\perp}\left(\XC\right), 
$$
for   $\pi_i^{\perp}$ denoting the singular part of $\pi_i$ with respect to $\mu_i$, the function $\varphi_i\colon [0, \infty) \to [0, \infty)$ is  convex and l.c.s.\ function and $ (\varphi_i)_{\infty}^\prime \coloneqq \lim _{s \rightarrow \infty} \varphi_i(s)/s$.     
        
        For \eqref{ass:U2} we have to impose that the cost function $c$ is the power of a metric $d$, i.e. $c=d^p$ for some $p\geq 1$.
        For $m_{\mu_1}=m_{\mu_2}$ we have
\begin{align*}
    \mathrm{ET}(\mu_1,\mu_2) &= \inf_{\pi \in \MC(\XC^2)} \int_{\XC^2} d^p \dif \pi + \DC_1(\pi_1 \mid \mu_1) + \DC_2(\pi_2 \mid \mu_2) \\
    &\leq \inf_{\Pi_{=}(\mu_1, \mu_2)} \int_{\XC^2} d^p \dif \pi + \DC_1(\pi_1 \mid \mu_1) + \DC_2(\pi_2 \mid \mu_2).
\end{align*}
Since for $\pi \in \Pi_{=}(\mu_1, \mu_2)$ it holds that $\pi_i = \mu_i$, we have
$$
    \DC_i(\pi \mid \mu_i) = \int_{\XC_i} \varphi_i(1) \dif \mu_i = \varphi_i(1) m_{\mu_i},
$$
which implies that \eqref{ass:U2} is met if $\varphi_i(1)=0$ for both $i\in \{1,2\}$. For \eqref{ass:U3}, we note for $a \leq 1$ that $\pi = (\id, \id)_{\#}a \mu$ is a feasible solution transport plan, hence
$$
    \mathrm{ET}(\mu,a\mu) \leq \DC_1(\pi_1 \mid \mu) + \DC_2(\pi_2 \mid a \mu) = \varphi_1(a) m_\mu + \varphi_2(1)m_\mu = F_1(a) m_\mu.
$$
If $a > 1$, we take $\pi = (\id, \id)_{\#}\mu$ bound
$$
    \mathrm{ET}(\mu,a\mu) \leq \DC_1(\pi_1 \mid \mu) + \DC_2(\pi_2 \mid a \mu) = \varphi_1(1) m_\mu + \varphi_2(1/a)m_\mu = \varphi_2(1/a) m_\mu,
$$
therefore the conditions are $\varphi_i(s)\leq 1-s$ for $s \leq 1$ and $\varphi_{i}(1/s) \leq s-1$ for $s > 1$,  which are both met if
$$
    \varphi_i(s)\leq 1-s \text{ for } 1\leq s \leq 1.  
$$
It appears that the question of for which $\varphi$ the entropy-transport induces a metric on $\MCsX$ has not been the subject of discussion. 
However, \cite{liero2018optimal} provide a list of examples that induce a metric on the corresponding space of measures, two of which we mention below in the next point.
        \item \emph{Gaussian--Hellinger} and \emph{Hellinger--Kantorovich} distances are special cases of ET problem for the Euclidean metric, defined by setting $c(x,y)\coloneqq \|x-y\|^2$ or $c(x,y)\coloneqq -2 \log \cos (\|x-y\| \wedge \pi)))$, respectively, and $\varphi_1(s)=\varphi_2(s) =s \log s - s +1 $. Condition \eqref{ass:U1} follows by \cite[Corollary 7.14]{liero2018optimal}. Since $\varphi_1(1)=\varphi_2(1) = 0$ and $\varphi_1(s)=\varphi_2(s) \leq 1-s$ for $s \leq 1$, \eqref{ass:U2}-\eqref{ass:U3} follow from the above considerations for the $\operatorname{ET}$ functional.
    \end{enumerate}
\end{example}

\begin{theorem}[Estimation of measure in generic unbalanced OT] \label{th:est_mu_in_U}
Let $(\XC, d)$ be a bounded c.s.m.s.\ satisfying \ref{ass:XC_covering_cond}. Then, for a $p$-dominated divergence $D \colon \MCsX \times \MCsX \to [0, \infty)$ and a weakly time-stationary $L^2$ ST point process  $\Xi $ on $\XC \times \tgz$ with spatial intensity $\mu \in \MCsX$ with corresponding empirical measure $\mun$, satisfying \ref{ass:Xi_mixing_Var_beta},
it holds for all $t \geq t_0$ that 
\begin{equation}\label{eq:est_Uhat}
    \EE[\U(\mun, \mu)]\lesssim_{A, p, \alpha, m_{\mu}} t^{(-1+\beta)/2p} + \begin{cases}
    t^{(-1+\beta)/2p} & \text{if } \alpha < 2p,\\
     t^{(-1+\beta)/2p}\log(2+t)^{1/p} &\text{if } \alpha = 2p,\\
    t^{(-1+\beta)/\alpha} \log(2+t) & \text{if } \alpha > 2p.
    \end{cases}
\end{equation}
Analogous bounds also hold mutatis mutandis for the two-sample case. 
\end{theorem}
The proof essentially reduces to upper-bounding $\EE[\U(\mun, \mu)]$ by the sum of two terms. The first term corresponds to the expected absolute mass deviation of $\mun$, which accounts for the first term in \eqref{eq:est_Uhat}. The second term is a variant of the empirical KR distance
(cf. \Cref{th:est_mu_inUOT_beta}). In particular, all the insights outlined in \Cref{rem:mu_in_krd} naturally extend to this setting.

\subsection{Proof of Theorem \ref{th:est_mu_inUOT_beta}}\label{sec:proof_main_result}

The idea which lies at the heart of the proof is a variant of the  dyadic discretization technique, a well-known tool in the statistical OT literature which leads to sharp convergence results (see e.g. \cite{dereich2013constructive, fournier2015rate, weed2019sharp, sommerfeld19FastProb}). For a similar analysis of the empirical KRD, we boil the problem down to the discrete KRD developed in \cite{heinemann2023kantorovich}, where such bounds with an explicit dependency on $C$ have already been developed. They are based on constructing a suitable \emph{ultrametric tree}, for which an explicit representation of the Wasserstein and Kantorovich--Rubinstein distance is available, see e.g., \cite{kloeckner2015geometric} for a proof. We briefly describe the construction now and defer technical calculations to \Cref{sec:proof_est_mu}.

Let $\mathfrak{T} = (V, E)$ be a rooted tree with at least two elements, i.e., a simple graph with vertex set $V$ where one element serves as the root and edge set $E$ which is connected and acyclic. Assuming every edge is assigned a positive length, for any two nodes $u, v \in V$ we define their tree distance $d_{\mathfrak{T}}(u,v)$ as the total length of the edges along the unique path connecting $u$ and $v$. For a node $v \in V$ different from the root, we denote by $\mathrm{par}(v)$ its parent, i.e., the unique node $v'$ such that $(v,v')\in E$ and $v'$ is closer to the root than $v$, and we denote by $\CC(v)$ its children, that is, the set of nodes such that the unique path from the root to these nodes passes through $v$. Further, a leaf node is a node $v \in V$ such that it is not a parent node of another node. 
\begin{definition}
A tree $\mathfrak{T} = (V, E)$ is \emph{ultrametric} if all the leaf nodes are at the same distance from the root. Equivalently, there exists a monotonically strictly decreasing \emph{height function} $h: V \to [0, \infty)$, such that $d_{\mathfrak{T}}(v, \mathrm{par}(v)) = h(\mathrm{par}(v)) - h(v)>0$ whenever $v$ is not the root and $h(v)=0$ for all leafs $v$ of $\mathfrak T$. 
\end{definition}
In the following, we build an appropriate ultrametric tree based on the covering sets of $\XC$ of different precisions $\varepsilon>0$, where each covering set will be associated with a specific level in the tree, i.e., they will all admit the same height value. 

To elaborate, for $\varepsilon > 0$, consider the $\varepsilon$-covering number $\NC_\varepsilon \coloneqq \NC(\varepsilon, \XC, d)$ of $\XC$ and  let $S_\varepsilon = \{s_1, s_2, \dots, s_{\NC_\varepsilon}\}$ be the associated covering set, that is, for every $x \in \XC$ there exists an $i \in \{1, 2, \dots, \NC_\varepsilon\}$ such that $d(x, s_i) \leq \varepsilon$. 
We will consider the depth level 
\begin{equation} \label{eq:L_definition}
	L \coloneqq L(\varepsilon) \coloneqq  \lceil 0\vee \log_2( \diam( S_\varepsilon)/\varepsilon) \rceil,
\end{equation}
 set $Q_{L+1} = S_\varepsilon$ and define for level $j \in \{0, \dots, L\}$ the set $Q_j$ as the $2^{-j}\diam (S_\varepsilon)$-covering set of $ S_\varepsilon$. %
As our vertex set, we consider $V = \{(q, j), \ q \in Q_j, \ j = 0, 1, \dots, L+1\}$ and in our construction the leaf nodes will be given by $\{(q, L+1), \ q \in S_\varepsilon\}$. Whenever the meaning is clear, we also drop the level indication and write $q \in Q_j$ instead of $(q, j) \in Q_j$. For $j=0, \dots, L$, a node $(q, j)$ at level $j$ is connected to a single node $\left(q', j+1\right)$ at level $j+1$ if $d\left(q, q'\right) \leq 2^{-j} \diam (S_\varepsilon)$; if there are multiple such nodes, ties are broken arbitrarily.  The length of the corresponding edge is set equal to $2^{-j} \diam (S_\varepsilon)$. Consequently, the height of each node $(s, j) \in V$ only depends on its assigned level and is given by
\begin{equation*}
    h_{L}((s,l)) := h_{L}(l) = \sum_{j=l}^L 2^{-j} \diam (S_\varepsilon)= (2^{1-l}-2^{-L}) \diam (S_\varepsilon),
\end{equation*}
making the tree $\mathfrak T$ ultrametric.
In our proof, we approximate $\KR_{p,C}^p(\mu, \nu)$ by its discrete version $\KR_{p,C}^p(\mu^\varepsilon, \nu^\varepsilon)$, where $\mu^\varepsilon$ and $\nu^\varepsilon$ are measures supported on the leaf nodes $Q_{L+1} = S_{\epsilon}$ of the ultrametric tree $\mathfrak{T}$. To this end, consider a Voronoi partition $\{V_s\}_{s\in S_{\epsilon}}$ of $\XC$ induced by $Q_{L+1}$, given for $s_i \in S_{\epsilon}$ (recall the enumeration from above) by %
\begin{align}\label{eq:VoronoiTesselationDef}
   V_{s_i} \coloneqq \{x \in \XC: d(x, s_i) \leq d(x, s') \text{ for all } s' \in S_\epsilon \} \backslash \bigcup_{1\leq h < i}V_{s_h} \subseteq \XC,
\end{align}
and we define $\mu^{\epsilon}\coloneqq \sum_{s\in S_\epsilon}\mu(V_s)\delta_s$ and $\nu^{\epsilon}\coloneqq \sum_{s\in S_\epsilon}\nu(V_s)\delta_s$ as the corresponding nearest-neighbor projection of $\mu$ and $\nu$, onto $S_{\epsilon} = Q_{L+1}$.
The next lemma provides a quantitative bound between $\KR_{p, C}^{p}(\mu, \nu)$ and $\KR_{p, C}^{p}(\mu^\varepsilon, \nu^\varepsilon)$. The proof is given in Appendix \ref{sec:proof_main_result}.
\begin{lemma}[Discrete approximation of the $(p, C)$-KR distance]\label{lem:uot_proj_bound}
For a c.s.m.s.\ $(\XC, d)$, let $\mu, \nu \in \MCsX$. Then, for $\varepsilon >0, \ p\geq 1$ and $C>0$, it holds that
\begin{equation} \label{eq:uot_proj_bound}
    \KR_{p, C}^{p}(\mu, \nu) \leq 3^{p-1} \inf_{\varepsilon > 0}\left(\KR^p_{p, C}(\mu^\varepsilon, \nu^\varepsilon) + (m_\mu +m_\nu) (C\wedge \varepsilon)^p\right).
\end{equation}
\end{lemma}
To derive an upper bound for the KRD on the  right-hand side of \eqref{eq:uot_proj_bound}, we note by  construction of the ultrametric $d_{\mathfrak{T}}$ that for all $s, s' \in S_\varepsilon$, 
\begin{equation*} %
    d(s, s') \leq d_{\mathfrak{T}}((s, L+1), (s', L+1)),
\end{equation*}
implying that
\begin{equation*}
    \KR_{p, C}^p(\mu^\varepsilon, \nu^\varepsilon) \leq \KR_{d_{\mathfrak{T}}^p, C}^p\left(\mu^{\epsilon}, \nu^{\epsilon}\right).
\end{equation*}

The following bound follows from the explicit expression for KRD on ultrametric trees \cite[Theorem 2.3]{heinemann2023kantorovich}; the exact calculations are deferred to Appendix \ref{sec:proof_est_mu}.

\begin{lemma}\label{lem:KR_tree_est}
Let $\mu^{\epsilon}, \nu^{\epsilon}$ be the discretized versions of measures $\mu, \nu \in \MCsX$, supported on the leaf nodes $Q_{L+1}$ of the ultrametric tree $\mathfrak T$ described above. 
Then, for any $p \geq 1 $, $C >0$, it holds for $\{W_{q,j}\}_{q\in Q_j}$ with $W_{q,j}\coloneqq \bigcup_{s\in \CC(q)}V_s$ that
\begin{equation*}
\begin{aligned}
	\KR_{d_{\mathfrak{T}}^p, C}^p\left(\mu^{\epsilon}, \nu^{\epsilon}\right) &\leq \frac{C^p}{2} |m_\mu - m_\nu| + 2^{p-1} \sum_{j=l_*}^{L+1} (h_L^p(j-1) - h_L^p(j))\sum_{q \in Q_j} |\mu(W_{q,j}) - \nu(W_{q,j})|
\end{aligned}
\end{equation*}
where 
\begin{align}\label{eq:l_star}
	l_* = 1 + \min\left(L,\left\lfloor 0 \vee \log_2\left( \frac{2 \diam(S_\varepsilon)}{C+\varepsilon}\right)  \right\rfloor\right).
\end{align}

\end{lemma}

Combining \Cref{lem:uot_proj_bound} and \Cref{lem:KR_tree_est}, we are ready to prove our main result.
\begin{proof} [Proof of \Cref{th:est_mu_inUOT_beta}]
By Jensen's inequality for the convex function $t\mapsto t^{p}, \ p \geq 1$ in conjunction with \Cref{prop:kr_properties}$(iv)$ and \Cref{lem:uot_proj_bound} we obtain 
\begin{align} 
   \left(\EE[\KR_{p, C}(\mun, \mu)]\right)^{p} &\leq \EE[\KR_{p, C}^p(\mun, \mu)] \notag
   \\&\leq (C^p m_\mu) \wedge\left(3^{p-1}\inf_{\varepsilon>0}\left( \EE[\KR^p_{p, C}(\mun^\varepsilon, \mu^\varepsilon)] +  2m_\mu(C\wedge \varepsilon)^p\right)\right). \label{eq:inf_KR_K}
\end{align}
To upper bound $\EE[\KR^p_{p, C}(\mun^\varepsilon, \mu^\varepsilon)]$, we use \Cref{lem:KR_tree_est}, which implies 
\begin{align}
	\EE[\KR^p_{p, C}(\mun^\varepsilon, \mu^\varepsilon)] \leq &\frac{C^p}{2} \EE[|m_{\mun}-m_\mu|] \notag \\
	& + 2^{p-1} \sum_{j=l_*}^{L+1} (h^p(j-1) - h^p(j))\sum_{q \in Q_j} \EE|\mun(W_{q,j}) - \mu(W_{q,j})|. \label{eq:KR_hash_est_old}
\end{align}
For the first term on the right-hand side we obtain by Assumption \ref{ass:Xi_mixing_Var_beta} that
$$
	\EE[|m_{\mun}-m_\mu|] \leq \sqrt{\Var(\mun(\XC))} \leq \kappa^{1/2} t^{(-1+\beta)/2}.
$$
For the inner sums in the second term, we bound 
\begin{align*}
	\sum_{q \in Q_j}\EE[|\mun(W_{q,j}) - \mu(W_{q,j})|] & = t^{-1}\sum_{q \in Q_j}  \EE \left| \Xi(W_{q, j} \times (0, t]) - t \mu(W_{q, j})\right| \\
    &\leq t^{-1}\sum_{q \in Q_j}  \sqrt{\Var  \left( \Xi(W_{q, j} \times (0, t])\right)} \\
    &\leq t^{-1} |Q_j|^{1/2} \sqrt{\sum_{q \in Q_j}  \Var  \left( \Xi(W_{q, j} \times (0, t])\right)} \\
    &\leq |Q_j|^{1/2} \kappa^{1/2} t^{(-1+\beta)/2},
\end{align*}
where the second  last inequality holds by the Cauchy--Schwarz inequality, and the last one follows by Assumption~\ref{ass:Xi_mixing_Var_beta}.

By definition of $Q_j$ and Assumption \ref{ass:XC_covering_cond}, it holds for each $j = 1, 2, \dots, L$ that $|Q_j| = \NC(2^{-j}\diam (S_\varepsilon), S_\varepsilon, d)
\leq A (\diam(S_\varepsilon))^{-\alpha} 2^{j \alpha}$%
. Next, note that %
\begin{align*}
    & h_{L}^p(j-1) - h_{L}^p(j) \leq 2^{p} 2^{p-jp} (\diam(S_\varepsilon))^{p} , \quad   j = 1, 2, \dots, L \quad \text{ and } \\
    & h_{2, L}^p(L) - h_{2, L}^p(L+1) = h_{2, L}^p(L) = 2^{-Lp}(\diam(S_\varepsilon))^{p}.
\end{align*}
Using the definitions of $L$ and $l_*$ from \eqref{eq:L_definition} and \eqref{eq:l_star}, respectively, we bound \eqref{eq:KR_hash_est_old} by
\begin{align}
    & 2^{p-1} A^{1/2} (\diam(S_\varepsilon))^{p-\alpha/2} \kappa^{1/2} 2^{L(\alpha/2 - p)} t^{(-1+\beta)/2} \notag\\
    & \hspace{0.5cm} + 2^{3p-1} A^{1/2} (\diam(S_\varepsilon))^{p-\alpha/2} \kappa^{1/2} t^{(-1+\beta)/2}\sum_{j=l_*}^{L} 2^{j(\alpha/2 - p)} \notag\\
     =  & \, 2^{p-1} A^{1/2} (\diam(S_\varepsilon))^{p-\alpha/2} \kappa^{1/2}2^{L(\alpha/2 - p)}  t^{(-1+\beta)/2} \notag\\
    & \hspace{0.5cm} + 2^{3p-1} A^{1/2} (\diam(S_\varepsilon))^{p-\alpha/2} \kappa^{1/2}t^{(-1+\beta)/2}\begin{cases}
    		L+1 - l_*, & \alpha = 2p, \\
    		\frac{1}{2^{\alpha/2 - p} - 1} \left(2^{(\alpha/2 - p)(L+1)} - 2^{(\alpha/2 - p)l_*} \right), & \alpha \neq 2p
    \end{cases} \notag\\
     \leq & \  2^{p-1} A^{1/2} \kappa^{1/2}\varepsilon^{p - \alpha/2}  t^{(-1+\beta)/2} \notag \\
     & \hspace{0.5cm} + 2^{3p-1} A^{1/2} \kappa^{1/2} t^{(-1+\beta)/2} \begin{cases}
    		\frac{1}{1-2^{\alpha/2 - p}}  (C+\varepsilon)^{p-\alpha/2}, & \alpha < 2p, \\
    		(L+1 - l_*), & \alpha = 2p, \\
    		\frac{2^{\alpha/2 - p}}{2^{\alpha/2 - p} - 1}  \varepsilon^{p-\alpha/2} , & \alpha > 2p.
    \end{cases} \label{eq:KR_hash_est}
\end{align}
It remains to minimize \eqref{eq:KR_hash_est}$+ \frac{C^p}{2}\kappa^{1/2}t^{(-1+\beta)/2} +  2m_\mu(C\wedge \varepsilon)^p$ in $\varepsilon > 0$ to obtain the desired bound \eqref{eq:est_mu_inUOT_beta}.
The exact calculations are provided in Appendix \ref{sec:proof_est_mu} and conclude the proof. \qedhere

\end{proof}

\section{ST point processes fulfilling the variance growth condition} \label{sec:ST_proc}

In this section, we examine Assumption \ref{ass:Xi_mixing_Var_beta} for spatio-temporal (ST) point processes on $\XC\times \TC$, where $\XC$ is a bounded c.s.m.s.\ and either $\TC = \RR$ or $\TC = \ZZ$.\footnote{Note that we now use the full time axis $\TC$ as suggested for a precise definition of weak stationarity in \Cref{fn:full_time_axis}. Our observations are still restricted to intervals $(0,t] \subseteq \tgz = (0,\infty)$ (recall that we extend interval notation to the case $\TC = \ZZ$).} We present basic tools for ST point processes in \Cref{sec:pp_definition}. This groundwork enables us to introduce Assumption~\ref{ass:Xi_mixing_cov_beta} in \Cref{sec:pp_assumptions}, which is slightly stronger than \ref{ass:Xi_mixing_Var_beta} in general and is expressed in terms of the time-reduced factorial covariance measure, so that it allows a direct interpretation in terms of the dependence of points. Moreover, it enables us to discuss several important examples of point processes satisfying \ref{ass:Xi_mixing_cov_beta} and hence \ref{ass:Xi_mixing_Var_beta} for both $\beta = 0$ and $\beta > 0$ in \Cref{sec:pp_examples}. 
 All proofs in this section are given in  \Cref{sec:proofs_ST_proc}. 

\subsection{Basic theory} \label{sec:pp_definition}

Here we give a brief introduction into elementary point process theory, which we adapt from \cite{daley2003introduction,daley2007introduction} to the spatio-temporal setting. 

To set notation, we will always consider an ST point process $\Xi$ on $\XC \times \TC$, which is defined as a random element on the space 
$ \mathfrak{N} = \mathfrak{N}(\XC \times \TC)$ (with a Borel $\sigma$-field $\NC$) 
of locally finite counting measures on $\XC \times \TC$, i.e., the number of points in $\XC \times I$ is finite for any bounded interval $I$. We standardize the intervals used to left-open and right-closed, denoting the set of such bounded intervals in $\TC$ by $\IC$. 

Moreover, any point process $\Xi $ on $\XC \times \TC$ admits a measurable representation as a sum of Dirac masses at its points, namely there exists a random $\overline{\mathbb{Z}}_{+}$-valued random variable $N$, a sequence $(T_i)_{i \in \NN}$ of $\TC$-valued random elements and a sequence $(X_i)_{i \in \NN}$ of $\XC$-valued random elements such that %
\begin{equation} \label{eq:Xi_sum_repr}
	\Xi = \sum_{i=1}^{N} \delta_{(X_i, T_i)}.
\end{equation}
If $\Xi$ has a minimal point with respect to the time component or by assigning negative indices, it is possible to require that the points are chronologically ordered, but we usually do not do so except where it is explicitly stated (e.g., representing the points of a spatio-temporal Poisson process in a bounded set requires a non-chronological enumeration).

An important tool for describing the process $\Xi$ are its \emph{moment measures}. For $k \geq 1$, we say that \emph{the $k$-th moment measure exists} if
\begin{equation} \label{eq:def_k-mom}
  M_k (B_1 \times \dots \times B_k) \coloneqq \EE[\Xi(B_1) \cdots \Xi(B_k)] 
\end{equation}
is finite for all bounded $B_1, \dots, B_k \in \BC(\XC \times \TC)$. The \emph{$k$-th moment measure} is then defined as the unique measure $M_k$ on $(\XC\times \TC)^k$ satisfying \eqref{eq:def_k-mom}. Using \eqref{eq:Xi_sum_repr}, we may write $M_k = \EE \Xi^k$ (meaning $M_k(A)=\EE [\Xi^k(A)]$ for $A \in \BC((\XC\times \TC)^k)$), where 
$$
	\Xi^k = \sum_{i_1, \dots, i_k=1}^N \delta_{((X_{i_1}, T_{i_1}), \dots, (X_{i_k}, T_{i_k}))}.
$$
We also introduce 
\begin{equation*}
	\Xi^{[k]} = \sum_{i_1, \dots, i_k=1}^{N, \neq} \delta_{((X_{i_1}, T_{i_1}), \dots, (X_{i_k}, T_{i_k}))},
\end{equation*}
where $\neq$ indicates that the sum is taken over pairwise different indices. We call $M_{[k]} \coloneqq \EE \Xi^{[k]}$ the \emph{$k$-th factorial moment measure} of $\Xi$. For our purposes, we focus on the first- and second-order characteristics only and rearrange the components so that spatial and temporal components stay together. In particular, we have for $A, B \in \BC(\XC), \ T, S \in \BC(\TC)$ 
\begin{equation} \label{eq:2facmom}
	M_{[2]} (A \times B \times T \times S) = M_{2} (A \times B \times T \times S) - M_1(( A \times T) \cap (B \times  S)).
\end{equation}
We usually prefer using %
 $M_{[2]}$ over $M_{2}$ because the spurious mass on the diagonal of $(\XC \times \TC)^2$ stemming from using the same points $(X_i,T_i)$ multiple times is removed.

Another important second-order characteristic of a point process is its \emph{factorial covariance measure} $\gamma_{[2]}$, which is a \emph{signed} measure on $\XC^2 \times \TC^2$ defined by 
\begin{equation} \label{eq:faccovm}
	\gamma_{[2]}(A \times B \times T \times S) = M_{[2]}(A \times B \times T \times S) - M_1(A \times T) M_1(B \times S).
\end{equation}
In the same way, we can define $\gamma_{2}$ based on $M_2$.

In what follows, we assume that our point process is $L^2$ and second-order (or weakly) time-stationary, meaning its first two (factorial) moment measures exist and are invariant under time-shifts, i.e., $M_1 \chi_t^{-1} = M_1$ and $M_{[2]} (\chi_t, \chi_t)^{-1} = M_{[2]}$ for all $t \in \TC$, where $\chi_t \colon \XC \times \TC \to \XC \times \TC$, $(x,s) \mapsto (x,s+t)$.
This allows us to present the following result, which can be viewed as a straightforward extension (and reformulation) of \citet[Proposition 8.1.I and 8.1.II(a)]{daley2003introduction}. %
\begin{proposition} \label{prop:time-reduction}
  Let $\Xi$ be a weakly time-stationary $L^2$ point process. Let $\ell=\Leb_{\RR}$ if $\TC=\RR$ and $\ell = \sum_{z \in \ZZ} \delta_z$ if $\TC=\ZZ$. Then, the following assertions hold. 
   \begin{enumerate}
      \item[(i)] There is a measure $\mu \in \MC(\XC)$ such that $M_1 = \mu \otimes \ell$.
      \item[(ii)] There is a locally finite measure $\breve{M}_{[2]}$ on $\XC^2 \times \TC$ such that
      \begin{equation} \label{eq:red2facmom}
        M_{[2]} \psi^{-1}(dx \, dy \, dr \, ds) = \ell(dr) \, \breve{M}_{[2]}(dx \, dy \, ds),
      \end{equation}
      where $\psi(x,y,r,s) \coloneqq (x,y,r,s-r)$.
      The measure $\breve{M}_{[2]}$ is symmetric in the sense that
      $\breve{M}_{[2]}(A \times B \times S) = \breve{M}_{[2]}(B \times A \times -S)$
      for all $A,B \in \BC(\XC)$ and $S \in \BC(\TC)$.
   \end{enumerate}
\end{proposition} 
\begin{remark}
	Using \Cref{prop:time-reduction}$(ii)$, we compute in particular 
	\begin{equation} \label{eq:rfmm_explicit}
      	\breve M_{[2]}(A \times B \times S) = \frac{1}{\ell(T)} M_{[2]}\psi^{-1}(A \times B \times T \times S)
      \end{equation}
      for $A, B \in \BC(\XC)$ and $S, T \in \BC(\TC)$ with $\ell(T)>0$. Note that, for the reduced measure~$\breve M_{[2]}$, the time component represents the time difference between points. Hence, including zero is important: in the non-simple case, the reduced measure captures the interaction of points occurring at the same timestamp. Without the zero, the measure is not fully described.
\end{remark}
We refer to $\breve{M}_{[2]}$ as the \emph{time-reduced second factorial moment measure}. We may perform the same reduction for the factorial covariance measure $\gamma_{[2]}$ to obtain the \emph{time-reduced factorial covariance measure} $\breve{\gamma}_{[2]}$ defined on $\XC^2 \times \TC$, which satisfies
\begin{equation} \label{eq:breve_gamma_def}
\breve{\gamma}_{[2]}(A\times B \times S) = \breve{M}_{[2]}(A\times B \times S)  - \mu (A) \, \mu(B) \, \ell(S)
\end{equation}
for all $A,B \in \BC(\XC)$ and $S \in \BC(\TC)$.
This is seen by taking integrals $\int_{[0,1]} \, dr$ on both sides in \eqref{eq:red2facmom} and using the definition of $\gamma_{[2]}$. We may interpret $\breve{\gamma}_{[2]}(A \times B \times dh)$ as the covariance of the point counts in spatial regions $A$ and $B$ at time points that are $h>0$ apart from on one another ($h=0$ is special due to factorization).

Time-reduced factorial covariance measures are a convenient tool because they describe some intuitive and interpretable quantities and are routinely computed at least for purely spatial point processes.

Note that the measures defined above are consistent with the corresponding measures for temporal processes whenever we fix a single $A \in \BC(\XC)$. In particular, the time-reduced factorial covariance measure $\breve{\gamma}_{[2],A}$ of the temporal point process $\Xi(A \times \cdot)$ is given by
\begin{equation}
  \breve{\gamma}_{[2],A}(\cdot) = \breve{\gamma}_{[2]}(A \times A \times \cdot). 
\end{equation}
By the symmetry stated in \Cref{prop:time-reduction}  and \eqref{eq:breve_gamma_def}, we have $\breve{\gamma}_{[2],A}(-S) = \breve{\gamma}_{[2],A}(S)$ for $S \in \BC(\TC)$.
The analog result holds for the reduced (raw) covariance measure $\breve{\gamma}_{2,A}$.

The Jordan decomposition $\breve{\gamma}_{[2],A} = \breve{\gamma}^+_{[2],A} - \breve{\gamma}^-_{[2],A}$ splits up $\breve{\gamma}_{[2],A}$ into (proper non-negative) measures that are mutually singular. Thus, $\breve{\gamma}^+_{[2],A}$ and $\breve{\gamma}^-_{[2],A}$ are concentrated on time-lags at which the point counts in $A$ are positively correlated and negatively correlated, respectively. We sometimes also use the Jordan decomposition of the full measure $\breve \gamma_{[2]} = \breve \gamma_{[2]}^+ - \breve \gamma_{[2]}^-$. We have
\begin{equation}\label{eq:relation_gamma_restrictedAndUnrestricted}
    \breve \gamma_{[2],A}^{+}(\,\cdot\,) \leq \breve \gamma_{[2]}^+(A \times A \times \,\cdot\,)
\end{equation}
for every $A \in \BC(\XC)$ because
\begin{equation*}
    \breve \gamma_{[2], A}^+([0,t]) = \sup_{B \subset [0,t]} \breve \gamma_{[2]}(A^2 \times B) \leq \sup_{\overline{B} \subset A^2 \times [0,t]} \breve \gamma_{[2]}(\overline{B}) =
    \breve \gamma_{[2]}^+(A^2 \times [0,t]),
\end{equation*}
where the suprema are taken over Borel measurable subsets. 

For our main results in \Cref{sec:main} we require a handle on $\Var(\Xi(A \times (0,t]))$ for $A \in \BC(\XC)$, which we can express by means of $\breve{\gamma}_{[2],A}$ in the following result.

\begin{proposition} \label{prop:Xi_Var_vs_Cov}
Let $\Xi$ be a weakly time-stationary $L^2$ point process. Then we have, for any $A \in \BC(\XC)$ and $t \in \tgz$,
\begin{align*}
  \Var(\Xi(A\times (0,t])) &= t \breve{\gamma}_{2,A}(\{0\}) + 2 \int_{(0,t)} (t-s) \; \breve{\gamma}_{2,A}(ds) \\
  &= t \mu(A) + t \breve{\gamma}_{[2],A}(\{0\}) + 2 \int_{(0,t)} (t-s) \; \breve{\gamma}_{[2],A}(ds).
\end{align*} Here, $\breve{\gamma}_{[2],A}(\{0\}) = 0$ if $\Xi(A \times \cdot)$ is simple.
\end{proposition}

\subsection{Controlling positive covariances} \label{sec:pp_assumptions}

Throughout the remainder of this section, we assume that $\Xi$ is a weakly time-stationary $L^2$ point process and analyze the relation between the time-reduced factorial covariance measure $\breve \gamma_{[2]}$ and our partitioned variance growth assumption \ref{ass:Xi_mixing_Var_beta}. The latter serves as a crucial element for the derivation of quantitative error bounds for the statistical error of the empirical KRD induced by spatio-temporal time-homogeneous point processes. In the following, we introduce for $\beta \in[0,1)$ an assumption in terms of the positive parts $\breve \gamma^{+}_{[2],A}$.

\begin{description}
\labitem{(RCov$_\beta$)}{ass:Xi_mixing_cov_beta}  
There are constants\footnote{In general, $\kappa' \neq \kappa$ from \ref{ass:Xi_mixing_Var_beta}. For the relation between the two, see \Cref{prop:PVar_vs_RCov} below.} $\kappa' \geq 0$ and $t_0 \in \tgz$ such that for any partition $(A_i)_{i=1,\ldots,N}$ of $\XC$ into measurable sets and any $t \in \tgz$, $t \geq t_0$, the positive parts $\breve \gamma_{[2], A_i}^+$ of the time-reduced factorial covariance measures $\breve \gamma_{[2],A_i}$ of $\Xi$ fulfill 
\begin{equation*}
	\sum_{i=1}^N \breve \gamma_{[2], A_i}^+([0,t]) \leq \kappa'  t^{\beta}.
\end{equation*}	
\end{description}

\begin{remark} \label{rem:RCov_XC}
\begin{enumerate}
	\item[$(i)$] Assumption \ref{ass:Xi_mixing_cov_beta} is satisfied if for all $t\in \TC_{>0}$, $t\geq t_0$ 
  \begin{equation} \label{eq:full_poscovmeasure}
	\breve \gamma_{[2]}^+(\XC^2 \times [0,t]) \leq \kappa'  t^{\beta}.
  \end{equation}	
  Indeed, from \eqref{eq:relation_gamma_restrictedAndUnrestricted} and since $\breve{\gamma}_{[2]}^+$ is a (non-negative) measure on $\XC^2\times \TC$, we have 
  \begin{equation*}
	\sum_{i=1}^N \breve \gamma_{[2], A_i}^+([0,t]) 
    \leq \sum_{i=1}^N \breve \gamma_{[2]}^+(A_i^2 \times [0,t]) 
    \leq \breve \gamma_{[2]}^+(\XC^2 \times [0,t]).
  \end{equation*}	
\item[$(ii)$] We say that $\Xi$ has \emph{(only) non-negative correlations} if $\breve \gamma_{[2]}$ is a (non-negative) measure and thus $\breve \gamma_{[2]} = \breve \gamma_{[2]}^+$.
In this case, Assumption \ref{ass:Xi_mixing_cov_beta} is satisfied if 
\begin{equation} \label{eq:full_poscovmeasure_nonneg}
	\breve \gamma_{[2],\XC}^{+}([0,t]) \leq \kappa' t^{\beta}.
  \end{equation}	
 This follows from part $(i)$ since
  \begin{equation*}
    \breve \gamma_{[2]}^+(\XC^2 \times [0,t]) = \breve \gamma_{[2],\XC}([0,t]).
  \end{equation*}
\end{enumerate}
\end{remark}

Based on \Cref{prop:Xi_Var_vs_Cov}, we will in the following show that Assumptions \ref{ass:Xi_mixing_cov_beta} and  \ref{ass:Xi_mixing_Var_beta} are strongly related to each other. 

\begin{proposition} \label{prop:PVar_vs_RCov}
  Let $\Xi$ be a weakly time-stationary $L^2$ point process. Then, the following implications hold for $\beta \in [0,1)$. %
  \begin{enumerate}
    \item[(i)] If \ref{ass:Xi_mixing_cov_beta} is met for $\kappa' \geq 0$ and $t_0\in \TC_{>0}$, it follows
    for any finite partition $(A_i)_{i=1}^N$ of $\XC$ and $t\in \TC_{>0}$, $t\geq t_0$ that 
    $$
    	\sum_{i=1}^N \Var(\Xi(A_i \times (0,t])) - t m_\mu \leq 2\kappa' t^{1+\beta},
    $$
    and Assumption \ref{ass:Xi_mixing_Var_beta} is met with $\kappa = 2\kappa'$.
    \item[(ii)] Assume that $\Xi$ has non-negative correlations. Then, if \ref{ass:Xi_mixing_Var_beta} is met for $\kappa \geq 0$ and $t_0 \in \TC_{>0}$, it follows for all $t\in \TC_{>0}$, $t\geq t_0/2$ that
    $$
        \breve \gamma_{[2],\XC}^+([0,t]) \leq 2^\beta \kappa t^\beta,
    $$
  and Assumption \ref{ass:Xi_mixing_cov_beta} is met with $\kappa' = 2^\beta \kappa$ by~\Cref{rem:RCov_XC}(ii).
\end{enumerate}
\end{proposition}
Total finiteness of some form of covariance measure, such as \eqref{eq:full_poscovmeasure}
with $\beta=0$, is a common condition in limit theorems for point processes (e.g., \citealp{ivanoff1982central,pawlas2009empirical,heinrich2010estimating,ross2016wireless,blaszczyszyn2019limit}). %
The present paper is concerned with bounding the expectation of a point process functional; for bounds and asymptotics on higher-order moments or full distributions, one typically needs finite total variation of higher-order cumulant measures %
or stronger mixing conditions. The appropriate mixing condition that goes well with finite cumulant measures is \emph{Brillinger mixing} \citep{brillinger1975statistical}, which is well understood in relation to other mixing conditions \citep{heinrich2013absolute,heinrich2018brillinger}. 

Assumption \ref{ass:Xi_mixing_cov_beta} can also be interpreted as follows: when $\beta=0$, the dependence structure is sufficiently weak that the effect of a single point in $\Xi$ is essentially local in time, that is, one point can only \emph{induce} a bounded number of additional points in the future. This ensures convergence rates much like in the classical i.i.d.\ setting. By contrast, when $\beta>0$, a single point may influence the system over arbitrarily long time scales, generating correlations that decay only polynomially. 

Combining \Cref{prop:PVar_vs_RCov} with \Cref{th:est_mu_inUOT}, we see that we achieve the same convergence rates for $\EE[\KR_{p,C}(\mun, \mu)]$ under \ref{ass:Xi_mixing_cov_beta} with $\beta=0$ as for the Poisson process. 
Letting $\beta \geq 0$ in \ref{ass:Xi_mixing_cov_beta} allows the covariance measure to grow with the length of the time-interval to infinity. In other words, if $\breve \gamma_{[2], A}^+([0,t])$ has a density, \ref{ass:Xi_mixing_cov_beta} for $\beta > 0$ means that covariances decay (on average) of order $t^{-(1-\beta)}$ as the time lag $t$ gets large. This relaxation incurs a cost on the convergence rate of the empirical KRD, as seen by combining \Cref{prop:PVar_vs_RCov} with \Cref{th:est_mu_inUOT_beta}.

We note that Assumptions~\ref{ass:Xi_mixing_Var_beta}  and \ref{ass:Xi_mixing_cov_beta} are weak in the sense that they allow for the ``almost'' maximum possible growth of a partitioned variance (and covariance measure) according to the following result.
\begin{proposition} \label{prop:max_cov_var_rate}
    For any weakly time-stationary $L^2$ point process $\Xi$ there exists a positive constant $\kappa>0$ such that for any finite partition $(A_i)_{i=1}^N$ of $\XC$ and $t\in \TC_{>0}, \, t > 1$ it holds that 
    $$
    	\sum_{i=1}^N \Var(\Xi(A_i\times (0,t])) \leq \kappa t^2 \quad \text{ and } \quad  \sum_{i =1}^{N}\gamma^+_{[2], A_i}([0,t])\leq  \breve \gamma^+_{[2]}(\XC^2 \times [0,t]) \leq \kappa t.
    $$
    \end{proposition} 

\subsection{Examples of point processes} \label{sec:pp_examples}
In this section, we discuss various spatio-temporal point processes which satisfy Assumption \ref{ass:Xi_mixing_cov_beta}.
We refer to \cite{moller2003statistical, illian2008statistical, diggle2013statistical} for comprehensive treatments of the statistical aspects of spatial and spatio-temporal point processes and to \cite{gonzalez2016spatio} %
 for a recent overview of spatio-temporal point process models.

\subsubsection{Poisson point process} \label{sec:example_PoissonPP}
As a prime example in continuous time (see, e.g., \citealp{last2018lectures, kutoyants2023introduction}),
we consider the time-homogeneous Poisson point process  $\Pi$ on $\XC \times (0, \infty)$ with intensity measure $\mu \otimes \mathrm{Leb}, \, \mu \in \MCsX$ (we also use notation $\Pi \sim \pop(\mu \otimes \Leb)$). To construct such a process, let $\mu\in \MCsX$ with $m_{\mu}>0$, and define $\Pi = \sum_{i=1}^\infty\delta_{(X_i, T_i)}$, where $X_i$ are i.i.d.\ with distribution $ \mu/m_{\mu}$ and $T_i \coloneqq \sum_{j = 1}^{i} \tau_j $ with i.i.d.\ random variables $\tau_1, \tau_2, \ldots \sim \mathrm{Exp}(m_{\mu})$ that are also independent of all $X_i$. 

Poisson processes are characterized by the property that, for bounded disjoint sets $B_1, \dots, B_k\in \BC(\XC\times (0, \infty))$, the random variables $\Xi(B_1), \dots, \Xi(B_k)$ are independent. 
It is easily checked that the Poisson point process satisfies $M_{[k]} = M_1^k$ for all natural $k$. Thus, for all $A \in \BC(\XC)$ we have $M_{[2],A}(d t \times d s) := M_{[2]}(A \times A \times d t \times d s) = (\mu(A))^2 \, \Leb^2(d t \dif s)$. By \eqref{eq:rfmm_explicit}, we obtain $\breve M_{[2],A}([0, t]) = (\mu(A))^2 t$, whence $\breve \gamma_{[2],A}([0, t]) = \breve M_{[2],A}([0, t]) - \mu(A) \mu(A) t = 0$. Thus, \ref{ass:Xi_mixing_cov_beta} holds with $\beta=0$.

\subsubsection{Binomial process} \label{sec:Bimomial}
The most basic example of time-discrete ST point process with $\tgz = \NN$ is perhaps the \emph{binomial} point process (see \citealp[Chapter~15]{kallenberg2021} or \citealp{munk2020statistical}, where it is called multinomial for microscopic applications), in which points are placed individually for each $i \in \NN$ at independent spatial locations $X_i \sim \mu$ for a probability measure $\mu \in \PC(X)$, so that $\Xi = \sum_{i=1}^\infty \delta_{(X_i,i)}$. Hence, this describes the classical i.i.d.\ sampling model. Note that $M_{1,A}$ and $M_{[2],A}$ are defined on subsets of $\NN$ and $\NN^{2}$, respectively. A straightforward calculation yields for any $A \in \BC(\XC)$ that
\begin{align*}
	M_1(A \times \{t\}) &= \EE [\Xi(A \times \{t\})] = \PP(X_t \in A) = \mu(A), \\
	M_{[2], A}(\{u,v\}) &= \EE [\Xi^{[2]}(A^2 \times \{u,v\})] = \begin{cases}
          \PP(X_u \in A) \PP(X_v \in A) = \mu(A)^2 &\text{if $u \neq v$}, \\
          0  &\text{if $u = v$},
       \end{cases} 
\end{align*}
where the last line holds because any two points of $\Xi$ have different time coordinates.
It therefore follows that 
$$
  \gamma_{[2],A}(\{u,v\}) = M_{[2], A}(\{u,v\}) - M_1(A \times \{u\})M_1(A \times \{v\}) = \begin{cases}
     0 &\text{if $u \neq v$}, \\
     -(\mu(A))^2 &\text{if $u = v$}.
  \end{cases}
$$
The negative covariance in the last line quantifies the inhibition of points at a single time point. Noting that $\breve{\gamma}_{[2],A}(\{h\}) = %
\gamma_{[2],A}(\{u,u+h\})$ for any $h \in \NN_0$ and arbitrary $u \in \NN$ by~\eqref{eq:rfmm_explicit} and~\eqref{eq:breve_gamma_def}, we obtain
$$
  \breve{\gamma}_{[2],A}([0,t]) = \sum_{h=0}^t \gamma_{[2],A}(\{1,1+h\}) = -(\mu(A))^2
$$
for any $t \in \NN_0$. We conclude that $\breve \gamma_{[2],A}= -\mu(A)^2 \delta_0$ and hence $\breve \gamma^+_{[2],A} = 0$, which implies the validity of \ref{ass:Xi_mixing_cov_beta} for $\beta=0$.

\subsubsection{Time-inhibitory processes}

The previous two examples were point processes with strong independence properties satisfying $\breve{\gamma}^{+}_{[2],A} = 0$ for every $A \in \BC(\XC)$.\footnote{In fact even $\breve{\gamma}_{[2]} \leq 0$, but for notational convenience we show statements about our examples typically for $\breve{\gamma}_{[2]}(A^2 \times \cdot\,)$ rather than $\breve{\gamma}_{[2]}(A \times B \times \cdot\,)$, where $A,B \in \BC(\XC)$.}
It is immediately clear that any process with \emph{overall non-positive time correlations} in the sense that $\breve{\gamma}_{[2],A} \leq 0$ for every $A \in \BC(\XC)$ (as a signed measure) satisfies \ref{ass:Xi_mixing_cov_beta} with $\beta=0$ since in this case $\breve{\gamma}^+_{[2],A} = 0$ holds as well.
A popular example is the \emph{determinantal point process}, which satisfies even $\breve{\gamma}_{[2]} \leq 0$ (\citealp{lavancier2015determinantal,vafaei2023spatio} for the spatio-temporal case).

At first glance, it may appear intuitive that concrete point process models implementing purely inhibitory mechanisms between points should satisfy Assumption~\ref{ass:Xi_mixing_cov_beta} as well. While this is \emph{usually} the case, it should be noted that such point processes may well have non-trivial measures $\breve{\gamma}^+_{[2],A}$ and there is, in fact, no general guarantee that~\ref{ass:Xi_mixing_cov_beta} holds. We give two special cases here for illustration, but concentrate on point processes with mostly excitatory mechanisms (``clustering'') in the remainder of Subsection~\ref{sec:pp_examples}. In order to be concise, we present these special cases for a purely temporal setting, i.e., set $\XC = \{0\}$.

First, consider the temporal \emph{Mat{\'e}rn type~I inhibition process}, which is obtained by starting from a homogeneous Poisson process $\Pi$ on $\XC\times \RR$ with intensity $\lambda>0$ and deleting every pair of points that are within distance $R>0$ of one another. It is straightforward to check (see \citealp[Example~8.1(c)]{daley2003introduction})
that the reduced factorial covariance measure has density
\begin{equation*}
  \breve{g}_{[2]}(r) = \begin{cases}
    -\lambda^2 e^{-4\lambda R} &\text{if $r \leq R$}; \\
    \lambda^2 e^{-2\lambda R} (e^{-\lambda r} - e^{-2\lambda R}) &\text{if $R < r \leq 2R$}; \\
    0 &\text{if $2R < r$}.
  \end{cases}
\end{equation*}
The density is positive between $R$ and $2R$ and thus $\breve{\gamma}^{+}_{[2],\XC}$ is non-trivial, which reflects the intuitive fact that being just beyond the taboo zone of a point makes it more likely to see any points because we know there can be no other points within the taboo zone. While clearly $\breve{\gamma}^{+}_{[2],\XC}(\RR_+)$ is always finite, i.e., \ref{ass:Xi_mixing_cov_beta} is satisfied here with $\beta=0$, even $\breve{\gamma}_{[2],\XC}(\RR_+) = \lambda (e^{-3\lambda R} - (2\lambda R+1) e^{-4 \lambda R})$ (not only the positive part) can get arbitrarily large as $\lambda \to \infty$ while $1.3 \leq \lambda R < \infty$. %

For the second special case, we start from $\Pi$ and delete all the points if one point is deleted, which shall happen with probability $p \in (0,1)$, independently of $\Pi$. Based on \Cref{sec:example_PoissonPP}, the reduced factorial covariance measure then has constant density $\breve{g}_{[2]}(r) = p \lambda^2 - p^2 \lambda^2 = p(1-p) \lambda^2 =: \alpha > 0$, whence $\breve{\gamma}^+_{[2],\XC}([0,t]) = \breve{\gamma}_{[2],\XC}([0,t]) = \alpha t$, so that Assumption~\ref{ass:Xi_mixing_cov_beta} is \emph{not} satisfied in spite of the inhibitory mechanism between the points.

\subsubsection{Spatio-temporal Poisson cluster processes} \label{sec:cluster_ppp}

Poisson cluster processes are superpositions of independent ``offspring'' point clusters positioned according to a Poisson process of (invisible) ``parent'' points. They form an important class comprising a number of well-established spatio-temporal point process models as special cases, including Neyman--Scott processes, Hawkes processes (both of which we describe in the sequel) and shot-noise Cox processes \citep{brix2002spatio}. 
As a special case of more general Poisson cluster processes on metric spaces (without an explicit time-component), spatio-temporal Poisson cluster processes have only non-negative correlations \citep[Equation~(6.3.12)]{daley2003introduction}.

To keep the greatest possible flexibility, we assume here that $\XC$ is a bounded measurable subset of a general c.s.m.s.\ $\XCbar$ and consider a Poisson process of parent points on $\XCbar \times \RR$ that is homogeneous in time but has only finitely many points in every bounded subset of $\XCbar$ in a finite time interval, i.e., has an expectation measure of the form $\Lambda \otimes \Leb$ for some \emph{locally} finite measure $\Lambda$ on $\XCbar$. Besides $\XCbar = \XC$, our main application case is $\XCbar = \RR^d$ equipped with the Euclidean metric (typically with $\Lambda$ a multiple of $\Leb_{\RR^d}$). %

Given a parent point at $(y,s) \in \XCbar \times \RR$, its offspring cluster is generated according to the distribution $P^{cl}(\cdot \mvert (y,s))$. Here $P^{cl}(\cdot \mvert \cdot)$ is a probability kernel from $\XCbar \times \RR$ to $\mathcal{N}(\XCbar \times \RR)$ that is \emph{time-shift equivariant} in the sense that 
\begin{equation} \label{eq:timehomo_kernel}
    P^{cl}(D \mvert (y,s)) = P^{cl}(T^{-1}_{s}(D) \mvert (y,0)),
\end{equation}
where $T_s \colon \mathfrak{N}(\XCbar \times \mathbb{R}) \to \mathfrak{N}(\XCbar \times \mathbb{R})$, $\xi \mapsto \xi(\cdot - (0,s))$, and \emph{causal} in the sense that 
\begin{equation} \label{eq:causal}
    P^{cl} \bigl( \{\xi \mvert \xi(\XCbar \times (-\infty,0)) \geq 1 \} \mvert (y,0) \bigr) = 0.
\end{equation}
Assume furthermore that $P^{cl}(\cdot \mvert (y,s))$ has (locally finite) first and second moment measures, which we denote by $M^{cl}_1(\cdot \mvert (y,s))$ and $M^{cl}_2(\cdot \mvert (y,s))$, respectively.

\begin{definition}
  Let $\Eta = \sum_{i=1}^{\infty} \delta_{(Y_i,S_i)}$ be a time-homogeneous Poisson process on $\XCbar \times \RR$ with expectation measure $\Lambda \otimes \mathrm{Leb}$ for some locally finite measure $\Lambda$ on $\XCbar$. Let furthermore $\bar{\Xi}_i \sim P^{cl}(\cdot \mvert (Y_i,S_i))$, $i \in \NN$, be independent.
  If $\bar{\Xi} = \sum_{i=1}^{\infty} \bar{\Xi}_i$ is locally finite and hence a point process, it is called a \emph{spatio-temporal Poisson cluster process}. 
\end{definition}
 In what follows, we consider the process $\Xi = \bar{\Xi}\vert_{\XC \times (0, \infty)} = \sum_{i=1}^{\infty} \Xi_i$, where $\Xi_i = \bar{\Xi}_i\vert_{\XC \times (0, \infty)}$, which is a Poisson cluster process in its own right. We give a general characterization for $\Xi$ being a well-defined stationary $L^2$ process, and an expression for the measures $\breve{\gamma}_{[2],A}$ which allows us to assess Assumption~\ref{ass:Xi_mixing_cov_beta}.
 Subsequently, we discuss two popular special cases of ST Poisson cluster processes, viz.\ the ST Neyman--Scott process (\Cref{sec:Neyman--Scott}) and the ST Hawkes process (\Cref{sec:Hawkes}).
\begin{proposition}[Expectation and covariance of an ST Poisson cluster process]
\label{prop:E_and_Var_for_PCP}
  The spatio-temporal Poisson cluster process $\Xi$ is well-defined, (strongly) time-stationary and $L^2$ if and only if, for every $A \in \BC(\XC)$,
  \begin{align}
    \int_{\XCbar} M^{cl}_1(A \times \RR_+ \mvert (y,0)) \; \Lambda(dy) &< \infty  \label{eq:L1_cond_poisson_cluster},\\
    \int_{\XCbar} M^{cl}_{[2],A}\bigl(\varphi^{-1}(\RR_+ \times [0,t]) \bigm| (y,0) \bigr) \; \Lambda(dy) &< \infty \quad \text{for every $t > 0$,} \label{eq:L2_cond_poisson_cluster}
  \end{align}
  where $\varphi \colon \RR^2 \to \RR^2$, $(r,s) \mapsto (r,s-r)$.
  Moreover, we have
  \begin{align}
    M_1(A \times (0, t]) &= t \, \mu(A) = t \int_{\XCbar} M^{cl}_1(A \times \RR_+ \mvert (y,0)) \; \Lambda(dy), \label{eq:mom1_poisson_cluster}\\ 
    \breve{\gamma}_{[2],A}([0,t]) = \breve{\gamma}^+_{[2],A}([0,t])
      &= \int_{\XCbar} M^{cl}_{[2],A}\bigl(\varphi^{-1}(\RR_+ \times [0,t]) \bigm| (y,0) \bigr) \; \Lambda(dy).
      \label{eq:rf_covm_poisson_cluster}
  \end{align}
\end{proposition}

For any partition $(A_i)_{i =1}^{N}$ of $\XC$ into measurable sets and every $t>0$ we have
\begin{align*}
    \sum_{i = 1}^{N}\breve{\gamma}^+_{[2],A_i}([0,t])
      &= \int_{\XCbar} \sum_{i = 1}^{N} M^{cl}_{[2],A_i}\bigl(\varphi^{-1}(\RR_+ \times [0,t]) \bigm| (y,0) \bigr) \; \Lambda(dy) \\
    &\leq \int_{\XCbar} M^{cl}_{[2],\XC}\bigl(\varphi^{-1}(\RR_+ \times [0,t]) \bigm| (y,0) \bigr) \; \Lambda(dy) = \breve{\gamma}^+_{[2],\XC}([0,t])
\end{align*}
since $M^{cl}_{[2]}\bigl(\,\cdot\, \times \varphi^{-1}(\RR_+ \times [0,t]) \bigm| (y,0) \bigr)$ are (non-negative) measures and $A_i^2$ are disjoint subsets of $\XC^2$.
It follows therefore that Assumption~\ref{ass:Xi_mixing_cov_beta} is satisfied for general $\beta \in [0,1)$ if there are $\kappa' \geq 0$ and $t_0 \in \RR_+$ such that 
  \begin{equation} \label{eq:cov_beta_poisson_cluster}
    \int_{\XCbar} M^{cl}_{[2],\XC}\bigl(\varphi^{-1}(\RR_+ \times [0,t]) \bigm| (y,0) \bigr) \; \Lambda(dy) \leq \kappa' t^{\beta} \quad \text{for all }t \geq t_0.
  \end{equation}
  For $\beta=0$, it suffices to ensure that
  \begin{equation} \label{eq:cov_zero_poisson_cluster}
    \int_{\XCbar} M^{cl}_{[2],\XC}\bigl(\RR_+ \times \RR_+ \bigm| (y,0) \bigr) \; \Lambda(dy) < \infty.
  \end{equation}

\begin{remark}\label{rmk:discussionPoissonCluster}
We can obtain $\breve{\gamma}_{[2],\XC}^+([0,t]) \asymp t^{\beta}$ for any $\beta \in [0,1)$ with a Poisson cluster process. To see this, let $\XC$ be arbitrary and define the distribution of a cluster for parent point $(y,0)$ as follows: let its number of points $N$ have probability mass function $p(n) \asymp n^{-(3-\beta)}$ for $n \to \infty$ and position these points at locations $\{(X_1,1),(X_2,2),\ldots,(X_N,N)\}$, where the joint distribution of the $X_i$ is arbitrary. %
Note that $M_{[2],\XC}^{cl}$ is the expectation measure of $\Zeta = \sum_{i,j=1}^{N, \neq} \delta_{(i,j)}$ here. We may assume without loss of generality (w.l.o.g.) that $\PP(N \geq 2)=1$ (otherwise condition on $N \geq 2$) and let $t \geq 2 =: t_0$. We have $\Zeta(\varphi^{-1}(\RR_+ \times [0,t])) = N(N-1)/2$ on $\{t \geq N\}$ and
\begin{equation*}
  N t / 4 \leq N \lfloor t \rfloor / 2 \leq \Zeta(\varphi^{-1}(\RR_+ \times [0,t])) \leq N t
\end{equation*}
on $\{t < N\}$.
Thus,
    \begin{align*}
      \breve{\gamma}_{[2],\XC}([0,t]) &\asymp \EE[N \min\{N,t\}] \\
      &= \EE[N^2 \mathds{1}\{N \leq t\}] + t \, \EE[N \mathds{1}\{N>t\}] \\
      &\asymp \int_1^t s^2 s^{-(3-\beta)} \; ds + t \int_t^{\infty} s \, s^{-(3-\beta)} \; ds \\
      &\asymp \frac{1}{\beta} t^{\beta} + \frac{1}{1-\beta} t^{\beta} \asymp t^{\beta},
    \end{align*}
where the third line is obtained by sandwiching sums with integrals.
\end{remark}

\subsubsection{Neyman--Scott process} \label{sec:Neyman--Scott}

To further illustrate \Cref{prop:E_and_Var_for_PCP}, consider a popular special case of a Neyman--Scott process \citet[Section 5.3]{moller2003statistical}, which is a Poisson cluster process where the cluster center process $\Eta$ is spatially homogeneous on $\XCbar = \RR^d$, so that $\Lambda = a \, \Leb_{\RR^d}$ for some $a > 0$, and the individual clusters are themselves independent Poisson processes $\bar{\Xi}_i$ with finite (temporally and spatially) equivariant intensity measure $M_1^{cl}(\cdot \mvert (y,s)) = M_1^{cl}(\cdot - (y,s) \mvert (0,0)) =: M_1^{cl}(\cdot - (y,s))$ on $\RR^d \times \RR_+$. 
Originating from cosmology \citep{neyman1958statistical}, Neymann-Scott models are used in meteorology \citep{cowpertwait2002space} and seismology \citep{de2016statistical}, among other fields.
We compute
\begin{align*}
  \int_{\XCbar} M^{cl}_1(A \times \RR_+ \mvert (y,0)) \; \Lambda(dy)
  &= \int_{\RR^d} \int_{\RR^d} \mathds{1}_{A-y}(x) \; M_1^{cl}(dx \times \RR_+) \; dy \\
  &= \int_{\RR^d} \int_{\RR^d} \mathds{1}_{A-x}(y) \; dy \; M_1^{cl}(dx \times \RR_+)  \\
  &= M^{cl}_1(\RR^d \times \RR_+) \, \Leb(A).
\end{align*}
Furthermore, since the clusters are Poisson processes, we have
\begin{align*}
  \int_{\XCbar} &M^{cl}_{[2],\XC}\bigl(\RR_+ \times \RR_+ \bigm| (y,0) \bigr) \; \Lambda(dy) \\
  &= \int_{\RR^d} \int_{\RR^d \times \RR^d} \mathds{1}_{\XC-y}(x) \mathds{1}_{\XC-y}(x') \; M_1^{cl}(dx \times \RR_+) \; M_1^{cl}(dx' \times \RR_+) \; dy \\
  &= \int_{\RR^d \times \RR^d} \int_{\RR^d} \mathds{1}_{\XC-x}(y) \mathds{1}_{\XC-x'}(y) \; dy \; M_1^{cl}(dx \times \RR_+) \; M_1^{cl}(dx' \times \RR_+) \\[1mm]
  &\leq M^{cl}_1(\RR^d \times \RR_+)^2 \, \Leb(\XC) < \infty.
\end{align*}

Therefore, \eqref{eq:L1_cond_poisson_cluster}, \eqref{eq:L2_cond_poisson_cluster} and \eqref{eq:cov_zero_poisson_cluster} are all satisfied and thus the Neyman--Scott process defined above is a (even strongly) time-stationary $L^2$ process satisfying Assumption~\ref{ass:Xi_mixing_cov_beta} with $\beta=0$. Note that we have $\mu = M_1(\cdot \times (0,1]) = M^{cl}_1(\RR^d \times \RR_+) \, \Leb_{\RR^d}\vert_{\XC}$ here due the spatial homogeneity of $\Eta$ and the spatial equivariance of the cluster distribution.

\subsubsection{Spatio-temporal Hawkes process} \label{sec:Hawkes}

A general spatio-temporal point process $\Xi$ on $\XC \times (0, \infty)$ can be defined (under assumptions which we outline below) via its conditional intensity $\lambda(x,t \mvert \xi_t)$, describing the (normalized) expected number of events occurring in an infinitesimal space-time interval $[x,x+dx)\times[t,t+dt)$ given that the history
of $\Xi$ in $\XC$ before time $t$ is $\xi_t$. If the conditional intensity is of the form
\begin{equation}\label{eq:Hawkes_density}
  \lambda(x,t \mvert \xi_t) = \lambda(x) + \int_{\XC} \int_0^t g(y,x,t-s) \; \xi_t(d y, \dif s),
\end{equation}
for measurable $\lambda \colon \XC \to \RR_+$ with $m_\lambda = \int_{\XC} \lambda(x) \, dx < \infty$ and $g \colon \XC^2 \times \RR \to \RR_+$ with $g(y,x,h)=0$ for $h<0$ and $\nubar = \sup_{y \in \XC} \nu(y) < 1$ for $\nu(y) = \int_{\XC \times \RR_+} g(y,x,h) \dif x \dif h$, we refer to $\Xi$ as \emph{spatio-temporal Hawkes process} (see \citealp{hawkes1971spectra} for the original purely temporal process) with background intensity $\lambda$ and excitation (or infection) intensity $g$. Compared to the usual definition of a spatio-temporal Hawkes process \citep{musmeci1992space,daley2003introduction}, where $\lambda(x)$ is constant and $g(y,x,t-s) = g(x-y,t-s)$, we only require temporal but not spatial stationarity.
The general intuition from \eqref{eq:Hawkes_density} is that ``immigrants'' (in a wide abstract sense) arrive at a constant rate $m_\lambda$, each immigrant being distributed in space according to the density $\lambda/m_\lambda$, and that each former point $(y,s)$ (immigrant or otherwise) produces on average $\nu(y)$ additional direct offspring points according to the spatial-temporal density $(x,t) \mapsto g(y,x,t-s)/\nu(y)$. Due to this cascading generation of offspring points, Hawkes processes are often referred to as \emph{self-exciting}.
Applications include meteorological, seismological, criminological, epidemic and plant ecological data (e.g., \citealp{ogata1998earthquake, mohler2011crime, meyer2012space}) and, more recently, neural \citep{zuo2020transformer,mei2017neural,chen2020neural} and social \citep{zhao2015seismic} networks. 

Following the arguments in \cite{hawkes1974cluster} (temporal case) and \cite{musmeci1992space} for our somewhat more general setting, a time-stationary Hawkes process satisfying \eqref{eq:Hawkes_density} %
can be constructed as a Poisson cluster process having a branching structure as follows.
Let $\Zeta_0 = \sum_{i=1}^{N_0} \delta_{(X_{0i},T_{0i})}$ be (for now) a general initial point process on $\XC \times \RR$. Given the $k$-th generation is $\Zeta_k = \sum_{i=1}^{N_k} \delta_{(X_{ki},T_{ki})}$, let $\Psi_{ki} \sim \pop(g(X_{ki},x,t-T_{ki}) \, dx \, dt)$, $1\leq i \leq N_k$, be independent (also of all previous processes) and define the next generation by $\Zeta_{k+1} = \sum_{i=1}^{N_k} \Psi_{ki}$. Assuming now that $\Zeta_0 = \Eta = \sum_{i=1}^{N_0} \delta_{(X_{0i},T_{0i})} \sim \pop(\lambda(x) \, dx \, dt)$, we obtain a process with the required spatio-temporal Hawkes distribution as $\Xi = \sum_{k=0}^{\infty} \Zeta_k$. A Poisson cluster process representation is obtained by performing the same construction independently for each $\delta_{(X_{0i},T_{0i})} =: Z^{(i)}_0$ in the $\pop(\lambda(x) \, dx \, dt)$ process. Setting $\Xi_i = \sum_{k=0} Z^{(i)}_k$, where $Z^{(i)}_k$ is the $k$-th generation obtained from $Z^{(i)}_0$, we have $\Xi = \sum_{i=1}^{N_0} \Xi_i$. 

We only discuss the case where the total excitation over $\XC$ is dominated uniformly in $y \in \XC$ by a temporal excitation function.
\begin{description}
\labitem{(Tdom)}{ass:subprob_temporal_domination}
There is a measurable function $\gbar_{\temp} \colon \RR_+ \to \RR_+$ so that $\nubar_{\temp} = \int_{\RR_+} \gbar_{\temp}(h) \dif h < 1$ and for all $h \geq 0$ it holds 
\begin{equation} \label{eq:stsep_majorant}
  \sup_{y \in \XC} \int_{\XC} g(y,x,h) \; dx \leq \gbar_{\temp}(h).
\end{equation}

\end{description}
This assumption implies that the temporal marginal of the ST Hawkes process is dominated by a purely temporal stationary Hawkes process with (constant) background intensity $m_{\lambda}$ and excitation function $\gbar_{\temp}(h)$; see the proof of \Cref{prop:hawkes_moments} below.
This is especially the case in the popular (but rather idealistic) setting that $g$ is space-time separable, i.e., of the form $g(y,x,h) = g_{\spat}(y,x) \, g_{\temp}(h)$, as can be seen by setting $\gbar_{\temp} = \bigl(\sup_{y \in \XC} \int_{\XC} g_{\spat}(y,x) \, dx)\bigr) \cdot g_{\temp}$.

Write generally $(f_1 \ast f_2)(y,x,t) = \int_{\XC} \int_{\RR} f_1(y,x',t') \, f_2(x',x,t-t') \, dx' \, dt'$ for integrable functions $f_1, f_2\colon \XC \times \XC \times \RR \to \RR_+$. We define $g^{\ast 1} = g$ and $g^{\ast k} = g \ast g^{\ast(k-1)}$ for $k \geq 2$. Furthermore, $g^{*0}(y,x,t) = \delta^{(d)}(x-y) \, \delta^{(1)}(t)$, where $\delta^{(j)}$ is the Dirac delta function on $\RR^j$ (centered at $0$). 
In the proposition below, we interpret $\lambda$ as a function $\XC \times \XC \times \RR \to \RR_+$ with $\lambda(y,x',t') = \lambda(x') \delta^{(1)}(t')$ for the purpose of the generalized convolution, leading to $(\lambda \ast f_2)(x,t) = \int_{\XC} \lambda(x') \, f_2(x',x,t) \, dx'$ for a function $f_2$ from above. 
\begin{proposition}[Expectation and covariance of an ST Hawkes process] \label{prop:hawkes_moments}
  For $\lambda \colon \XC \to \RR_+$ satisfying $m_{\lambda} = \int_{\XC} \lambda(x) \, dx < \infty$ and $g \colon \XC^2 \times \RR \to \RR_+$ satisfying $g(y,x,h)=0$ for $h<0$ and Assumption~\ref{ass:subprob_temporal_domination}, a (strongly) time-stationary Hawkes process $\Xi$ exists and is $L^2$. We then have, for every $A \in \BC(\XC)$ and $t > 0$,
  \begin{equation} \label{eq:mom1_hawkes}
    \EE[\Xi(A\times (0, t])] = t \, \mu(A) = t \int_{A \times \RR_{+}} \biggl(\lambda \ast \sum_{k=0}^{\infty} g^{\ast k}\biggr)(x, r)  \; dx \; dr \leq t \, \frac{m_{\lambda}}{1-\nubar}.
  \end{equation}
  In particular, $\EE[\Xi(\XC \times (0, t])] = t \, m_\mu = t \, \frac{m_\lambda}{1-\nu}$ if $g$ is constant in the first argument and hence $\nu(y) =: \nu$ is constant as well. 
  
  Furthermore, writing $\bar{G}_{\temp}(s) := \sum_{k=0}^{\infty} \gbar_{\temp}^{*k}(s)$ for the intensity at time $s$ of a cluster generated with the excitation function $\gbar_{\temp}$ and started from a single point at time $0$ (represented by the delta function $\bar{g}_{\temp}^{\ast 0}$), we have for $t \geq 0$
  \begin{equation} \label{eq:rf_covm_hawkes}
    \breve \gamma_{[2],\XC} ([0,t]) \leq \frac{m_\lambda}{1-\nubar_{\temp}} \, \biggl( \int_0^t \int_0^{\infty} \bar{G}_{\temp}(s) \, \bar{G}_{\temp}(s+r) \; ds \; dr - 1\biggr).
  \end{equation}
  
  In particular, since $\int_0^{\infty} \bar{G}_{\temp}(s) \, ds \leq (1-\bar{\nu}_{\temp})^{-1} < \infty$, Assumption~\ref{ass:Xi_mixing_cov_beta} is met with $\beta=0$ by~\Cref{rem:RCov_XC}(ii).
\end{proposition}
The proof of \Cref{prop:hawkes_moments} is based on \Cref{prop:E_and_Var_for_PCP}. %

\begin{example}
Computing expressions of type $\sum_{k=1}^\infty g^{*k}$ explicitly is not straightforward. A particularly nice example is given by the exponential kernel $\gbar_{\temp}(t) = \alpha e^{-\beta t}$ \citep{hawkes1971spectra}, where $0<\alpha < \beta$. Using the properties of the Laplace transform $\LC[f](s) = \int_0^\infty f(t) e^{-st} \dif t$, namely that $\LC[f * g] = \LC[f]\LC[g]$ and $\LC[f+g] = \LC[f] + \LC[g]$, we conclude that
 $$
  	\sum_{k=1}^\infty f^{*k} = \LC^{-1}\left[ \sum_{k=1}^\infty \LC[f]^{k} \right].
  $$
  For the exponential kernel we have $\LC[\gbar_{\temp}](s) = \int_0^\infty \alpha e^{-(\beta+s)t} \dif t = \frac{\alpha}{\beta + s}$ and
  $$
  	\sum_{k=1}^\infty \LC[\gbar_{\temp}]^{k} = \sum_{k=1}^\infty \left(\frac{\alpha}{\beta + s} \right)^k = \frac{\alpha}{\beta-\alpha +s}.
  $$
  Hence, for $s \geq 0$ we obtain, 
  upon denoting $\delta^{(1)}(\cdot)$ the Dirac function at zero on the real line,  
  $$
  	\bar{G}_{\temp}(s) - \delta^{(1)}(s) = \sum_{k=1}^\infty \gbar_{\temp}^{*k}(s) = \LC^{-1}\left[ \frac{\alpha}{\beta-\alpha +s} \right] = \alpha e^{-(\beta-\alpha)s}.
  $$
\end{example}

\subsubsection{Log-Gaussian Cox process}

Another popular spatial and spatio-temporal point process is the Cox (or doubly stochastic) point process, which is obtained by randomizing the parameter measure of a Poisson process. More precisely, let $\Lambda$ be a random measure on $\XC \times (0,\infty)$ (see \citealp[Definition~9.1.VI]{daley2007introduction}). We call a point process $\Xi$ a \emph{Cox process with directing measure $\Lambda$} if its conditional distribution given $\Lambda$ is $\pop(\Lambda)$ almost surely.

One of the most studied models, particularly in spatial and spatio-temporal ecology and epidemics, is the log-Gaussian Cox process (LGCP, see \citealp{moller1998log,brix2001spatiotemporal,diggle2006spatio,simpson2016going,lindgren2024inlabru}). Here, the directing measure is given by a density that is a log-Gaussian random field. More precisely, we consider a product-measurable 
time-stationary Gaussian random field $(Z(x,t))_{x \in \XC, t>0}$ on $\XC \times (0,\infty)$ with mean function $m(x,t) = m(x)$ and covariance function\footnote{That is, $k \colon (\XC \times \TC)^2 \to \RR$ is symmetric and positive semidefinite.} $k(x,x',|t-t'|) \coloneqq k((x,t),(x',t'))$. Define then $\Lambda(A) = \int_A \exp(Z(x,t)) \, dx \, dt$ for all $A \in \BC(\XC \times (0,\infty))$ assuming the integral is finite for bounded $A$. This can be assured (including the abovementioned product-measurability) by assuming continuity of $m$ and a mild continuity condition for $k$, see \citet[Subsection~5.6.1]{moller2003statistical}. 
There are nice formulae for the first two moment measures of such a process. By \citet[Proposition~5.4]{moller2003statistical}, the densities of the first and the second reduced factorial moment measure are given by
\begin{equation*}
  m_1(x) = \exp(m(x) + k(x,x,0)/2) \quad \text{and} \quad \breve g_{[2]}(x,x',\tau) = m_1(x)m_1(x')\exp(k(x,x',\tau)),
\end{equation*}
respectively, so that
\begin{align*}
	M_1(A \times (0, t]) &= t \, \mu(A) = t \int_A m_1(x) \dif x, \\
	\breve \gamma_{[2],A}([0, t]) &= \int_{A^2}\int_{0}^t m_1(x)m_1(x') \left(\exp(k(x,x',\tau))-1\right) \dif \tau \dif x \dif x'.
\end{align*}
The Lebesgue density of $\breve \gamma_{[2],A}^+$ is given by
\begin{align*}
	\breve g^{+}_{[2],A}(\tau) &= \max\left\{0, \int_{A^2}m_1(x)m_1(x') \left(\exp(k(x,x',\tau))-1\right) \dif x \dif x' \right\} \\
	&\leq \int_{A^2}m_1(x)m_1(x') \left(\exp(k^{+}(x,x',\tau))-1\right) \dif x \dif x',
\end{align*}
where $k^+$ denotes the positive part of $k$.

A simplifying assumption on $k^+$ is a space-time separable bound \citep{gelfand2010handbook} of the form
$$
	k^+(x,x',\tau) \leq k_{spat}(x, x') k_{temp}(\tau),
$$ 
where $k_{spat} \colon \XC^2 \to \RR_+$ and $k_{temp} \colon \RR \to \RR_+$ are measurable (not necessarily positive semidefinite) functions. We further assume that $k_{spat}$ and $k_{temp}$ are bounded\footnote{Note that $k$ itself is bounded by Cauchy--Schwarz if $\sup_{x \in \XC} k(x,x,0) = \sup_{(x,t) \in \XC \times (0,\infty)} \Var(Z(x,t))$ is bounded.} by positive constants $L$ and $L'$, respectively, and that $k_{temp}(\tau) \to 0$ as $\tau \to \infty$. By the convexity of the function $r \mapsto \exp(Lr)-1$ for $ r \in [0,L']$ , we obtain for $K = (\exp(LL')-1)/L'$ that
\begin{equation*} 
	\exp(k^+(x,x',\tau)) -1 \leq \exp(L k_{\temp}(\tau)) -1 \leq  K \, k_{temp}(\tau) \quad \text{for all } x, \, x' \in \XC \text{ and } \tau \geq 0,
\end{equation*}
which yields 
\begin{equation} 
  \breve \gamma^{+}_{[2],A}([0, t]) \leq K \mu(A)^2 \int_{0}^t k_{\temp}(\tau) \, d\tau \quad \text{for all $t \geq 0$}.
\end{equation}
Therefore, for any partition $(A_i)_{i =1}^{N}$ of $\XC$ into measurable sets and every $t \geq 0$, 
\begin{align} \label{eq:lgcp_gammabound}
    \sum_{i = 1}^{N}\breve{\gamma}^+_{[2],A_i}([0,t])
  &\leq K \left( \sum_{i = 1}^{N} \mu(A_i)^2 \right) \int_{0}^t k_{\temp}(\tau) \, d\tau \notag\\
  &\leq K m_\mu^2 \int_{0}^t k_{\temp}(\tau) \, d\tau.
\end{align}

\begin{example}	
We check Assumption~\ref{ass:Xi_mixing_cov_beta} for two choices of $k_{\temp}$ that are particularly popular kernels in the machine learning community \citep{rasmussen2003gaussian, liu2020gaussian}, namely, the exponentiated quadratic kernel $k_{temp}(\tau) = \sigma^2 \exp\left(-\frac{\tau^2}{2 l^2}\right)$ and the rational quadratic kernel $k_{temp}(\tau)= \sigma^2 \left(1 + \frac{\tau^2}{2 a l^2} \right)^{-a}, \ a > 0$. We assume $k_{\temp}$ to be standardized to have both variance $\sigma^2$ and length scale $l$ set to~$1$ (working with the general kernels $\sigma^2 k_{\temp}(\tau/l)$ would give the same conclusions with respect to the choice of $\beta$). 
\begin{enumerate}
	\item[$(i)$] For the exponentiated quadratic kernel, we obtain
	\begin{align*}
		\int_{0}^t k_{\temp}(\tau) \, d\tau = \int_{0}^t \exp(-\tau^2/2)\dif \tau \leq  \int_{0}^\infty \exp(-\tau^2/2)\dif \tau
	\end{align*}
        for all $t \geq 0$,
	which is a finite constant independent of $t$. Therefore, Assumption~\ref{ass:Xi_mixing_cov_beta} is satisfied with $\beta=0$ by~\eqref{eq:lgcp_gammabound}.
	\item[$(ii)$] For the rational quadratic kernel we have
	\begin{align*}
		\int_{0}^t k_{\temp}(\tau) \, d\tau &\leq  \int_{0}^t \frac{1}{\left(1 + \tau^2/2a \right)^{a}} \dif \tau
		\leq \sqrt{2a} \int_{0}^t \frac{1}{\tau^{2a}} \dif \tau %
		\leq \frac{\sqrt{2a}}{1-2a} t^{1-2a} %
	\end{align*}
        for all $t \geq 0$ and $a \in (0, 1/2)$.
	Thus it follows that Assumption~\ref{ass:Xi_mixing_cov_beta} is satisfied with $\beta = 1-2a \in (0,1)$ by~\eqref{eq:lgcp_gammabound}.
\end{enumerate}
\end{example}

\section{Lower bounds on estimation of the Kantorovich--Rubinstein distance}\label{sec:lowerbounds}
In this section, we derive several results that are complementary to our upper bounds on the statistical error of the empirical KRD. We first derive lower bounds that arise from the empirical plug-in estimators (\Cref{subsec:lowerBoundPlugIn}). In addition to these results, we also obtain minimax lower bounds for the setting of the Poisson point process (\Cref{subsec:minimaxLowerBounds}), which confirm that the plug-in approach is nearly rate optimal in $t$ and $C$. 
We defer all proofs to Appendix \ref{sec:proofs_LowerBounds}.

\subsection{Lower bounds for plug-in estimators of the Kantorovich--Rubinstein distance}\label{subsec:lowerBoundPlugIn}

Our first lower bound for the empirical KRD is derived from accurately estimating the total mass of the spatial intensity measure. The second lower bound follows from the necessity of correctly allocating mass to each connected component. The third arises from a volumetric argument, which we rigorously formulate using covering numbers, and which is inspired by \citet[Proposition 4.6]{weed2019sharp}. 

\begin{proposition}\label{lem:lowerBoundPlugIn}
Let $(\XC, d)$ be a c.s.m.s.\ and let $\Xi$ be a weakly stationary ST point process on $\XC\times \tgz$ with spatial population measure $\mu \in \MCsX$ and empirical measure $\hat \mu_t(\cdot) \coloneqq t^{-1}\Xi(\cdot \cap (0, t])$ for $t\in \TC_{>0}$. Then, the following holds for $p\geq 1$ and $C>0$. \begin{enumerate}
  \item[$(i)$] (Total mass estimation) It holds 
  \begin{align*}
    \EE\left[\KR_{p,C}(\hat \mu_t, \mu)\right]\geq \frac{C}{2}\EE\left[\left|m_{\hat \mu_t} - m_{\mu}\right|^{1/p}\right].
  \end{align*}
  \item[$(ii)$] (Mass estimation of components) Assume that the support of $\mu$ can be decomposed into $\supp(\mu) = \XC_1 \cup \XC_2$ with $\delta\coloneqq \inf_{x_1 \in \XC_1, x_2\in \XC_2} d(x_1,x_2)>0$. Then it holds 
  \begin{align*}
     \EE\left[\KR_{p,C}(\hat \mu_t, \mu)\right]\geq 
      \left( \frac{C}{2}\wedge \delta \right) \EE\left[\big(\left|\hat \mu_t(\XC_1) - \mu(\XC_1) \right| + \left|\hat \mu_t(\XC_2) - \mu(\XC_2) \right|\big)^{1/p}\right].
  \end{align*}
  \item[$(iii)$] (Dimensionality) Assume for that $\diam(\XC)=1$ and for $\alpha>0$ that $\NC(\epsilon, \supp(\mu), d)\geq \epsilon^{-\alpha}$ for all $\epsilon \in (0,C]$ and $C\in (0,1]$. %
  Then it holds
  \begin{align*}
    \EE\left[\KR_{p,C}(\hat \mu_t, \mu)\right]\geq \EE\left[4^{-1}m_{\mu}^{1/p}((t m_{\hat\mu_t}+1)^{-1/\alpha}\wedge (C/2^{1/p})) \right].
  \end{align*}
\end{enumerate}
\end{proposition}

To put the above bound lower bounds into context, we provide some implications for specific ST point processes that fulfill \ref{ass:Xi_mixing_Var_beta} with $\beta= 0$ before discussing the case $\beta>0$. 

\begin{example}[Binomial and Poisson point processes]
  \begin{enumerate}
    \item[$(i)$] We first focus on the binomial point process from \Cref{sec:Bimomial}, which fulfills \ref{ass:Xi_mixing_Var_beta} for $\beta = 0$ and models the classical i.i.d.\ sampling from a probability measure $\mu\in \PC(\XC)$. Here, $\hat \mu_t =\frac{1}{t}\sum_{i = 1}^{t} \delta_{X_i}$ is an empirical measure indexed over $t\in \NN$ with $X_1,  X_2, \dots \iid \mu$. Then, the bound provided by \Cref{lem:lowerBoundPlugIn}$(i)$ is vacuous since $m_{\hat \mu_t}= m_{\mu} = 1$. Moreover, if $\mu$ has disconnected support with $\supp(\mu) = \XC_1 \cup \XC_2$ and $\inf_{x_1\in \XC_1, x_2\in \XC_2}d(x_1, x_2) = \delta >0$, then it follows from \Cref{lem:lowerBoundPlugIn}$(ii)$ combined with the Paley-Zygmund inequality for $C>0$ and $t\in \NN$ that 
    \begin{align*}
      \EE[\KR_{p,C}(\hat \mu_t, \mu)]
      &\gtrsim_{p} (C\wedge \delta) (\mu(\XC_1)\mu(\XC_2))^{1/2p}t^{-1/2p}.
    \end{align*}
    Finally, if $\mu$ is supported on a compact smooth submanifold of $\RR^d$ with dimension $\alpha$ and diameter $1$, and if $\mu$ is absolutely continuous with respect to the volume measure, then it holds (after potential rescaling) that $\NC(\epsilon, \mu, \|\cdot\|)\geq \epsilon^{-\alpha}$ for all $\epsilon \in (0,1]$. Hence, \Cref{lem:lowerBoundPlugIn} implies with $m_{\hat \mu_t} = 1$ for $C\in (0,1]$ and $t\in \TC_{>0}$ that 
    \begin{align*}
      \EE[\KR_{p,C}(\hat \mu_t, \mu)]\gtrsim_{p} C\wedge (t^{-1/\alpha}).
    \end{align*}
    Further, when $C = t^{-1/\alpha}$, the above bound can be rewritten to $C \wedge t^{-1/\alpha} = C^{1-\alpha/2p} t^{-1/2p}$. Overall, these lower bounds align with the upper bounds displayed in \Cref{th:est_mu_inUOT_beta} for $\alpha \neq 2p$ and $\beta = 0$, and for $\alpha =2p$ they coincide up to a logarithmic term  in $t$.

    \item[$(ii)$] Now let $\mu\in \MCsX$ be a general spatial intensity measure and consider a time-homogeneous Poisson point process (\Cref{sec:example_PoissonPP}) with corresponding empirical measure $\hat \mu_t$ for $t>0$. Then, from \Cref{lem:lowerBoundPlugIn}, using the same dimension assumptions on $\mu$ as for the binomial process setting in $(i)$, we obtain from the Paley-Zygmund inequality and Jensen's inequality for the convex function $x\mapsto x^{-1/\alpha}\wedge (C/2^{1/p})$ that 
    \begin{align*}
       \EE[\KR_{p,C}(\hat \mu_t, \mu)]\gtrsim_  {p}  C m_{\mu}^{1/2p} t^{-1/2p}  +   m_{\mu}^{1/p}(C\wedge (tm_{\mu})^{-1/\alpha}). 
    \end{align*}
    In particular, for $C = (tm_{\mu})^{-1/\alpha}$ we observe that $C = C^{1-\alpha/2p} (tm_{\mu})^{-1/2p}$, and the right term in the lower bounds aligns with the right term in \Cref{th:est_mu_inUOT_beta} for $\beta = 0$ and $\alpha<2p$. Comparing these bounds with the binomial point process, we observe that the term $Cm_{\mu}^{1/2p}t^{-1/2p}$ always manifests even if the support of $\mu$ is connected. 
  \end{enumerate}
\end{example}
In the following, we provide two examples for point processes that fulfill \ref{ass:Xi_mixing_Var_beta} only for positive $\beta\in (0,1)$ and which also achieve convergence rates  for the empirical KRD which are slower than parametric. %
\begin{example}[Long-range dependence]\label{ex:lb_beta_positive}
  We will consider a time-homogeneous spatio-temporal point process $\Xi = \sum_{i = 1}^\infty W_i \delta_{(X_i, T_i)}$, where $W_i \in \{0,1\}$, $X_i \in \XC$, and $T_i \in \RR$ are independent random variables, and $X_1, X_2, \dots, \iid  \mu$ for some probability measure $\mu \in \PC(\XC)$. In particular, the weights $\{W_i\}_{i \in \NN}$ effectively act as a temporal thinning of the process, and the process could also be expressed as $\Xi= \sum_{i = 1}^{\infty} \delta_{\tilde X_i, \tilde T_i}$ for i.i.d.\ random variables $\tilde X_i \in \XC$ and independent thereof but not i.i.d.\ $\tilde T_i \in \RR$.

  In particular, setting $N_t \coloneqq \sum_{i = 1}^{\infty} W_i \delta_{T_i}((0,t])$, which will admit finite variance in our two settings, we obtain for arbitrary $A\in \BC(\XC)$ and $t \in \tgz$ by our independence assumptions from the law of total expectation and variance that
  \begin{align}
    	\EE[\Xi(A \times (0, t])] &= \EE\left[\sum_{i=1}^{\infty} W_i\mathds{1}(X_i \in A)\mathds{1}(T_i \in (0,t])\right] = \mu(A) \EE[N_t],\label{eq:lb_example_mean}\\ 
      \Var(\Xi(A \times (0, t])) &= \mu(A)\left(1-\mu(A)\right) \EE[N_t] + \mu(A)^2 \Var(N_t).\label{eq:lb_example_variance}
  \end{align}
  In the following, we describe two different settings for weights and the time points which will lead to different convergence behaviors for the empirical KRD than the Poisson point processes.

	\begin{enumerate}
\item[$(i)$] Let $\tgz =\NN$. 
We consider deterministic time points $T_i \coloneqq i$ and a stationary sequence of random weights $\{W_i\}_{i \in \NN}$ which will be defined as the map of a stationary Gaussian process. Concretely, we consider a stationary sequence $(\tau_i)_{i \in \NN}$ of Gaussian random variables with %
\[
	\Cov(\tau_t, \tau_1) \asymp t^{2H-2}
\]
for some $H \in (1/2, 1)$. Such a stationary sequence of random variables exists by considering the increments of a fractional Brownian motion with Hurst index $H$, cf. Equations (1--2) of \cite{bai2018instability}.
We then define for $k\in \NN$ the weights as $W_i = G_k(\tau_i)$, where the binary-valued function $G_k\colon \RR\to \{0,1\}$ is selected according to Proposition \ref{prop:G_exists} and fulfills $$\Var\left(\sum_{i=1}^t G_k(\tau_i)\right) \asymp t \,\Var(G_k(\tau_1)) + \frac{2}{(a+1)(a+2)} t^{(2H-2)k +2}.$$ 
Together with \eqref{eq:lb_example_mean}-\eqref{eq:lb_example_variance} this implies that $\Xi$ is weakly time-stationary and that \ref{ass:Xi_mixing_Var_beta} is satisfied with $\beta = (2H-2)k + 1 \in (0,1)$. Moreover, \cite{bai2018instability} (see Proposition \ref{prop:gaussian_clt} for details) show for $t\to \infty$ that $$
		t^{(1-H)k - 1}\left(\sum_{i=1}^{t}G_k(\tau_i) - t\EE[G_k(\tau_1)]\right) 
	$$ converges in distribution to some non-degenerate law. Hence, there exist positive constants $\rho>0$ and $t_0 \in [1, \infty)$ such that for all $t \geq t_0$ it holds with probability at least $1/2$ that
$
	|N_t -  m_\mu t| \geq \rho t^{(-1+\beta)/2}/2 
$, and by Markov inequality combined with \Cref{lem:lowerBoundPlugIn}$(i)$ we infer
	$$
	\EE[\KR_{p,C}(\hat \mu_t, \mu)] \gtrsim_{\rho} \frac{C^p}{2} t^{(-1+\beta)/2},
  $$
  matching the rate in $t$ detailed in \Cref{th:est_mu_inUOT_beta} for $\alpha<2p$ and $\beta \in (0,1)$ and large $C$.
\item[$(ii)$] Now let $\tgz = (0, \infty)$. We mention here a slightly simpler construction which, however, leads to a lower bound that does not match the upper bound in \Cref{th:est_mu_inUOT_beta} for the regime $\alpha<2p$ and $\beta\in (0,1)$. 
	We consider constant weights $W_i = 1$ for all $i\in \NN$ and choose the time-point such that their temporal increments are heavy-tailed. 
    Let $T_2-T_1, T_3-T_2, \dots \overset{i.i.d}{\sim} \mathrm{Pareto}\left(\gamma, \frac{\gamma - 1}{\gamma} m_{\mu}^{-1}\right) =: F$ with $\gamma \in (1,2)$ and assume that $T_1$ has a distribution with cumulative distribution function $G(t) = m_\mu \int_{0}^t  (1-F(s)) \dif s$. We then obtain  $\EE[N_t] = m_\mu t$ \citep[Chapter XI, Section 4]{fellerIntroductionProbabilityTheory1971} and $\Var(N_t) \asymp t^{3-\gamma}$ (Proposition \ref{prop:renewal_mean_var}). 
    Hence, choosing $\beta = 2-\gamma$ and plugging it into \eqref{eq:lb_example_mean}-\eqref{eq:lb_example_variance}, we verify that $\Xi$ is weakly time-stationary and that \ref{ass:Xi_mixing_Var_beta} holds with $\beta \in (0,1)$. Next, the classical renewal theory result  \citep{gnedenko1968limit} in combination with the general inverse CLT argument \citep[Theorem 7.3.1]{whitt2002stochastic} (included as Theorem \ref{prop:renewal_clt}) implies that $t^{-1/(2-\beta)}(N_t- m_\mu t)$ 
    converges in distribution to some stable non-degenerate law as $t \to \infty$. Repeating the argument from the previous example, we conclude that
$$
	\EE[\KR_{p,C}(\hat \mu_t, \mu)] \gtrsim_{\rho} \frac{C^p}{2} t^{1/(2-\beta) -1},
$$
which for $\beta \in (0,1)$ is decreasing strictly slower than the parameter rate but faster than the rate $t^{(-1+\beta)/2}$ obtained in \Cref{th:est_mu_inUOT_beta}.
\end{enumerate}
\end{example}

\begin{remark}[Extension of lower bounds]
The dimension-based lower bound cannot leverage non-trivial dependencies of the ST point process since $t m_{\hat \mu_t}$ grows with order $t$ due to weak stationarity. Assumption~\ref{ass:Xi_mixing_Var_beta} with $\beta>0$ captures this dependence, and we conjecture that it is indeed the sharp condition, giving errors of order $t^{(-1+\beta)/\alpha}$.  Proving matching lower bounds for specific models remains difficult, however, because one must exploit the dependence structure of the respective ST point process and all currently known examples described in \Cref{sec:pp_examples} with $\beta>0$ are technically demanding in this respect.
\end{remark}

\subsection{Minimax lower bounds for estimating the Kantorovich--Rubinstein distance}
\label{subsec:minimaxLowerBounds}

In this section, we derive a minimax lower bound on the estimation of measures with respect to the KRD and of the KRD itself. In particular, these results emphasize that our upper bounds in \Cref{th:est_mu_inUOT_beta} are not improvable. 
Notably, while the pointwise bounds in \Cref{lem:lowerBoundPlugIn} were stated for general ST point processes, in the following we only focus on Poisson point processes whose spatial intensity measure is contained in $\MC(\XC,1)$, the collection of measures on $\XC$ with total mass bounded by one. 

\begin{proposition} \label{prop:KR_minimax}
Let $(\XC,d)$ be a bounded c.s.m.s.\ with $\diam(\XC) = 1$ which fulfills $\NC(\epsilon, \XC, d) \cdot \epsilon^{\alpha}\in [a, A]$ for all $\epsilon \in (0,1]$ and some $A\geq a>0$, $\alpha\geq 0$. Denote by $\Pi_{\mu}$ and $ \Pi_{\nu}$ two independent time-homogeneous ST Poisson point processes with spatial intensity measures $\mu, \nu \in \MCsX$, respectively. Then, for $p\geq 1$ and $C>0$, the following holds. 
\begin{enumerate}[label=$(\roman*)$]
    \item (Measure estimation) For $t\geq 1$ sufficiently large (dependent on $a, p, \alpha$) it holds that %
    \begin{align}\label{eq:minimaxLowerBoundMeasure}
        &\inf_{\tilde \mu_t}\sup_{\mu \in \mathcal{M}(\XC,1)} \EE\left[ \KR_{p,C}(\tilde \mu_t, \mu)\right]\gtrsim_{p,\alpha, a} %
        C\wedge \left(C t^{-1/2p} + C^{1-\alpha/2p} t^{-1/2p} + t^{-1/\alpha}\right),
    \end{align}
    where the infimum is taken over all (measurable) estimators $\tilde \mu_t\in \MCsX$ based on $\Pi_{\mu}(\cdot \times (0,t])$.
    \item (Functional estimation) Let $\kappa>0$ be a positive constant. Then, for $t\geq 1$ sufficiently (dependent on $a, A, p, \alpha, \kappa$) it holds 
    \begin{align}
        &\inf_{\tilde \KR_t}\sup_{\mu, \nu \in \mathcal{M}(\XC,1)} \EE\left[ \left|\tilde \KR_t - \KR_{p,C}(\mu, \nu)\right|\right]\notag \\ %
        &\gtrsim_{a, A, p, \alpha, \kappa} C\wedge \left(C t^{-1/2p}+ \begin{cases}
        C^{1-\alpha/2p} t^{-1/2p}\log(t)^{-1+1/2p} &\text{ if } \alpha\leq 2p \text{ and } C\in I(t), \\
        (t \log(t))^{-1/\alpha}&\text{ if } \alpha> 2p
        \end{cases}\right),%
        \label{eq:minimaxLowerBoundEstimationKRD}
\end{align}
where we define\footnote{Here we use the convention that $I(t) = \emptyset$ whenever the left-hand endpoint exceeds the right-hand one.} $I(t) \coloneqq [(2 t/(a\log(t)))^{-1/\alpha}/2, \min((3/a)^{-1/\alpha}, 2(t/a)^{-1/(\alpha \kappa)})]$, and the infimum is taken over all (measurable) estimators $\tilde \KR_t\in \RR_+$ based on independent $\Pi_\mu(\cdot \times (0,t])$ and $\Pi_\nu(\cdot \times (0,t])$.
\end{enumerate}
\end{proposition}

The lower bounds in Proposition \ref{prop:KR_minimax}$(i)$ match the three error terms achieved by the empirical plug-in estimator in Theorem \ref{th:est_mu_inUOT_beta} up to a logarithmic term for $\alpha = 2p$.  Consequently, no plug-in procedure, i.e., no estimator obtained by replacing the unknown measures by an estimator (based on observations of an ST Poisson point process) can improve the statistical error bound from \eqref{eq:est_mu_inUOT_beta} for $\beta = 0$ in $t$ and $C$. Moreover, based on \Cref{thm:UOT_beta_distinct} and \Cref{prop:KR_minimax}$(ii)$, we also infer that the empirical KRD is also nearly minimax rate optimal in estimating the KRD uniformly in $C$ and $t$ up to logarithmic terms on $t$. We note, however,  that the sharp minimax estimation risk remains open and based on prior work on functional estimation \citep{lepski1999estimation,cai2011testing, jiao2018minimax,han2021optimal, wang2021optimal}, there is reason to believe that the lower bound is not improvable.

\begin{ackno}
    All authors gratefully acknowledge support from the Deutsche Forschungsgemeinschaft (DFG, German Research Foundation) through DFG RTG 2088. Additional support was provided to S. Hundrieser and A. Munk by DFG CRC 1456 and DFG EXC 2067/1-390729940, and to A. Munk by  DFG FOR 5381. S. Hundrieser further gratefully acknowledges funding from the German National Academy of Sciences Leopoldina under grant number LPDS 2024-11. In~addition, M.\ Struleva and S.\ Hundrieser both thank Tudor Manole for helpful discussions on minimax lower bounds in functional estimation and  Leoni Wirth for insightful conversations on spatio-temporal point processes. 
\end{ackno}

\appendix


\section{Proofs of properties of $(p,C)$-KRD}\label{sec:proofs2}

\begin{proof}[Proof of \Cref{prop:Prop_d_trunc}]
To prove the existence of an optimal plan, we show that the $(p,C)$-KR problem can be viewed as a special case of the so-called Optimal Entropy-Transport problem introduced by \cite{liero2018optimal}, for which the existence result is readily available (see Theorem 3.1 therein). Let us first introduce the necessary notation. For a given convex l.c.s. \emph{entropy function} $\varphi:[0, \infty) \rightarrow[0, \infty]$ we introduce the \emph{recession constant} $\varphi_{\infty}'$ by
\begin{align*}
    (\varphi)_{\infty}^\prime & \coloneqq \lim _{s \rightarrow \infty} \frac{\varphi(s)}{s}. 
\end{align*}
Given an \emph{entropy function} $\varphi$, the \emph{relative entropy} of a measure $\gamma \in \MC(\XC)$ to $\mu \in \MC(\XC) $ is defined as
\begin{equation}
    \DC\left(\gamma \mid \mu \right)\coloneqq\int_{\XC} \varphi \left(\frac{\dif \gamma}{\dif \mu}\right) \dif \mu + \left( \varphi \right)_{\infty}^{\prime} \gamma^{\perp}\left(\XC\right), 
\end{equation}
where $\gamma=\frac{\dif \gamma}{\dif \mu}\mu+\gamma^{\perp}$ is the Lebesgue decomposition of $\gamma$ with respect to $\mu$. 
Hence, for an entropy function $\varphi$, a continuous function $c \colon \XC \times \XC \to \RR$ and finite non-negative Borel measures $\mu, \nu \in \MC(\XC)$ on $\XC$, the \emph{Entropy-Transport problem} is formulated as
\begin{equation} \label{eq:liero_primal}
    \mathrm{ET}(\mu, \nu) = \inf_{\pi \in \MC(\XC_1 \times \XC_2)} \int_{\XC_1\times\XC_2} c(x_1, x_2) \dif \pi + \DC\left( \pi_1 \mid \mu\right) + \DC\left(\pi_2 \mid \nu\right).
\end{equation}
According to \citet[Theorem 3.1]{liero2018optimal},
for a feasible problem $\mathrm{ET}(\mu, \nu)$ to have at least one optimal solution, the entropy function has to be superlinear, that is, 
\begin{equation} \label{eq:superlin}
    (\varphi)_{\infty}'=\infty.
\end{equation}
To relate the $(p,C)$-KR problem to the $\mathrm{ET}$ problem, first note that it is feasible since we always have $\pi \equiv 0 \in \Pi_{\leq}(\mu, \nu)$. Further, observe that the mass penalizing terms in \eqref{eq:def_UOT} can be equally rewritten as
\begin{align} \label{eq:alt_penalization1}
     \frac{C^p}{2}\int_{\XC_1} \dif(\mu- \pi_1) = \frac{C^p}{2} \int_{\XC_1} \left(1 - \frac{\dif \pi_1}{\dif \mu} \right)\dif \mu + \iota\left(\sup_{x \in \XC_1} \frac{\dif \pi_1}{\dif \mu}(x) -1\right)
\end{align}
(and in the same manner for $\nu$ and $\pi_2$), where
\begin{equation*}
    \iota(x) \coloneqq \begin{cases} 
        0 & \text{ if } x < 0, \\ 
        +\infty & \text{ if } x \geq 0.
    \end{cases}
\end{equation*}
Due to \eqref{eq:alt_penalization1}, we infer that if we put
\begin{equation} \label{eq:KRD_entropy_func}
    \varphi(x) = \begin{cases} 
        \frac{C^p}{2}(1-x) & \text{ if } x \in [0,1], \\ 
        +\infty & \text{ else,}
    \end{cases}
\end{equation}
then $\KR_{p,C}^p$ has the form \eqref{eq:liero_primal}. The existence of an optimal solution follows immediately since \eqref{eq:superlin} trivially holds.  

To show $(ii)$, let us denote $d_C(x_1, x_2) = d(x_1, x_2) \wedge C$ and define $\KR_{d_C^p}(\mu, \nu)$ for the associated KRD induced by the metric $d_C$. 
Let $\pi$ be an optimal solution to the $\KR_{p,C}$-problem. Then, since $d_C^p(x_1, x_2) \leq d^p(x_1, x_2)$ for any $x_1, x_2 \in \XC$ it follows that
\begin{align*}
    \KR^p_{p, C}(\mu, \nu) &\geq \int_{\XC^2} d^p_C(x_1, x_2) \dif \pi(x_1, x_2) + C^p \left(\frac{m_\mu + m_\nu}{2} - m_\pi\right) \\
    &\geq \inf_{\gamma \in \Pi_\leq(\mu, \nu)} \int_{\XC^2} d^p_C(x_1, x_2) \dif \gamma(x_1, x_2) + C^p \left(\frac{m_\mu + m_\nu}{2} - m_\gamma \right) \\
    &= \KR^p_{d^p_C,C}(\mu, \nu).
\end{align*}
Conversely,  let $\pi_C$ be an optimal solution for the problem with cost function $d^p_C$. Then, upon defining $\tilde \pi_C(\cdot)\coloneqq \pi_C(\cdot\cap \mathcal{D}_C)$ for $\mathcal{D}_C = \{(x,x')\in \XC^2 \,|\, d(x,x')\leq C\}$ it follows that 
\begin{align*}
    \KR^p_{d^p_C,C}(\mu, \nu) &=  \int_{\XC^2} d^p\wedge C^p \dif \pi_C + C^p \left(\frac{m_\mu + m_\nu}{2} - m_{\pi_C} \right) \\
    &= \int_{\mathcal{D}_C} d^p\wedge C^p \dif \pi_C + C^p\pi_C(\XC^2\setminus \mathcal{D}_C) + C^p \left(\frac{m_\mu + m_\nu}{2} - m_{\pi_C} \right) \\
    &= \int_{\XC^2} d^p \dif \tilde \pi_C + C^p \left(\frac{m_\mu + m_\nu}{2} - m_{\tilde \pi_C} \right) \geq \KR_{p,C}^p(\mu, \nu), 
\end{align*}
where the last inequality follows from the fact that $\tilde \pi_C$ is a sub-coupling between $\mu$ and $\nu$ and the definition of the $(p,C)$-KRD. This concludes the proof of Assertion $(ii)$.

Finally, for Assertion $(iii)$ we notice from above computation it follows for any unbalanced OT plan $\pi$ between $\mu$ and $\nu$ with respect to the $(p,C)$-KRD that 
\begin{align*}
   \KR^p_{p, C}(\mu, \nu) &= \int_{\XC^2} d^p(x_1,x_2) \dif \pi(x_1, x_2) + C^p\left(\frac{m_{\mu} + m_{\nu}}{2} - m_{\pi}\right)\\
   &\geq \int_{\XC^2} d_C^p(x_1,x_2) \dif \pi(x_1, x_2)+ C^p\left(\frac{m_{\mu} + m_{\nu}}{2} - m_{\pi}\right)\\
   &\geq \KR_{d_{C},Cx}^p(\mu, \nu) = \KR^p_{p, C}(\mu, \nu),
\end{align*}
which yields that $\pi(\XC^2\setminus \mathcal{D}_C) = 0$. 
\end{proof}

\begin{proof}[Proof of \Cref{prop:linkOT}]
For the proof of Assertion $(i)$ we first note  by \Cref{prop:Prop_d_trunc}$(ii)$ that substituting the cost function $d$ by $d_{C} \coloneqq d \wedge C$ does not change the value of $\KR_{p, C}(\mu, \nu)$. 
Hence, by Lemma 3.1 in \cite{heinemann2023kantorovich} and due to the fact that the balanced Wasserstein distance metrizes weak convergence of measures \citep[Theorem~6.9]{villani2008optimal}, it follows for any $a>0$ and any weakly converging finitely supported sequences $(\tilde\mu_n)_{n \in\NN}, (\tilde\nu_n)_{n \in\NN} \subset \MCsX$  on $\XC$ with respective limits $\mu$ and $\nu$ that
\begin{align*}
    \OT_{\tilde{d}^p_{p,C}}(\tilde \mu, \tilde\nu) = \lim_{n \to \infty} \OT_{\tilde{d}^p_{p,C}}(\tilde \mu_n, \tilde\nu_n) &= \lim_{n \to \infty} \OT_{\tilde{d}^p_{p,C}}(\tilde \mu_n + a\delta_{\mathfrak{d}}, \tilde\nu_n + a\delta_{\mathfrak{d}}) \\
    &= \OT_{\tilde{d}^p_{p,C}}(\tilde \mu + a\delta_{\mathfrak{d}}, \tilde\nu + a\delta_{\mathfrak{d}}).
\end{align*}
In particular, by the almost sure consistency of empirical measures from probability measures \citep{varadarajan1958weak} such sequences always exist. 
Therefore, it is enough to show Equality \eqref{eq:UOT=OT} for an arbitrary fixed parameter $K \geq m_\mu \vee m_\nu$.
Hence, assume w.l.o.g.\ that $m_\mu \geq m_\nu$ and set $K = m_\mu$ (otherwise, if $m_\mu < m_\nu$, set $K=m_\nu$), i.e.
$$
    \tilde \mu_{K} = \mu, \quad \tilde \nu_{K} = \nu + (m_\mu - m_\nu) \delta_{\mathfrak{d}}.
$$
To prove the claim that $\KR_{p,C}^p(\mu, \nu) =  \OT_{\tilde d_{p,C}^p}(\tilde \mu_K, \tilde \nu_K)$ we thus need to confirm the equality
\begin{equation} \label{eq:Kmu_UOT=OT}
  \begin{aligned}
     \inf_{\pi \in \Pi_{\leq}(\mu, \nu)} \int_{\XC^2}  d_{C}^p  \dif \pi + C^p 
    \left( \frac{m_\mu + m_\nu}{2} - m_\pi\right) &=\inf_{\tilde \pi \in \Pi_=(\tilde \mu_{K}, \tilde \nu_{K})} \int_{\tilde\XC^2} \tilde d_{p,C}^p  \dif \tilde \pi.
  \end{aligned}
\end{equation}
Let $\tilde \pi \in \Pi_=(\tilde \mu, \tilde \nu)$ be a balanced optimal transport plan for the right-hand side of \eqref{eq:Kmu_UOT=OT}, which exists by \citet[Theorem 4.1]{villani2008optimal}. Observe that $\tilde\pi(\{\mathfrak d\}\times \{\mathfrak d\}) \leq \tilde\pi(\{\mathfrak d\}\times \tilde\XC) = \tilde \mu (\{\mathfrak d\}) = 0$, hence  $\tilde\pi(\{\mathfrak d\}\times \{\mathfrak d\}) =0$, which leads to $\tilde\pi(\XC \times \{\mathfrak d\}) = \tilde\pi(\tilde \XC \times \{\mathfrak d\}) = \tilde \nu (\{\mathfrak d\}) = m_\mu - m_\nu$. Also note that $m_{\tilde\pi} = m_{\tilde\mu} = m_{\tilde\nu} = m_{\mu}$ and
\begin{align*}
    m_{\tilde\pi|_{\XC^2}} &= \tilde\pi(\XC \times \XC) = \tilde\pi(\tilde\XC^2) - \tilde\pi(\{\mathfrak d\} \times \XC) - \tilde\pi(\XC \times \{\mathfrak d\}) \\
    &= \tilde\pi(\tilde\XC \times \tilde \XC) - \tilde\mu(\{\mathfrak d\}) - \tilde\nu(\{\mathfrak d\}) \\
    &= m_\mu - 0 - m_\mu + m_\nu = m_\nu.
\end{align*}
The right-hand side of \eqref{eq:Kmu_UOT=OT} is therefore equal to
\begin{align*}
    \OT_{\tilde{d}^p_{p,C}}(\tilde\mu_{K}, \tilde\nu_{K}) &=\int_{\tilde\XC^2} \tilde d_{p,C}^p \dif \tilde\pi = \int_{\XC^2}   d_{C}^p  \dif \tilde \pi + \int_{\XC\times\{\mathfrak{d}\}} \frac{C^p}{2} \dif \tilde \pi \\
    &= \int_{\XC^2}   d_{C}^p  \dif \tilde \pi +  \frac{C^p}{2} (m_\mu - m_\nu) \\
    &= \int_{\XC^2}   d_{C}^p  \dif \tilde \pi +  C^p
    \left(\frac{m_{\mu} + m_{\nu}}{2} - m_{\tilde \pi|_{\XC^2}}\right)\geq \KR_{p,C}(\mu, \nu), 
\end{align*}
where the last inequality follows since $\tilde \pi \in \Pi_=(\tilde\mu_{K}, \tilde\nu_{K})$ induces a subcoupling $\tilde \pi|_{\XC^2}\in\Pi_{\leq}(\mu, \nu)$. 

To show the complementary inequality, assume w.l.o.g.\ that $m_{\mu}\geq m_{\nu}$ and let $\pi\in \Pi_{\leq}(\mu, \nu)$ be an unbalanced OT plan for the left-hand side of \eqref{eq:Kmu_UOT=OT}. If $m_{\pi} = m_{\nu}$, then define $\pi^* \coloneqq \pi$. Otherwise, if $m_{\pi}< m_{\nu}\leq m_{\mu}$, we define
$$
    \pi^*\coloneqq \pi + \frac{(\mu-\pi_1)}{ (m_{\mu-\pi_1})} \otimes (\nu-\pi_2)\in \Pi_{\leq}(\mu, \nu),
$$
which fulfills $m_{\pi^*} = m_{\nu}$. 
Therefore, it follows that
\begin{align*}
    \int_{\XC^2} d^p_{C} \dif \pi +  C^p\left(\frac{m_\mu + m_\nu}{2} - m_{\pi}\right)
    &= \int_{\XC^2}\tilde d^p_{C} \dif \pi + C^p m_{\check \pi} + \frac{C^p}{2} (m_\mu + m_\nu - 2(m_{\pi}+m_{\check\pi})) \\
    &\geq \int_{\XC^2}\tilde d^p_{C} \dif \pi^* + \frac{C^p}{2} (m_\mu + m_\nu - 2m_{\pi^*}).
\end{align*}
Finally, upon defining $\tilde \pi^* \coloneqq \pi^* + (\mu-\mathrm p^{1}_\# \pi^*) \otimes \delta_{\mathfrak{d}}\in \Pi_=(\tilde\mu_{K}, \tilde\nu_{K})$, we have %
\begin{align*}
    \int_{\XC^2} d^p_{C} \dif \pi^* + \frac{C^p}{2}(m_\mu + m_\nu - 2m_{\pi^*}) &= \int_{\XC^2} d^p_{C} \dif \pi^* + \frac{C^p}{2}(m_\mu - m_\nu ) \\
    &= \int_{\XC^2} d^p_{C} \dif \tilde \pi^* + \int_{\XC\times\{\mathfrak{d}\}} \frac{C^p}{2}\dif \tilde \pi^*\\
    &= \int_{\tilde \XC^2}\tilde d^p_{p,C} \dif \tilde \pi^* \geq \OT_{d^p_{p,C}}(\tilde \mu_K, \tilde \nu_K), 
\end{align*}
which completes the proof of Assertion $(i)$.

To prove Assertion $(ii)$ we use a similar reasoning. For the first inclusion, let $\pi$ be an unbalanced OT plan between $\mu$ and $\nu$ with respect to the $(p,C)$-KRD. Further, consider the augmented measures $\tilde \mu_K$ and $\tilde \nu_K$ with $K= m_{\mu}\vee m_{\nu}$ and assume w.l.o.g.\ that $m_{\mu}\geq m_{\nu}$. Then it follows that $\tilde \mu_K= \mu$ and $\tilde \nu_K = \nu  + (m_{\mu} - m_{\nu})\delta_{\mathfrak{d}}$. Now consider a feasible transport plan $\breve \pi \in \Pi_=(\tilde \mu_K - \pi_1, \tilde \nu_K - \pi_2)$. Then, it follows that $\tilde \pi = \pi + \breve \pi \in \Pi_{=}(\tilde\mu_K, \tilde \nu_K)$ and thus
\begin{equation}\label{eq:marginalConstraintsExtendedOTPlan}
  \begin{aligned}
   &\tilde \pi(\{\mathfrak{d}\}\times \tilde \XC) = \tilde \mu_{K}(\{\mathfrak{d}\}) = 0\; \text{ and } \; \\
   &\tilde \pi(\XC\times \{\mathfrak{d}\})= \tilde \pi(\tilde \XC\times \{\mathfrak{d}\}) = \tilde \nu_{K}(\{\mathfrak{d}\}) = (m_\mu - m_{\nu}).
  \end{aligned}
\end{equation}
In consequence, it also follows that $\tilde \pi(\tilde \XC\times \tilde \XC)  = m_{\mu}$ and $\tilde \pi(\XC\times \XC)= m_{\nu}$. We thus infer, 
\begin{align*}
 \OT_{\tilde d^p_{p,C}}(\tilde \mu_K, \tilde \nu_K) \leq &\int_{\tilde \XC^2} \tilde d^p_{p,C}\dif \tilde \pi\\
=& \int_{\XC^2}d^p\wedge C^p \dif \tilde \pi + \frac{C^p}{2}\pi(\XC\times \{\mathfrak{d}\})\\
=& \int_{\XC^2} d^p\wedge C^p \dif \tilde \pi + C^p\left(\frac{m_{\mu} + m_{\nu}}{2} - m_{\tilde{\pi}(\cdot \cap \XC^2)}\right)\\
=& \int_{\XC^2} d^p\wedge C^p \dif \pi + C^p\left(\frac{m_{\mu} + m_{\nu}}{2} - m_{\pi}\right) +  \int_{\XC^2} (d^p \wedge C^p)  - C^p \dif \breve \pi \\
\leq & \,\KR_{p,C}^p(\mu, \nu) = \OT_{\tilde d^p_{p,C}}(\tilde \mu_K, \tilde \nu_K),
\end{align*}
where the last inequality follows from the fact that $(d^p \wedge C^p)  - C^p\leq 0$. The above computation shows that $\tilde \pi$ is an OT plan between $\tilde \mu_K$ and $\tilde \nu_K$ for cost function $d^p_{p,C}$ and moreover that $\breve\pi(\mathcal{D}_C) = 0$. From this, it is also evident that $\pi(\cdot \cap \mathcal{D}_C)= \tilde \pi(\cdot \cap\mathcal{D}_C)$. 

For the converse statement, let $\tilde \pi$ be an OT plan between $\tilde \mu_K$ and $\tilde \nu_K$ for cost function $d^p_{p,C}$. W.l.o.g.\ assuming again that $m_{\mu}\geq m_{\nu}$ it follows that $\tilde \pi$ fulfills the marginals constraints from \eqref{eq:marginalConstraintsExtendedOTPlan} and the conclusion. It thus follows for $\pi \coloneqq \tilde \pi(\cdot \cap \mathcal{D}_C)$ that 
\begin{align*}
  \KR_{p,C}^p(\mu, \nu) &= \OT_{d^p_{p,C}}(\tilde \mu_K, \tilde \nu_K)\\
  &= \int_{\tilde \XC^2}d_{p,C}^p\dif \tilde \pi\\
  &= \int_{\XC^2}d^p\wedge C^p \dif \tilde \pi + \int_{\XC\times \{\mathfrak{d}\}} \frac{C^p}{2}\dif \tilde \pi\\
  &= \int_{\mathcal{D}_C} d^p\dif \tilde \pi + C^p \tilde \pi(\XC^2 \backslash \mathcal{D}_C) + C^p\left(\frac{m_{\mu}+ m_{\nu}}{2} - m_{\tilde \pi}\right)\\
  &= \int_{\XC^2}d^p \dif \pi  + C^p\left(\frac{m_{\mu}+ m_{\nu}}{2} - m_{\pi}\right) \geq \KR_{p,C}^p(\mu, \nu),
\end{align*}
where the fourth equality follows from \eqref{eq:marginalConstraintsExtendedOTPlan} and the last equality follows since $\pi\in \Pi_{\leq}(\mu, \nu)$. Hence, $\pi$ is an unbalanced OT plan between $\mu$ and $\nu$, and by definition of $\pi$ in terms of $\tilde \pi$ the assertion follows. 
\end{proof}

\begin{proof}[Proof of \Cref{prop:kr_properties}]

    Based on \Cref{prop:Prop_d_trunc}, Assertion $(i)$ follows once we show for $K\geq m_{\mu}\vee m_{\nu}$ that $$\OT_{\tilde d_{p,C_1}^p}(\tilde \mu_K, \tilde \nu_K)\leq \OT_{\tilde d_{p,C_2}^p}(\tilde \mu_K, \tilde \nu_K).$$ Indeed, this follows since $\tilde d_{C_1}\leq \tilde d_{C_2}$ on $\tilde \XC\times \tilde \XC$ and since the OT cost is non-decreasing in the cost function.  Likewise, for Assertion $(ii)$ we note that 
     $$\OT_{d_{p,C}^p}^{1/p}(\tilde \mu_K, \tilde \nu_K)\leq \OT_{d_{q,C}^p}^{1/p}(\tilde \mu_K, \tilde \nu_K) \leq \OT_{d_{q,C}^q}^{1/q}(\tilde \mu_K, \tilde \nu_K),$$
     where the first inequality follows from $2^{-1/p} \leq 2^{-1/q}$, which yields $d^p_{p,C} \leq d^p_{q,C}$, and the second inequality follows by \citet[Remark 6.6]{villani2008optimal}.
     For Assertion $(iii)$ we note for $K = m_\mu\vee m_\nu$ that $\OT_{\tilde d_{p,C}^p}(\widetilde {a\mu}_{aK}, \widetilde {a\nu}_{aK})=  a \OT_{\tilde d_{p,C}^p}(\tilde \mu_{K}, \tilde \nu_{K})$, which yields the first claim using \Cref{prop:Prop_d_trunc}.  For the second statement, we consider for $a\in [0,1]$ the transport plan $\pi_a = (\textup{Id},\textup{Id})_{\#}(a \mu)$ and for $a>1$ we take $\pi_a = (\textup{Id},\textup{Id})_{\#}\mu$. Note that $\pi_a\in \Pi_{\leq}(\mu, a\mu)$ for every $a\in [0,\infty)$ and that 
     \begin{align*}
    \KR_{p,C}^p(\mu, a\mu) &\leq C^p \left(\frac{m_\mu + a m_{\mu}}{2}- \min\{1,a\}m_{\mu}\right) 
    = \frac{C^p}{2}|a-1|m_{\mu}.
     \end{align*}    
     The second claim now follows from the lower bound on the $(p,C)$-KRD in Assertion $(iv)$, which we show next. 
    To show the first inequality assume w.l.o.g.\ that $m_\mu \geq m_\nu$. Then, for an unbalanced OT plan  $\pi$ between $\mu$ and $\nu$, it follows, 
    $$
        \KR_{p,C}^p(\mu, \nu) \geq C^p \left(\frac{m_\mu + m_\nu}{2} - m_\pi \right) = \frac{C^p}{2} \left(m_{\mu-\pi_1} + m_{\nu-\pi_2}\right).
    $$
    Since $m_{\pi_1} = m_{\pi_2} \leq m_\nu$ we infer that the right-hand side is lower bounded by $\frac{C^p}{2}(m_\mu - m_\nu) = \frac{C^p}{2}|m_\mu - m_\nu|$.  The second inequality follows by plugging the zero measure into the objective of \eqref{eq:def_UOT}. 
    
    Finally, for Assertion $(v)$ we note under $m_\mu\wedge m_\nu=0$ that the lower and upper bound from Assertion $(iv)$ coincide. 
    \end{proof}

    \begin{proof}[Proof of \Cref{prop:interpolate_C}]
        For Assertion $(i)$ note that under the described assumption on $C$ we have $d^p(x_1, x_2)\wedge C^p= C^p \mathds{1}(x_1\neq x_2)$ for $x_1\in \supp(\mu), x_2\in \supp(\nu)$ and hence using \Cref{prop:Prop_d_trunc} we obtain
        \begin{align*}
          \KR_{p,C}^p(\mu, \nu) &= \inf_{\pi\in \Pi_{\leq}(\mu, \nu)} \int C^p\mathds{1}(x_1\neq x_2) \mathrm{d}\pi(x_1,x_2) + C^p\left(\frac{m_\mu + m_\nu}{2} - m_\pi\right)
        \end{align*}
        Plugging $\pi\in \Pi_{\leq}(\mu, \nu)$ into the objective on the right-hand side, we obtain 
        \begin{align*}
          C^p \pi(\XC^2\backslash\{(x,x): x\in \XC\}) + C^p\left(\frac{m_\mu + m_\nu}{2} - m_\pi\right)&=C^p\left(\frac{m_\mu + m_\nu}{2} - \pi(\{(x,x): x\in \XC\})\right)
        \end{align*}
         Now define the measure $\xi$ in terms of $\dif \xi \coloneqq (f_\mu \wedge f_\nu)\dif(\mu + \nu)$. Based on the marginal constraints on $\pi$ it follows that the marginals of the restricted measure $\tilde \pi\coloneqq \pi|_{\{(x,x): x\in\XC\}}$ are dominated by $\xi$, and  consequently \begin{align}\label{eq:ineq_TV_plan}
            \pi(\{(x,x): x\in \XC\}) \leq \xi(\XC).
        \end{align}
        Meanwhile, the measure $\overline \pi \coloneqq (\id, \id)_{\#}(\xi)$ is contained in $\Pi_{\leq}(\mu, \nu)$ and attains the bound in \eqref{eq:ineq_TV_plan}. Hence, we conclude that 
        \begin{align*}
           \KR_{p,C}^p(\mu, \nu) &= C^p\left(\frac{m_\mu + m_\nu}{2} - m_\xi\right) \\
           &= C^p\int \frac{f_\mu(x) + f_\nu(x)}{2} - f_\mu(x)\wedge f_\nu(x) \dif(\mu+\nu)(x)\\
           &= C^p \int \frac{|f_\mu(x) - f_\nu(x)|}{2}\dif(\mu + \nu)(x) = \frac{C^p}{2} \textup{TV}(\mu, \nu).
        \end{align*}
        
        For Assertion $(ii)$ note that under the described setting, since the supports of $\mu$ and $\nu$ are disjoint, $\textup{TV}(\mu, \nu) = \int f_\mu(x) - f_\nu(x)\dif(\mu + \nu)(x) = m_\mu + m_\nu.$
        
        Finally, for Assertion $(iii)$ we employ the link to the augmented OT problem (\Cref{prop:linkOT}). Observe that $C \geq d(x_1, x_2)$ for $x_1\in \supp(\mu), \ x_2 \in \supp(\nu)$ implies that $\tilde d_{p,C}(x_1, x_2) = d(x_1, x_2)$ on $\supp(\mu)\times \supp(\nu)$. Further, by choosing $K\coloneqq m_\mu = m_\nu$ in the definition of the lifted OT problem it follows that $\tilde \mu_K = \mu$ and $\tilde \nu_K = \nu$ and the lifted OT functional coincides with the Wasserstein distance on $\supp(\mu)\times \supp(\nu)$.
        \end{proof}

        \begin{proof}[Proof of \Cref{prop:metricKRD}]
            Symmetry of the $(p,C)$-KRD follows by symmetry of sub-couplings and the objective. Further, $\KR_{p,C}(\mu, \nu)=0$ if $\mu =\nu$. Conversely, if $\KR_{p,C}(\mu, \nu)=0$, then necessarily by \Cref{prop:kr_properties}$(i)$ it follows that $\mu$ and $\nu$ admit identical mass. Hence, in the augmented OT problem \Cref{prop:linkOT} one can choose $\tilde \mu = \mu$ and $\tilde \nu=\nu$, which implies by $\OT_{\tilde d^p_{p,C}}(\mu, \nu) = \KR_{p,C}^p(\mu, \nu) =0$ and since $\OT_{\tilde d^p_{p,C}}$ is a metric on the space of measures with identical mass, that $\mu= \nu$. 
            To show the triangle inequality, we use \Cref{prop:linkOT}. For three measures $\mu, \nu, \tau \in \MC(\XC)$ and $K\coloneqq \max\{m_\mu, m_\nu, m_\tau\}$ it holds that
            \begin{align*}
                \KR_{p, C}(\mu, \tau) = \OT_{\tilde{d}^p_{p,C}}^{1/p}(\tilde{\mu}_K, \tilde{\tau}_K) &\leq \OT_{\tilde{d}^p_{p,C}}^{1/p}(\tilde{\mu}_K, \tilde{\nu}_K) + \OT_{\tilde{d}^p_{p,C}}^{1/p}(\tilde{\nu}_K, \tilde{\tau}_K) \\
                &= \KR_{p, C}(\mu, \nu) + \KR_{p, C}(\nu, \tau),
            \end{align*}
            where the inequality follows by the triangle inequality for the Wasserstein distance \cite[Chapter 6, p.\ 106]{villani2008optimal}. 
            \end{proof}

            \begin{proof}[Proof of \Cref{prop:weak_conv_KRD}]
                For $(i)$, observe that $\mu_n \to \mu$ weakly implies $m_{\mu_n} \to m_\mu$. In particular, if $m_{\mu} = 0$, then by \Cref{prop:kr_properties}$(v)$ it follows that $\KR_{p,C}(\mu_n, \mu) = \frac{C^p}{2}m_{\mu_n} \searrow 0$. Moreover, if $m_{\mu} > 0$, then for all $n$ large enough it also follows that $m_{\mu_n}>0$ and the renormalized measures $\mu/m_{\mu}$ and $\mu_n/m_{\mu_n}$ are well-defined. Hence, using \Cref{prop:kr_properties}$(iii)$, we have
                \begin{align*}
                    \KR_{p,C}(\mu_n, \mu) &= m^{1/p}_\mu \KR_{p,C}\left(\frac{1}{m_\mu} \mu_n, \frac{1}{m_\mu} \mu \right) \\
                    &\leq m^{1/p}_\mu \KR_{p,C}\left(\frac{1}{m_{\mu_n}} \mu_n, \frac{1}{m_\mu} \mu_n \right) 
                    + m^{1/p}_\mu \KR_{p,C}\left(\frac{1}{m_{\mu}} \mu, \frac{1}{m_{\mu_n}} \mu_n \right) \\
                    &\leq \frac{C}{2^{1/p}} \left|1-\frac{m_\mu}{m_{\mu_n}} \right|^{1/p} + m^{1/p}_\mu \KR_{p,C}\left(\frac{1}{m_{\mu}} \mu, \frac{1}{m_{\mu_n}} \mu_n \right) .
                \end{align*}
                By the convergence $\lim_{n\to \infty}m_{\mu_n}= m_\mu$ it follows that $\lim_{n\to \infty}|1-m_\mu/m_{\mu_n}|^{1/p}=0$. Further, as  $\frac{1}{m_{\mu}} \mu$ and $\frac{1}{m_{\mu_n}} \mu_n $ both have mass 1, it follows from \Cref{prop:interpolate_C}$(iii)$ that
                $$
                    \KR_{p,C}\left(\frac{1}{m_{\mu}} \mu, \frac{1}{m_{\mu_n}} \mu_n \right) = W_{p, d \wedge C}\left(\frac{1}{m_{\mu}} \mu, \frac{1}{m_{\mu_n}} \mu_n \right),
                $$
                which vanishes for $n \to \infty$ since $\frac{1}{m_{\mu_n}} \mu_n $ weakly converges to $\frac{1}{m_{\mu}} \mu$ and $W_{p,d\wedge C}$ metrizes weak convergence \citep[Section 6]{villani2008optimal}. 
                
                For the converse implication let $\mu, (\mu_n)_{n\in \NN}\subseteq \MC(\XC)$ be such that $\lim_{n\to \infty}\KR(\mu_n, \mu) = 0$, then by \Cref{prop:kr_properties}$(iv)$ it follows that $\lim_{n\to \infty}m_{\mu_n}= m_{\mu}$. Now, if $m_{\mu} = 0$ it must follow that $\mu_n\to \mu = 0$. Meanwhile, under $m_{\mu}>0$ it follows for all $n$ large enough that $m_{\mu_n}>0$ and thus using \Cref{prop:kr_properties}$(iii)$ that
                \begin{align*}
                    \KR_{p,C}\left(\frac{1}{m_{\mu_n}} \mu_n,\frac{1}{m_{\mu}} \mu\right) &\leq \KR_{p,C}\left(\frac{1}{m_{\mu_n}} \mu_n, \frac{1}{m_{\mu}} \mu_n \right)+ \KR_{p,C}\left(\frac{1}{m_{\mu}} \mu_n, \frac{1}{m_{\mu}} \mu \right)\\
                    &\leq \frac{C}{2^{1/p}}\left|1- \frac{m_{\mu_n}}{m_{\mu}}\right|^{1/p} +m_{\mu}^{-1/p} \KR_{p,C}(\mu, \mu_n)
                \end{align*}
                By assumption, both terms on the right-hand side tend to zero for $n\to \infty$, hence we conclude that $\mu_n/m_{\mu_n}$ weakly converges to $\mu/m_{\mu}$. Jointly with the convergence $\lim_{n\to \infty}m_{\mu_n}= m_{\mu}$, the weak convergence of $\mu_n$ to $\mu$ for $n\to \infty$ is evident. 

                Assertion $(ii)$ follows from triangle inequality for $\KR_{p,C}$ and Assertion $(i)$, i.e., for $n\to \infty$ it holds 
                \begin{align*}
                    \left|\KR_{p, C}(\mu_n, \nu_n) -  \KR_{p, C}(\mu, \nu)\right|\leq \KR_{p, C}(\mu_n, \mu)+ \KR_{p, C}(\nu_n, \nu) \searrow 0.& \qedhere
                \end{align*}

                \end{proof}

\section{Additional proofs for the empirical KRD}\label{sec:proof_est_mu}

\begin{proof}[Proof of \Cref{lem:trivialUpperBound}]
  Combining Jensen's inequality with \Cref{prop:kr_properties}$(iv)$, and due to $$\EE[m_{\mun}] = \frac{1}{t} \EE[\Xi(\XC \times (0, t])] = \frac{1}{t} \mu(\XC) t = m_\mu,$$ we obtain from first-order time-stationary of $\Xi$ that
    \begin{equation*}
   \EE[\KR_{p, C}(\mun, \mu)]\leq \EE[\KR^p_{p, C}(\mun, \mu)]^{1/p} \leq C \EE[m_{\mun} + m_\mu]^{1/p} = 2^{1/p}C^p m_\mu^{1/p}. 
\end{equation*}
The second assertion follows by the triangle inequality.
\end{proof}

\begin{proof}[Proof of \Cref{prop:consistency}]
  Let $\epsilon>0$. Then we obtain for every open set  $A\subseteq \XC$  by Chebyshev inequality and our assumptions on $\Xi$ that
\begin{align*}
  \PP( |\hat{\mu}_t(A) - \mu(A)| \geq \eps) &= \PP \left( \left(\frac{1}{t}\Xi(A \times (0, t]) - \mu(A) \right)^2 \geq \eps^2 \right) \\
  &\leq \frac{1}{t^2 \eps^2} \Var \left( \Xi(A \times (0, t]) \right)\\
  &= o(1)\quad \text{as $t \to \infty$}.
\end{align*}
Hence,  it follows for $t\to \infty$ that $\hat{\mu}_t(A) \stackrel{P}{\to } \mu(A)$. Since $A$ was arbitrary and $\XC$ is separable,  we infer weak convergence $\mun \stackrel{w}{\to}\mu$ in probability. This asserts from \Cref{prop:weak_conv_KRD} that $\KR_{p,C}(\mun, \mu)\stackrel{P}{\to } 0$. 
Moreover, using the asymptotic variance bound for $A=\XC$, we obtain for $t \to \infty$ that
\begin{align*}
  \EE\left[\KR_{p,C}^{2p}(\hat \mu_t, \mu)\right] \leq \frac{C^{2p}}{2}\EE\left[(m_{\hat \mu_t}^2 + m_{\mu}^2)\right]=  \frac{C^{2p}}{2}\left(t^{-2}\Var\left(\Xi(\XC\times [0,t])\right) +2m_{\mu}^2\right) \leq o(1) + 2,
\end{align*}
which is uniformly bounded for all $t$ sufficiently large, say, for $t\geq t_0>0$. Hence, it follows that $\{\KR_{p,C}^{p}(\hat \mu_t, \mu)\}_{t\geq t_0}$ is uniformly integrable. The assertion now follows by combining convergence in probability with uniform integrability. 
\end{proof}

\begin{proof}[Proof of \Cref{lem:uot_proj_bound}]
As shown in \Cref{prop:metricKRD}, $\KR_{p, C}$ defines a metric on the space $\MCsX$. Therefore, employing H\"older's inequality, we get
\begin{align} 
    \KR_{p, C}^{p}(\mu, \nu) &\leq \left(\KR_{p, C}(\mu, \mu^\varepsilon) + \KR_{p, C}(\mu^\varepsilon,\nu^\varepsilon) + \KR_{p, C}(\nu^\varepsilon, \nu)\right)^p \notag \\
    &=3^p \left(\frac{1}{3}\KR_{p, C}(\mu, \mu^\varepsilon) + \frac{1}{3}\KR_{p, C}(\mu^\varepsilon,\nu^\varepsilon) + \frac{1}{3}\KR_{p, C}(\nu^\varepsilon, \nu)\right)^p \notag \\
    &\leq 3^{p-1}\left(\KR_{p, C}^{p}(\mu, \mu^\varepsilon) + \KR_{p, C}^{p}(\mu^\varepsilon,\nu^\varepsilon) + \KR_{p, C}^{p}(\nu^\varepsilon, \nu)\right).\label{eq:uot_dif_proj_bound}
\end{align}
To bound $\KR^p_{p, C}(\mu, \mu^\varepsilon)$, we construct a coupling $\pi_\epsilon = (\id, \mu^\varepsilon) \in \Pi_{\leq}(\mu, \mu^\varepsilon)$, for which $m_{\pi} = m_{\mu} = m_{\mu^\varepsilon}$. By construction, $d(x, P^\varepsilon(x)) \leq \varepsilon$ for every $x\in \XC$. It therefore follows that
\begin{align*}
    \KR^p_{p, C}(\mu, \mu^\varepsilon) &\leq \int_{\XC}C^p\wedge \varepsilon^p \dif \mu = (C\wedge \varepsilon)^p m_\mu.
\end{align*}
Similarly, $\KR_{p, C}^{p}(\nu, \nu^\varepsilon)\leq (C\wedge \varepsilon)^p m_\nu$. 
This observation in conjunction with \eqref{eq:uot_dif_proj_bound} leads to the desired bound. 
\end{proof}

\begin{proof}[Proof of \Cref{lem:KR_tree_est}]
The result follows from \citet[Theorem 2.3]{heinemann2023kantorovich} which provides for discrete measures ($\mu^{\varepsilon}$ and $\nu^{\varepsilon}$ supported on $S_\varepsilon = Q_{L+1}$ in our notation) an upper bound
\begin{equation} \label{eq:ultratree_KRD_expr}
\begin{aligned}
	\KR_{p, C}^p(\mu^{\varepsilon}, \nu^{\varepsilon}) &\leq 2^{p-1} \sum_{j=l}^{L+1} 
	\sum_{s \in Q_j} (h_L^p(j-1) - h_L^p(j))|\mu^{\varepsilon}(\CC(s)) - \nu^{\varepsilon}(\CC(s))|\\
	&+ \left(\frac{C^p}{2} - 2^{p-1}h_L^p(l)\right) \sum_{s' \in Q_{l}} |\mu^{\varepsilon}(\CC(s')) - \nu^{\varepsilon}(\CC(s'))|
\end{aligned}
\end{equation}
whenever $h_{L}(l) \leq C/2 < h_L(l-1), \ l = 1, 2, \dots, L+1$. %
We first note that by construction, $\mu^{\epsilon}(\CC(s))$ for $s \in Q_j$ equals $\mu(W_{s,j})$ as it accumulates all the mass coming from children of $s$. Further, we bound the second summand in \eqref{eq:ultratree_KRD_expr} by 
$
	\frac{C^p}{2} |m_\mu - m_\nu|.
$

Observe that the sum on the right-hand side of \eqref{eq:ultratree_KRD_expr} is decreasing in $l$, as we sum over fewer non-negative summands. Therefore, rewriting $h_{L}(l) \leq C/2$ as (using \eqref{eq:L_definition} and assuming $\varepsilon < \diam(S_\varepsilon)$ and $l\leq L+1$)
\begin{align*}
	l &\geq 1 +\min\left\{L,- 0 \vee\log_2\left( \frac{C/2}{\diam(S_\varepsilon)} + 2^{-L}\right)\right\} \\
	&\geq 1 \min\left\{L,- 0 \vee\log_2\left( \frac{C}{2 \diam(S_\varepsilon)} + \frac{\varepsilon}{2 \diam(S_\varepsilon)}\right)\right) \\
	&\geq 1 + \min\left\{L,\left\lfloor 0 \vee \log_2\left( \frac{2 \diam(S_\varepsilon)}{C+\varepsilon}\right)  \right\rfloor\right\} \eqqcolon l_*, 
\end{align*}
we obtain the desired bound.
\end{proof}

\begin{proof}[Omitted calculations in the proof of \Cref{th:est_mu_inUOT_beta}] 
	It remains to minimize the term \eqref{eq:KR_hash_est} plus $ \frac{C^p}{2} \kappa^{1/2}t^{(-1+\beta)/2} +  2m_\mu\varepsilon^p$ in $\epsilon>0$, which we treat separately for the three cases $\alpha < 2p$, $\alpha=2p$ and $\alpha>2p$.
  
  For $\alpha <2p$, we have a sum of three non-decreasing terms in $\epsilon$, and hence the minimum is obtained at minimal value $\varepsilon = 0$. We obtain
\begin{equation*}
	\EE[\KR_{p, C}^p(\mun, \mu)] \leq \frac{C^p}{2}\kappa^{1/2}t^{(-1+\beta)/2} + \frac{2^{3p-1}}{1-2^{\alpha/2 - p}} A^{1/2} C^{p-\alpha/2} \kappa^{1/2}t^{(-1+\beta)/2}.
\end{equation*}
For $\alpha = 2p$, we note that 
\begin{align*}
  |L +1-l_*| &\leq \left\lceil 0 \vee \log_2\left(\frac{\diam(S_\epsilon)}{\epsilon}\right)\right\rceil - \left\lfloor 0 \vee \log_2\left(\frac{\diam(S_\epsilon)}{C+\epsilon} \right)\right\rfloor\\
  &\leq 2 + \left( 0 \vee \log_2\left(\frac{\diam(S_\epsilon)}{\epsilon}\right)\right) - \left( 0 \vee \log_2\left(\frac{\diam(S_\epsilon)}{C+\epsilon} \right)\right)\\
  &\leq 2 + \left( \log_2\left(\frac{\diam(S_\epsilon)}{\epsilon}\right) - \log_2\left(\frac{\diam(S_\epsilon)}{C+\epsilon}\right) \right)\\
  &\leq 2+ \log_2\left(\frac{C+\epsilon}{\epsilon}\right).
\end{align*}
Hence, we need to minimize the following expression in $\epsilon>0$, 
\begin{align*}
  	\frac{C^p}{2} \kappa^{1/2}t^{(-1+\beta)/2} +  2m_\mu(C\wedge \varepsilon)^p + 2^{3p-1} A^{1/2} \left(2+\log_2\left( \frac{C+\epsilon}{\varepsilon}\right)\right) \kappa^{1/2}t^{(-1+\beta)/2}.
\end{align*}
Plugging in $\epsilon = C\wedge (m_{\mu}^{-1/p}A^{1/2p}\kappa^{1/2p}t^{(-1+\beta)/2p})$ yields the upper bound 
\begin{align*}
  \EE[\KR_{p, C}^p(\mun, \mu)]
 &\leq \;\frac{C^p}{2} \kappa^{1/2}t^{(-1+\beta)/2} + 2m_{\mu}(C^p\wedge (m_{\mu}^{-1}A^{1/2}\kappa^{1/2}t^{(-1+\beta)/2})) + 2^{3p} A^{1/2}\kappa^{1/2}t^{(-1+\beta)/2}\\
  & \quad +2^{3p-1}A^{1/2}\kappa^{1/2}\log_2\left(\frac{2C}{C\wedge (m_{\mu}^{-1/p}A^{1/2p}\kappa^{1/2p}t^{(-1+\beta)/2p})}\right)t^{(-1+\beta)/2p}\\
  &\lesssim_p \frac{C^p}{2} \kappa^{1/2}t^{(-1+\beta)/2} +  A^{1/2}\kappa^{1/2}t^{(-1+\beta)/2}(1\vee \log(C^pt^{(1-\beta)/2}(A\kappa)^{-1/2}m_{\mu})).
\end{align*}
Finally, for $\alpha > 2p$, we find the minimum at
$$
	\varepsilon = \left( \frac{\alpha}{2} - p\right)^{2/\alpha}(2p m_\mu)^{-2/\alpha} A^{1/\alpha}  \left(2^{p-1} + 2^{3p-1} \frac{2^{\alpha/2 - p}}{2^{\alpha/2 - p} - 1}  \right)^{2/\alpha} \kappa^{1/\alpha}t^{(\beta-1)/\alpha}
$$
which again goes to 0 with $t \to \infty$. Plugging it back in, we find that 
\begin{align*}
	\EE[\KR_{p, C}^p(\mun, \mu)]\,&\leq 4\left( \frac{\alpha}{2} - p\right)^{2p/\alpha} m_{\mu}(2pm_\mu)^{-2p/\alpha}  \left(2^{p-1} + 2^{3p-1} \frac{2^{\alpha/2 - p}}{2^{\alpha/2 - p} - 1}  \right)^{2p/\alpha}
	\!\!\!\!\!\!A^{p/\alpha} \kappa^{p/\alpha}t^{p(1-\beta)/\alpha} \\
	& \hspace{1cm} + \frac{C^p}{2} \kappa^{1/2}t^{(-1+\beta)/2} \\
	&\lesssim_{p, \alpha} \frac{C^p}{2} \kappa^{1/2}t^{(-1+\beta)/2} + A^{p/\alpha} m_\mu^{1-2p/\alpha} \kappa^{p/\alpha}t^{p(1-\beta)/\alpha}.
\end{align*}
It remains to take the $p$-th root of both sides (recalling that $(a+b)^{1/p} \leq a^{1/p} + b^{1/p}$ for $a, b, > 0$, $p \geq 1$) in conjunction with Jensen's inequality to obtain the bound \eqref{eq:est_mu_inUOT_beta}.
\end{proof}

\begin{proof}[Proof of \Cref{th:est_mu_in_U}]
We treat the case $\mu = \nu$. If $\mu \neq \nu$, the result follows by the triangle inequality. 
Since $\U$ is $p$-dominated, we write
\begin{align*}
    \EE \left[\U(\mun,\mu)\right] &\leq  \EE \left[\U\left(\mu,\frac{m_\mu}{m_{\mun}}\mun\right)\right] + \EE \left[\U\left(\mun,\frac{m_\mu}{m_{\mun}}\mun\right)\right].
\end{align*}
For the second term, we have $\EE \left[\U\left(\mun,\frac{m_\mu}{m_{\mun}}\mun\right)\right] \leq \left(\EE \left[|m_\mu - m_{\mun}|\right] \right)^{1/p}$.
The right-hand side does not depend on $\U$ and is upper bounded by $\kappa^{1/2p} t^{(-1+\beta)/2p}$ by \ref{ass:Xi_mixing_Var_beta}.

Again using $p$-dominance of $\U$, we bound the first term by $\EE \left[W_p\left(\mu,\frac{m_\mu}{m_{\mun}}\mun\right)\right]$. Observe that $W_p\left(\mu,\frac{m_\mu}{m_{\mun}}\mun\right) = \KR_{p,C} \left(\mu,\frac{m_\mu}{m_{\mun}}\mun\right)$ for $C$ large enough (e.g. we can put $C = \diam(\supp \, \mu ) + \diam(\supp \, \nu )+1$).  The bound then follows by \Cref{th:est_mu_inUOT}.
\end{proof}

\section{Proofs of properties of ST point processes}
\label{sec:proofs_ST_proc}

\subsection{Proofs on moment measures for time-homogenous ST point processes}

\begin{proof}[Proof of \Cref{prop:time-reduction}]
The result is a straightforward extension of Proposition 8.1.I and 8.1.II(a) in \citet{daley2003introduction}. For self-containedness and since the former has only a sketch of a proof, we provide a complete proof here.
  
  $(i)$ $M_1 \chi_t^{-1} = M_1$ implies that for any $A \in \BC(\XC)$, the measure $M_1(A \times \cdot)$ is a translation invariant measure that is finite on bounded intervals, hence a multiple of Lebesgue measure if $\TC=\RR$ and a multiple of the counting measure if $\TC=\ZZ$. Thus, setting $\mu = M_1(\cdot \times (0,1])$, we have $M_1(A \times I) = \mu(A) \, \ell(I)$ for all $A \in \BC(\XC)$ and $I \in \IC$ regardless whether $\TC=\RR$ or $\TC=\ZZ$. 
  
  $(ii)$ We first show that $M_{[2]} \psi^{-1}$ is translation invariant in the first time-component. Indeed, for any $t \in \TC$, writing $\tau_t \colon \TC \to \TC$, $s \mapsto s+t$, we have
  \begin{align*}
    \psi^{-1}(A \times B \times \tau_t^{-1}(I) \times J) %
    &= A \times B \times \{(r,s) \in \TC \times \TC \mvert \tau_t(r) \in I, \tau_t(s)-\tau_t(r) \in J \} \\
    &= (\chi_t,\chi_t)^{-1} ( \psi^{-1}(A \times B \times I \times J) )
  \end{align*}
  for all $A,B \in \BC(\XC)$ and $I,J \in \IC$. Thus, $M_{[2]}(\chi_t,\chi_t)^{-1} = M_{[2]}$ yields
  \begin{align*}
    M_{[2]}\psi^{-1}(A \times B \times \tau_t^{-1}(I) \times J) = M_{[2]}\psi^{-1}(A \times B \times I \times J).
  \end{align*} 
  
  Proceeding similarly as in $(i)$, we set $\breve{M}_{[2]}(A \times B \times J) = M_{[2]}\psi^{-1}(A \times B \times (0,1] \times J)$ and obtain from the translation invariance in the first time component that 
  \begin{equation*}
    M_{[2]}\psi^{-1}(A \times B \times I \times J) = \breve{M}_{[2]}(A \times B \times J) \, \ell(I)
  \end{equation*}  
  for all $A,B \in \BC(\XC)$ and $I,J \in \IC$ regardless whether $\TC=\RR$ or $\TC=\ZZ$. 
  
  The required symmetry of $\breve{M}_{[2]}$ is seen as follows. By construction,
  \begin{equation} \label{eq:2facmom_sym}
    M_{[2]}(A \times B \times I \times J) = M_{[2]}(B \times A \times J \times I).
  \end{equation}
  Writing $\varphi \colon \TC^2 \to \TC^2$, $(r,s) \mapsto (r,s-r)$ and $\sigma \colon \TC^2 \to \TC^2$, $(r,s) \mapsto (s,r)$, we obtain
  \begin{align*}
    \breve{M}_{[2]}(A \times B \times J) &= M_{[2]}(A \times B \times \varphi^{-1}((0,1] \times J)) \\
    &= M_{[2]}(B \times A \times \sigma^{-1} (\varphi^{-1}((0,1] \times J))) \\
    &= M_{[2]}(B \times A \times \varphi^{-1}((0,1] \times (-J))) \\
    &= \breve{M}_{[2]}(B \times A \times (-J)),
  \end{align*}
  where the third equality follows by the joint translation invariance $M_{[2]}(\chi_t, \chi_t)^{-1} = M_{[2]}$. More precisely, let $J = (a,b]$ for $a,b \in \TC$, $a<b$. By additivity, it is enough to consider the case $b-a \leq 1$. Write $C(r,s) = (r,r+1] \times (s,s+1]$ and $S=\{(r,s) \in \TC \mvert a < r-s \leq b \}$. Then 
  \begin{align*}
    \sigma^{-1} (\varphi^{-1}((0,1] \times J)) &= \{(r,s) \in \TC^2 \mvert 0 < s \leq 1, a < r-s \leq b \} \\
    &= (S \cap C(a,0)) \, \dot{\cup} \, (S \cap C(a+1,0))
  \end{align*}
  and 
  \begin{align*}
    \varphi^{-1}((0,1] \times (-J)) &= \{(r,s) \in \TC^2 \mvert 0 < r \leq 1, a < r-s \leq b \}\\
    &= (S \cap C(0,-a)) \, \dot{\cup} \, (S \cap C(0,-a-1)). %
  \end{align*}
  Since
  \begin{align*}
    (\tau_a,\tau_a)^{-1}(S \cap C(a,0))) &= (S \cap C(0,-a)) \quad \text{and} \\
    (\tau_{a+1},\tau_{a+1})^{-1}(S \cap C(a+1,0)) &= (S \cap C(0,-a-1)),
  \end{align*}
  the third equality above follows from $M_{[2]}(B \times A \times (\tau_x,\tau_x)^{-1}(\cdot)) = M_{[2]}(B \times A \times \cdot \,)$.
\end{proof}

\begin{proof}[Proof of \Cref{prop:Xi_Var_vs_Cov}]
By combining \eqref{eq:2facmom} and \eqref{eq:faccovm} (for $\gamma_2$ rather than $\gamma_{[2]}$), we obtain
\begin{equation*}
  \Var (\Xi(A \times (0,t])) = \gamma_{[2], A}( (0,t]^2) + M_1(A \times (0,t])
    =\gamma_{[2], A}((0,t]^2) + t \mu(A).
\end{equation*}

Let again $\varphi \colon \TC^2 \to \TC^2$, $(r,s) \mapsto (r, s-r)$, and write 
\begin{align*}
  B_t := \varphi((0,t]^2) &= \{(r,s) \in \TC^2 \mvert 0 < r \leq t, -r < s \leq t-r \} \\
  &= \{(r,s) \in \TC^2 \mvert \max\{0,-s\} < r \leq \min\{t,t-s\}, -t < s < t \}.
\end{align*}
Then
\begin{align*}
  \gamma_{[2], A}((0,t]^2) &= \gamma_{[2], A}( \varphi^{-1}(B_t)) \\
  &= \left( \ell \otimes \breve{\gamma}_{[2], A}\right)(B_t) \\
  &= \int_{(-t,t)} \int_{(0 \vee (-s),\,(t-s) \wedge t]}  \ell(dr) \; \breve{\gamma}_{[2], A}( ds) \\
  &= \int_{(-t,t)} (t-|s|)  \breve{\gamma}_{[2], A}( ds) \\
  &= t \breve{\gamma}_{[2], A}(\{0\}) + 2 \int_{(0,t)} (t-s) \; \breve{\gamma}_{[2], A}(ds),
\end{align*}
where the last line follows from the symmetry of the measure $\breve{\gamma}_{[2],A}$.
\end{proof}

\begin{corollary} \label{cor:Xi_Var_vs_Cov}
Let $\Xi$ be a weakly time-stationary $L^2$ point process. Suppose that there is an $A \in \BC(\XC)$ such that $\breve{\gamma}_{[2],A}$ is a non-negative measure and therefore $\breve{\gamma}_{[2],A} = \breve{\gamma}^+_{[2],A}$.
Then, for any $t \in \tgz$ and $\alpha \in [0,1]$,
\begin{equation*}
  2  \ell\bigl((0,(1-\alpha) t]\bigr) \, \breve{\gamma}_{[2],A}((0,\alpha t])\leq  \Var \left( \Xi(A \times (0,t]) \right) - t (\mu(A) + \breve \gamma_{[2],A}(\{0\})).%
  \end{equation*}

\end{corollary}

\begin{proof}[Proof of \Cref{cor:Xi_Var_vs_Cov}]
Note that, for any $\alpha \in [0,1]$,
  \begin{equation*}
    U := \{(r,s) \mvert r \in (0,t-s], s \in (0,t) \} \supseteq (0, (1-\alpha) t] \times (0, \alpha t].
  \end{equation*}
  Therefore, by non-negativity of $\breve \gamma_{[2],A}$, we infer that 
  \begin{align*}
    \int_{(0,t)} (t-s) \; \breve{\gamma}_{[2],A}(ds) = \int_{(0,t)} (t-s) \; \breve{\gamma}^+_{[2],A}(ds) &= (\ell \otimes \breve{\gamma}^+_{[2],A})(U) \\[-1.5mm]
    &\geq (\ell \otimes \breve{\gamma}^+_{[2],A})((0, (1-\alpha) t] \times (0, \alpha t]).
  \end{align*}
Utilizing \Cref{prop:Xi_Var_vs_Cov}, we arrive at the desired bound.
\end{proof}

\begin{proof}[Proof of \Cref{prop:PVar_vs_RCov}]
$(i)$ 
Assume that \ref{ass:Xi_mixing_cov_beta} is met. 
Observe that we have, for any $A \in \BC(\XC)$,
  \begin{equation*} %
    \int_{(0,t)} (t-s) \; \breve{\gamma}_{[2],A}(ds) \leq \int_{(0,t)} (t-s) \; \breve{\gamma}^{+}_{[2],A}(ds) \leq t\, \breve{\gamma}^{+}_{[2],A}((0,t)).
  \end{equation*}
Hence, invoking \Cref{prop:Xi_Var_vs_Cov}, we obtain since $t\breve \gamma_{[2],A}(\{0\}) \leq 2t\breve \gamma_{[2],A}^+(\{0\})$ and since $\breve \gamma_{[2],A}^+$ is a (non-negative) measure for $t \geq t_0$ that 
\begin{align*}
	\Var \left( \Xi(A \times (0,t]) \right) - t \mu(A) \leq t \breve{\gamma}^{+}_{[2],A}(\{0\}) + 2 t \, \breve{\gamma}^{+}_{[2],A}((0,t)) \leq  2 t \, \breve{\gamma}^{+}_{[2],A}([0,t]).
\end{align*}
Therefore, for any finite partition $(A_i)_{i = 1, \dots, N}$ of $\XC$, we have
$$
	\sum_{i=1}^N \Var \left( \Xi(A_i \times (0,t]) \right) - t \mu(A) \leq 2 t \sum_{i=1}^N  \breve\gamma_{[2], A_i}^+([0,t]) \leq 2 \kappa' t^{1+\beta}.
$$

\noindent    
$(ii)$ Assume that \ref{ass:Xi_mixing_Var_beta} holds. 
Since we assume non-negative correlations only, by \Cref{rem:RCov_XC}$(ii)$, it suffices to bound $\breve\gamma_{[2], \XC}([0,t])$.
Recall that either $\TC = \RR$ or $\TC = \ZZ$. In the former case we consider general $t\in \TC_{>0}$ with $t\geq t_0$
 whereas in the latter case we limit ourselves to even $t$. From \Cref{cor:Xi_Var_vs_Cov}, 
 choosing $\alpha = 1/2$, we obtain for both cases, $\TC = \RR$ and $\TC=\ZZ$, that $2\ell((0,t/2])= t$ and consequently  
\begin{align*}
       t \breve \gamma_{[2], \XC}([0,t/2]) = 2 \frac{t}{2} \breve \gamma_{[2], \XC}((0,t/2]) + t\breve \gamma_{[2],\XC}(\{0\})  \leq \Var(\Xi(\XC\times (0,t])) - t m_{\mu} \leq \kappa t^{1+\beta}.
    \end{align*}
After the change of variables, $t'\coloneqq t/2$, we obtain for any $t' \in \TC_{>0}$ with $t' \geq t_0/2$ that
\begin{align*}
       \breve \gamma_{[2], \XC}([0,t']) \leq \kappa 2^\beta (t')^{\beta}. & \qedhere
    \end{align*}

\end{proof}

\begin{proof}[Proof of \Cref{prop:max_cov_var_rate}]

Let $(A_i)_{i = 1}^{N}$ be an arbitrary finite partition of $\XC$ into measurable sets and observe that 
\begin{align*}
  \sum_{i=1}^N \Var(\Xi(A_i \times (0,t])) &\leq \sum_{i=1}^N \EE[\Xi(A_i \times (0,t])^2] \\
  &\leq \EE\left[\sum_{i=1}^N \Xi(A_i \times (0,t])^2\right] \\
  &\leq \EE\left[\left(\sum_{i=1}^N \Xi(A_i \times (0,t])\right)^2\right] \\
  &= \EE[\Xi(\XC \times (0,t])^2],
\end{align*}
where the third inequality follows from the elementary inequality $(\sum_{i = 1}^{N} k_i^2) \leq \left(\sum_{i = 1}^{N} k_i\right)^2$ for non-negative integers $k_1, \dots, k_N\in \NN_0$. To bound the right-hand side in the above display, define $T = \lceil t \rceil$ and set $\tilde \Xi_k \coloneqq \Xi(\XC \times (k-1, k])$ for $k = 1, \dots, T$. This way, it follows  that (deterministically) $\Xi(\XC \times (0,t]) \leq \sum_{k=1}^T \tilde \Xi_k$. Therefore, using the Cauchy--Schwarz inequality and weak time-stationarity of $\Xi$, which implies  $\EE[\tilde \Xi_k ^2] = \EE[\tilde \Xi_1 ^2]$, we infer 
\begin{align}\label{eq:bound_Xi_secondMoment}
	\EE[\Xi(\XC \times (0,t])^2] \leq \EE\left[\left(\sum_{k=1}^T \tilde \Xi_k \right)^2 \right] &= \sum_{k, j=1}^T \EE[ \tilde \Xi_k  \tilde \Xi_j ] \leq  \sum_{k,j=1}^T \sqrt{ \EE[ \tilde \Xi_k^2]\EE[ \tilde \Xi_j ^2]} = T^2 \EE[\tilde \Xi_1^2].
\end{align}
Upper bounding $T = \lceil t \rceil$ by $2t$ since $t\geq 1$, we arrive at the bound
$$
	\sum_{i=1}^N \Var(\Xi(A_i \times (0,t])) \leq 4 t^2 \EE[(\Xi(\XC \times (0,1])^2].
$$

Moreover,  recalling the definition of $\breve\gamma_{[2]}$ from \eqref{eq:breve_gamma_def}, we have for any $A, B \in \BC(\XC)$ and $S\in \BC(\TC)$ that $$\breve \gamma^+_{[2]}(A\times B\times S) \leq \breve M_{[2]}(A\times B\times S) + \mu(A)\mu(B) \ell(S).$$
Further,  by \eqref{eq:rfmm_explicit} with $T \coloneqq (0,t]$ and $t> 1$ it holds that $t/2\leq \ell(T)\leq t$ and for the function $\psi(x,y,r,s)\coloneqq (x,y,r,s-r)$ 
that \begin{align*}
    \breve M_{[2]}(\XC\times \XC\times [0,t]) &= \frac{1}{l(T)} M_{[2]}\psi^{-1}(\XC\times \XC\times (0,t] \times [0,t])\\
    &\leq \frac{2}{t} M_{[2]}(\XC\times \XC\times (0,t]\times (0,2t]) \\
    &\leq \frac{2}{t} \EE[\Xi(\XC \times (0,t])\cdot \Xi(\XC \times (0,2t])]\leq 16 t \EE[(\Xi(\XC \times (0,1])^2],
\end{align*}
where the last inequality follows analogously to \eqref{eq:bound_Xi_secondMoment}. 
Consequently, we obtain \begin{align*}
    \breve \gamma_{[2]}^+(\XC\times \XC\times [0,t]) &\leq 16 t \EE[(\Xi(\XC \times (0,1])^2] +  \ell([0,t]) \mu(\XC)^2 \\&\leq t\left( 16 \EE[(\Xi(\XC \times (0,1])^2]  + 2 \mu(\XC)^2\right),
\end{align*}
for the last inequality we used that $\ell([0,t])= 1+\lfloor t \rfloor\leq 2t$ for both $\TC = \RR$ and $\TC = \ZZ$. The assertion now follows at once via Remark \ref{rem:RCov_XC}$(i)$ by selecting $\kappa \coloneqq \left(16 \EE[(\Xi(\XC \times (0,1])^2] +  2\mu(\XC)^2\right)$
\qedhere

\end{proof}

\subsection{Proofs for properties of examples of ST point processes}

\begin{proof}[Proof of \Cref{prop:E_and_Var_for_PCP}]
  Recall that we denote by $\IC$ the set of bounded left-open and right-closed intervals of $\TC=\RR$.
  
  For the claimed well-definedness, we have to check that $\Xi(A \times I)$ is almost surely finite for every $A \in \BC(\XC)$ and every $I \in \IC$. This follows directly from the finiteness of $\EE[\Xi(A \times I)]$ shown below, which in turn uses only the local finiteness of $M_1^{cl}$ and $\EE \Eta$, but not of $\Xi$.
To show time-stationarity, we first note that, in general, in order to show equality of point process distributions, it suffices by \citet[Proposition 9.2.III]{daley2007introduction},
  to check equality on sets of the form 
  \begin{align*}
    \bigl\{\xi \in \mathfrak{N}(\XC \times \RR) \bigm| \xi(A_j \times I_j) = k_j \ \text{for $1 \leq j \leq l$}\bigr\}
  \end{align*}
  for any $l \in \NN$, $A_1,\ldots, A_l \in \BC(\XC)$, $I_1,\ldots,I_l \in \IC$ and $k_1,\ldots,k_l \in \NN_0$.
  Write $\tilde{S}_i := S_i+s$ and $\tilde{\Eta} := \sum_{i=1}^{\infty} \delta_{(Y_i,\tilde{S}_i)} = T_s(\Eta) \eqinlaw \Eta$ (by time-stationarity of the cluster center process $\Eta$). 
  Time-stationarity then follows (after checking that $\Xi$ is a well-defined point process) by
  \begin{align*}
    \PP \bigl( \Xi(A_j \times (I_j-s)) = k_j \, \forall j \bigr) &= \EE \left[\PP\bigl( \textstyle\sum_{i=1}^{\infty} \Xi_i (A_j \times (I_j-s)) = k_{j} \ \forall j \bigm| \Eta \bigr)\right] \\[1mm]
    &= \EE \Bigl[\sum_{(k'_{ij})} \PP \bigl( \Xi_i(A_j \times (I_j-s)) = k'_{ij} \ \forall i,j \bigm| Y_i,S_i \bigr)\Bigr] \\
    &= \EE \Bigl[\sum_{(k'_{ij})} \PP \bigl( \Xi_i(A_j \times I_j) = k'_{ij} \ \forall i,j \bigm| Y_i,\tilde{S}_i \bigr)\Bigr] \\
    &= \EE \left[ \PP\bigl( \textstyle\sum_{i=1}^{\infty} \Xi_i (A_j \times I_j) = k_j \ \forall j \bigm| \tilde{\Eta} \bigr)\right] \\[1mm]
    &=\PP \bigl( \Xi(A_j \times I_j) = k_i \, \forall j \bigr),
  \end{align*}
  where the sums are taken over all matrices $(k'_{ij})_{i \in \NN, 1 \leq j \leq l}$ with entries $k'_{ij} \in \NN_0$ and column sums $\sum_{i=1}^{\infty} k'_{ij} = k_j$ for all $j \in \{1,\ldots,l\}$. The third equality follows by \eqref{eq:timehomo_kernel} and the last equality by $\tilde{\Eta} \eqinlaw \Eta$.

  The considerations for the moment measures are generalizations from \citet[Section~6.3]{daley2003introduction}. For self-containedness, we spell out the full arguments.

  For $A \in \BC(\XC)$ and general $I \in \IC$, we have
  \begin{align*} 
  \EE [ \Xi(A\times I)] &= \EE \left[ \EE [\Xi (A\times I) \mvert \Eta ] \right] \\
  &= \EE \left[\sum_{i=1}^\infty \EE \left[ \Xi_i(A \times I) \mvert Y_i,S_i \right]\right] \\
  &= \EE \left[\int_{\XCbar \times \RR} M^{cl}_1(A \times I) \mvert (y,s)) \; \Eta(dy \, ds)\right] \\
  &= \int_{\XCbar} \int_{\RR} M^{cl}_1(A \times (I-s)) \mvert (y,0)) \; ds \; \Lambda(dy) \\
  &= \int_{\XCbar} \int_{\RR} \int_{\RR} \mathds{1}(r \in I-s) \, M^{cl}_1(A \times dr \mvert (y,0))  \; ds \; \Lambda(dy) \\
  &= \int_{\XCbar} \int_{\RR} \int_{I-r} ds \, M^{cl}_1(A \times dr \mvert (y,0))  \; \Lambda(dy) \\
  &= |I| \int_{\XCbar} M^{cl}_1(A \times \RR_+ \mvert (y,0))  \; \Lambda(dy),
  \end{align*}
  using \eqref{eq:timehomo_kernel} for the third line and Campbell's formula \citet[Equation~(9.5.2)]{daley2007introduction} for the fourth line. Furthermore, we have used Tonelli's theorem for the second last and \eqref{eq:causal} for the last line. We have used local finiteness but not total finiteness of $M^{cl}_1(\cdot \mvert (y,0))$ for this argument. Thus, \eqref{eq:mom1_poisson_cluster} follows (with the possibility that both sides of the equation are $\infty$) and the expectation measure of $\Xi$ is locally finite if and only if $\int_{\XCbar} M^{cl}_1(\cdot \mvert (y,0)) \, \Lambda(dy)$ is totally finite.

  We do a similar computation for the second factorial moment measure of $\Xi$. Splitting up $\Xi^{[2]}$ into contributions coming from a single cluster and contributions from combining different clusters, we obtain for $A \in \BC(\XC)$ and $I,J \in \IC$
  \begin{align}
  &\EE \bigl[ \Xi^{[2]}(A \times A \times I \times J) \bigr] \notag\\
  &\hspace*{1mm} = \EE \bigl[ \EE \bigr[\Xi^{[2]} \bigl(A \times A \times I \times J \bigr) \bigm| \Eta \bigr] \bigr] \notag\\
  &\hspace*{1mm} = \EE\left[\sum_{i=1}^\infty \EE \left[ \Xi_i(A \times I) \, \Xi_i(A \times J) \bigm| Y_i,S_i \right]\right] - \EE\left[\sum_{i=1}^\infty \EE \left[ \Xi_i(A \times (I \cap J))  \bigm| Y_i,S_i \right]\right]  \notag\\
  &\hspace*{10mm} {} + \EE\left[\sum_{\substack{i,j=1 \\ i\neq j}}^\infty \EE \left[ \Xi_i(A \times I) \, \Xi_j(A \times J) \mid Y_i,S_i,Y_j,S_j \right]\right] \notag\\
  &\hspace*{1mm} = \int_{\XCbar \times \RR} M^{cl}_{[2],A}\bigl(I \times J \bigm| (y,s) \bigr) \; \Lambda(d y ) \; \dif s \notag\\
  &\hspace*{10mm} {} + \int_{\XCbar \times \RR} \int_{\XCbar \times \RR} M^{cl}_{1,A}\bigl(I \bigm| (y, s) \bigr) \, M^{cl}_{1,A}\bigl(J \bigm| (y', s') \bigr) \; \Lambda(d y) \dif s \; \Lambda(d y') \dif s' \notag\\
  &\hspace*{1mm} = \int_{\XCbar \times \RR} M^{cl}_{[2],A}\bigl(I \times J \bigm| (y,s) \bigr) \; \Lambda(d y ) \; \dif s + M_{1,A}(I) M_{1,A}(J) \in [0,\infty]. \label{eq:f2mom_poisson_cluster}
  \end{align}
  This computation uses that the term subtracted in the third line is finite (which we know since $\Xi$ is $L^1$ by the previous part) but makes no assumptions on any of the other terms.

  In an abuse of notation we refer to the left-hand side as $M_{[2], A}(I \times J)$ even in case the result is infinite. 
On the right-hand side of \eqref{eq:f2mom_poisson_cluster}, the first summand may be infinite as well on this stage, whereas the second summand is always finite due to \eqref{eq:L1_cond_poisson_cluster} by the previous part. Note that the above definition on rectangles fully specify unique measures if we just interpret $M_{[2], A}$, $\gamma_{[2],A}$ by their generalized definitions using $\EE$ and the measures $\Xi^{[2]}$ and~$\Xi$. By definition, we have
  \begin{align*}
  \gamma_{[2],A}(I \times J) = \int_{\XCbar \times \RR} M^{cl}_{[2],A}\bigl(I \times J \bigm| (y,s) \bigr) \; \Lambda(d y ) \dif s.
  \end{align*}
By time-stationarity, we obtain that $\gamma_{[2],A}$ is invariant under joint shifts. Thus, we can reduce it, using the function $\varphi \colon \RR^2 \to \RR^2$, $(u,v) \mapsto (u, v-u)$, to obtain, for any bounded interval $J \subset \RR$,
  \begin{align}
    \breve{\gamma}_{[2],A}(J) &= \gamma_{[2],A}\bigl(\varphi^{-1}((0,1] \times J) \bigr) \notag\\
    &= \int_{\XCbar \times \RR} M^{cl}_{[2],A}\bigl(\varphi^{-1}((0,1] \times J) \bigm| (y,s) \bigr) \; \Lambda(d y ) \; ds \notag\\
    &= \int_{\XCbar \times \RR} M^{cl}_{[2],A}\bigl(\varphi^{-1}((-s,-s+1] \times J) \bigm| (y,0) \bigr) \; \Lambda(d y ) \; ds \notag\\
    &= \int_{\XCbar} \int_{\RR} \int_{\RR} \mathds{1}(u \in (-s, -s+1]) \; M^{cl}_{[2],A}\bigl(\varphi^{-1}(du \times J) \bigm| (y,0) \bigr) \; ds \; \Lambda(d y )  \notag\\
    &= \int_{\XCbar} \int_{\RR} \int_{-u}^{-u+1} ds \; M^{cl}_{[2],A}\bigl(\varphi^{-1}(du \times J) \bigm| (y,0) \bigr) \; \Lambda(d y )  \notag\\
    &= \int_{\XCbar} M^{cl}_{[2],A}\bigl(\RR_+^2 \cap \varphi^{-1}(\RR \times J) \bigm| (y,0) \bigr) \; \Lambda(d y) ,\label{eq:rfcov_poisson_cluster}
  \end{align}
  using $\varphi^{-1}((0,1] \times J) - (s,s) = \varphi^{-1}((-s,-s+1] \times J)$ for the third line and that $M^{cl}_{[2],A}(\cdot \mvert (y,0))$ is concentrated on $\RR_+^2$ by~\eqref{eq:causal} for the last line.

  Since the right-hand side of \eqref{eq:rfcov_poisson_cluster} is always non-negative, $\breve{\gamma}_{[2],A} = \breve{\gamma}^+_{[2],A}$ is a proper (non-negative) measure %
  and \eqref{eq:rf_covm_poisson_cluster} follows. 
  
It remains to check that local finiteness of $\EE \Xi^{[2]}$ is equivalent to \eqref{eq:L2_cond_poisson_cluster}. %
By time-stationarity, this is the case if and only if $\EE [\Xi^{[2]}(A^2 \times (0,t] \times [0,t])] < \infty$ for all $A \in \BC(\XC)$ and all $t \geq 0$. By \eqref{eq:f2mom_poisson_cluster} and $M_{1,A}([0,t])<\infty$ due to \eqref{eq:L1_cond_poisson_cluster}, this holds if and only if $$\int_{\XCbar \times \RR} M^{cl}_{[2],A}((0,t] \times [0,t] \mvert (y,s)) \; \Lambda(d y ) \; \dif s < \infty.$$
  This last statement is seen to be equivalent to \eqref{eq:L2_cond_poisson_cluster} by using
  \begin{align*}
      M^{cl}_{[2],A}\bigl(\varphi^{-1}((0,t/2] \times [0,t/2]) \bigm| (y,s) \bigr)
      &\leq M^{cl}_{[2],A}\bigl((0,t] \times [0,t] \bigm| (y,s) \bigr) \\
      &\leq M^{cl}_{[2],A}\bigl(\varphi^{-1}((0,t] \times [-t,t]) \bigm| (y,s) \bigr)
  \end{align*}
  and applying \eqref{eq:rfcov_poisson_cluster} for $(0,t]$ rather than $(0,1]$ (using that $M^{cl}_{[2], A}$ is a second factorial moment measure and hence symmetric), which only adds a factor of $t$ on the right-hand side. 
\end{proof}

\begin{proof}[Proof of \Cref{prop:hawkes_moments}]
  According to the construction provided right before Assumption~\ref{ass:subprob_temporal_domination}, $\Xi$ is a spatio-temporal Poisson cluster process, as it is straightforward to see that the time-shift equivariance \eqref{eq:timehomo_kernel} and the causal property \eqref{eq:causal} are satisfied by the fact that the infection intensity $g$ depends only on the time difference and is zero if this difference is negative. We aim at applying \Cref{prop:E_and_Var_for_PCP} for $\XCbar = \XC$, for which we need to control the first and second factorial moment measure of a generic cluster initiated by a single point. 

  Let $(y,s) \in \XC \times \RR_+$ and write $\Xi' = \sum_{k=1}^{\infty} Z_k$ for the total progeny of $Z_0 = \delta_{(y,s)}$ obtained by the usual branching construction, not counting the parent point $(y,s)$.
  We may then write $\Xi' = \sum_{i=1}^{N_1} \Xi'_i$, where, given $Z_1$, the $\Xi'_i$ are the independent subclusters initiated by each of the points of $Z_1 = \sum_{i=1}^{N_1} \delta_{(X_{1i},T_{1i})}$, each including its respective parent point. 

  Noting that $Z_1 = \Psi_{01} \sim \pop(g(y,x,t-s) \, dx \, dt)$, we obtain for the expectation measure $M^{cl}_1(\,\cdot \mvert (y,s))$ of the whole cluster $\Xi'+\delta_{(y,s)}$ initiated by $(y,s)$, letting $A \in \BC(\XC)$ and $I \in \IC$, 
  \begin{equation} \label{eq:hawkes_expection_recursion}
  \begin{aligned}
    M_1^{cl}(A \times I \mvert (y,s)) - \mathds{1}(y \in A, s \in I ) &= \EE[\Xi'(A \times I)] \\
    &= \EE \bigl[ \EE \bigl[\Xi'(A \times I) \bigm| Z_1 \bigr] \bigr] \\
    &= \EE \sum_{i=1}^{N_1} \EE[\Xi'_i(A \times I) \mvert (X_{1i},T_{1i})] \\
    &= \EE \int_{\XC \times \RR_+} M_1^{cl}(A \times I \mvert (x,t)) \; \Psi_{01}(\dif x \dif t) \\
    &= \int_{\XC \times \RR_+} M_1^{cl}(A \times I \mvert (x,t)) \, g(y,x,t-s) \dif x \dif t,
    \end{aligned}
  \end{equation}
  where we used Campbell's formula for the last line. 

  We can now repeatedly apply \eqref{eq:hawkes_expection_recursion} to itself to obtain a candidate for a solution. Since $\int_{\XC} \int_{\RR_+} g(y,x,t-s) \, g^{\ast(k-1)}(x,x',t'-t) \, dx \, dt = g^{\ast k}(y,x',t'-s)$, we get an infinite series %
  \begin{equation} \label{eq:hawkes_expection_recursion_sol}
  \begin{aligned}
      M_1^{cl}(A \times \RR_+ \mvert (y,s)) 
      &= \mathds{1}(y \in A) \begin{aligned}[t]
      		&+\int_{A \times \RR_+} \, g(y,x,t-s) \; dx \; dt \\
      &+ \int_{A \times \RR_+} \int_{\XC \times \RR_+} g(y,x,t-s) \, g(x,x',t'-t) \; dx \; dt \; dx' \; dt' \\[0.5mm]
      &+ \ldots 
      \end{aligned} \\
      &=  \int_{A \times \RR_+} \sum_{k=0}^{\infty} g^{*k}(y,x,t-s) \; dx \; dt.
  \end{aligned}
  \end{equation}
  The infinite series of successive spatio-temporal convolutions $(x,t) \mapsto \sum_{k=0}^{\infty} g^{*k}(y,x,t-s)$ is a well-defined $L_1$ function since
  \begin{align} \label{eq:hawkes_expectation_finite}
    \int_{\XC \times \RR_+} \sum_{k=0}^{\infty} g^{*k}(y,x,t-s) \; dx \; dt
    = \sum_{k=0}^{\infty} \int_{\XC \times \RR_+} g^{*k}(y,x,t-s) \; dx \; dt \leq \sum_{k=0}^{\infty} \nubar^k = \frac{1}{1-\bar{\nu}} < \infty
  \end{align}
due to $\nubar = \sup_{y \in \XC} \int_{\XC \times \RR_+} g(y,x,h) \, dx \, dh < 1$. It is clear from the construction that there can be no other solution, and it is immediately checked that the right-hand side of \eqref{eq:hawkes_expection_recursion_sol} solves the integral equation~\eqref{eq:hawkes_expection_recursion}. 
  Setting $s=0$ and integrating over $y \in \XC$ w.r.t.\ the finite measure $\Lambda(dy) = \lambda(y) \, dy$ yields that \eqref{eq:L1_cond_poisson_cluster} is satisfied and thus that~\eqref{eq:mom1_hawkes} holds (using~\eqref{eq:hawkes_expectation_finite} for the last inequality) once we have established also \eqref{eq:L2_cond_poisson_cluster}.

For bounding the factorial second moment measure $M^{cl}_{[2],\XC}(\,\cdot \mvert (y,0))$, we can make direct use of the known form of the factorial second moment measure of a purely temporal stationary Hawkes process with excitation function $\gbar_{\temp}$. In each generation $k$ of the general branching construction of the progeny $\Xi'$ of $Z_0=\delta_{(y,0)}$, we add up independent $\pop(g(X_{ki},x,t-T_{ki}) \, dx \, dt)$-processes $\Psi_{ki}$, $1\leq i \leq N_k$. The time-marginal of such a process is $\Psi_{ki}(\XC \times \cdot) \sim \pop\bigl((\int_{\XC} g(X_{ki},x,t-T_{ki}) \, dx) \, dt)\bigr)$. Using Assumption~\ref{ass:subprob_temporal_domination}, we can couple $\Psi_{ki}(\XC \times \cdot)$ with a temporal Poisson process $\bar{\Psi}_{ki} \sim \pop(\gbar_{\temp}(t-T_{ki}) \, dt)$ such that $\Psi_{ki}(\XC \times \cdot) \leq \bar{\Psi}_{ki}$, $1\leq i \leq N_k$. (This coupling can be achieved by adding, possibly on an enlarged probability space, an independent Poisson process whose intensity is the pointwise difference of the intensities of the two processes.) In addition, we typically require further offspring processes $\bar{\Psi}_{ki}$, $N_k+1 \leq i \leq \bar{N}_k$, say, of parent points in $\bar{\Psi}_{k-1,i}$ that are not present in $\Psi_{k-1,i}(\XC \times \cdot)$. In this way, we obtain the total progeny $\bar{\Xi}'$ of $\bar{Z}_0=\delta_{0}$ based on the excitation function $\gbar_{\temp}$ as $$\bar{\Xi}' := \sum_{k=0}^{\infty} \sum_{i=1}^{\bar{N}_k} \bar{\Psi}_{ki} \geq \sum_{k=0}^{\infty} \sum_{i=1}^{N_k} \Psi_{ki}(\XC \times \cdot) = \Xi'(\XC \times \cdot).$$

Denoting the first and second factorial moment measures of $\bar{\Xi}'+\delta_{0}$ by $\Mbar^{cl}_{1}$ and $\Mbar^{cl}_{[2]}$, respectively (without explicit conditioning on 0), we can bound the second factorial moment measure of the whole cluster $\Xi'+\delta_{(y,0)}$, for $A \in \BC(\XC)$ and $I \in \IC$, as
\begin{align*}
  M^{cl}_{[2], A}(\varphi^{-1}(\RR_+ \times I) \mvert (y,0)) &\leq M^{cl}_{[2], \XC}(\varphi^{-1}(\RR_+ \times I) \mvert (y,0)) \\
  &\leq \Mbar^{cl}_{[2]}(\varphi^{-1}(\RR_+ \times I)) \\
  &= \frac{1}{1-\nubar_{\temp}} (\Mbar^{cl}_{1})^2(\varphi^{-1}(\RR_+ \times I))  - \delta^{(1)}(I)
\end{align*}
by Equation (6.3.25) in \cite{daley2003introduction}. %

Since
\begin{equation} \label{eq:domcluster_expectation_finite}
  \Mbar^{cl}_{1}(\RR_+) = \int_{\RR_+} \bar{G}_{\temp}(t) \, dt \leq \frac{1}{1-\bar{\nu}_{\temp}} < \infty
\end{equation}
is obtained if we apply \eqref{eq:hawkes_expection_recursion_sol} and \eqref{eq:hawkes_expectation_finite} to $\bar{\Xi}'+\delta_{0}$ (i.e., using a trivial spatial component), we see that \eqref{eq:L2_cond_poisson_cluster} holds.

\Cref{prop:E_and_Var_for_PCP} yields therefore that $\Xi$ is a well-defined, %
time-stationary and $L^2$ point process satisfying \eqref{eq:mom1_hawkes} and 
  \begin{align*} 
    \breve \gamma_{[2],\XC} ([0,t]) &= \int_{\XC} M^{cl}_{[2],\XC}\bigl(\varphi^{-1}(\RR_+ \times [0,t]) \bigm| (y,0) \bigr) \; \Lambda(dy) \\
    &\leq \frac{m_\lambda}{1-\nubar_{\temp}} \Bigl((\Mbar^{cl}_{1})^2(\varphi^{-1}(\RR_+ \times [0,t])) - 1 \Bigl) \\
    &= \frac{m_\lambda}{1-\nubar_{\temp}} \, \biggl( \int_0^t \int_0^{\infty} \bar{G}_{\temp}(s) \, \bar{G}_{\temp}(s+r) \; ds \; dr - 1\biggr),
  \end{align*}
which by \eqref{eq:domcluster_expectation_finite} is bounded in $t$, and by \Cref{rem:RCov_XC}$(ii)$ yields that Assumption~\ref{ass:Xi_mixing_cov_beta} holds with $\beta=0$.
\end{proof}

\section{Proofs and background for the lower bounds on empirical KRD}\label{sec:proofs_LowerBounds}
\subsection{Proofs and background for lower bounds for plug-in estimators} \label{sec:MAD_lb}

\begin{proof}[Proof of \Cref{lem:lowerBoundPlugIn}]
  For Assertion $(i)$ we note from \Cref{prop:kr_properties}$(iv)$ that, deterministically, 
  \begin{align*}
    \KR_{p,C}(\hat \mu_t, \mu) \geq \frac{C}{2}\left|m_{\hat \mu_t} - m_\mu\right|^{1/p}.
  \end{align*}
  Taking expectations yields the inequality.

  To show Assertion $(ii)$, we define $K\coloneqq m_{\mu}\vee m_{\hat \mu_t}$ and invoke \Cref{prop:linkOT}$(i)$ with augmented measures $\tilde \mu_{t,K}, \tilde \mu_K$, which yields
  \begin{align*}
    \KR_{p,C}^p(\hat \mu_t, \mu) &= \OT_{\tilde d^p_{p,C}}(\tilde \mu_{t,K}, \tilde \mu_K)\\
    &= \inf_{\pi\in \Pi_{=}(\tilde \mu_{t,K}, \tilde \mu_K)}\int_{\tilde \XC\times \tilde \XC} \tilde d^p_{p,C}(x,x')\dif\pi(x,x')\\
    &\geq\inf_{\pi\in \Pi_{=}(\tilde \mu_{t,K}, \tilde \mu_K)} \left(\frac {C^p}{2}\wedge \delta^p\right) \pi( \XC_1\times (\XC_2\cup\{\mathfrak{d}\}) \cup \XC_2 \times (\XC_1\cup\{\mathfrak{d}\}))\\
    &= \left(\frac {C^p}{2}\wedge \delta^p\right)\left( |\tilde \mu_{t,K}(\XC_1) - \tilde\mu_{K}(\XC_1)| + |\tilde \mu_{t,K}(\XC_2) - \tilde\mu_{K}(\XC_2)| \right).
  \end{align*}
  Taking the $p$-th roots and taking expectations yields the claim.

   Finally, for Assertion $(iii)$ we invoke \citet[Proposition 6]{weed2019sharp} which asserts for $\mu\in \MCsX$ if $\NC(\epsilon, \supp(\mu), d\wedge C)\geq \epsilon^{-\alpha}$ for all $\epsilon< \epsilon'$ that any finitely supported measure $\nu$ with $n\geq {\epsilon'}^{-\alpha}$ support points and $m_{\nu} = m_{\mu}$ fulfills
  \begin{align*}
    \KR_{p,C}(\mu, \nu) = \OT^{1/p}_{d^p\wedge C^p}(\mu, \nu) &\geq m_{\mu}^{1/p} 4^{-1} n^{-1/\alpha}.
  \intertext{
  Moreover, when $m_{\nu} < m_{\mu}$, then the measure $\tilde \nu \coloneqq \nu  + \delta_{\mathfrak{d}}(m_{\mu}-  m_{\nu})$ on $\tilde \XC$ has $n+1$ support points. Further, if $\NC(\epsilon, \supp(\mu), \tilde d_{p,C})\geq \epsilon^{-\alpha}$ for all $\epsilon< \epsilon'$, then it follows for $n\geq {\epsilon'}^{-\alpha}$ that }
    \KR_{p,C}(\mu, \nu) = \OT^{1/p}_{\tilde d_{p,C}^p}(\mu, \tilde \nu  ) &\geq m_{\mu}^{1/p}4^{-1} (n+1)^{-1/\alpha}.
  \intertext{
  Likewise, when $m_{\mu}>m_{\nu}$ it follows for $\tilde \mu \coloneqq \mu + \delta_{\mathfrak{d}}(m_{\nu} - m_{\mu})$ in case $\NC(\epsilon, \supp(\mu), \tilde d_{p,C})\geq \epsilon^{-\alpha}$ for all $\epsilon< \epsilon'$ also that $\NC(\epsilon, \supp(\tilde \mu), \tilde d_{p,C})\geq \epsilon^{-\alpha}$. Hence, under this assumption, we infer for $n\geq {\epsilon'}^{-\alpha}$ that }
    \KR_{p,C}(\mu, \nu) = \OT^{1/p}_{\tilde d_{p,C}^p}(\tilde \mu, \nu) &\geq m_{\nu}^{1/p} 4^{-1} n^{-1/\alpha}.
  \end{align*}
  Meanwhile, for $n\leq {\epsilon'}^{-\alpha}$, we deduce by approximating an $n$-atomic measure by a discrete measure with $\lceil{\epsilon'}^{-\alpha}\rceil$ support points and due to continuity of the $(p,C)$-KRD under weak convergence (\Cref{prop:weak_conv_KRD}) that 
  \begin{align*}
    \KR_{p,C}(\mu, \nu) \geq (m_{\mu}\vee m_{\nu})^{1/p} \epsilon'.
  \end{align*}

  It remains to confirm that we can choose $\epsilon'$ proportional to $C$ for all three cases. Indeed, based on \Cref{lem:impactCoveringC}$(i)$ it follows for $\epsilon< C$ that $\NC(\epsilon, \supp(\mu), d) = \NC(\epsilon, \supp(\mu), d\wedge C)$. Moreover, when $\epsilon< C/2^{1/p}$, then it also follows that $\NC(\epsilon, \supp(\tilde \mu), d\wedge C) = 1+ \NC(\epsilon, \supp(\mu), d\wedge C)\geq \epsilon^{-\alpha}$. Choosing  $\epsilon' = C/2^{1/p}$ therefore yields by the above considerations the validity of the lower bound 
  \begin{align*}
      \KR_{p,C}(\mu, \nu) &\geq (m_{\mu}\vee m_{\nu})^{1/p} 4^{-1} (n+1)^{-1/\alpha}.
  \end{align*}
  Plugging in $\hat \mu_t$ for $\nu$ and taking expectations on both sides yields the assertion.
\end{proof}

\subsection{Background and construction of ST point processes in Example \ref{ex:lb_beta_positive}}

To proceed with \Cref{ex:lb_beta_positive}$(i)$, let us first introduce the Hermite rank of a function. 
\begin{definition}
	Let $Z$ be a standard normal random variable and let $\gamma(x) = (2 \pi)^{-1 / 2} e^{-x^2 / 2}x$ be the standard Gaussian density on $\RR$. With slight abuse of notation denote by $L^2(\gamma) = \{G : \EE[G(Z)^2] < \infty\}$ the $L^2$-space with Gaussian measure. It is known that $\{H_m(\cdot)/\sqrt{m!} \,\colon\, m \in \NN\}$, where $H_m$ are the (probabilists') Hermite polynomials, forms an orthonormal basis of $L^2(\gamma)$. The \emph{Hermite rank} $k$ of a function $G : \RR \to \RR$ is then defined as
	\begin{equation*}
		k = \inf\left\{m \in \NN: \EE[G(Z)H_m(Z)] = \int_\RR G(x) H_m(x) \gamma(x)\dif x  \neq 0 \right\}.
	\end{equation*}
\end{definition}

\begin{proposition}[Existence of a binary function with a given rank] \label{prop:G_exists}
Let $k \in \NN$. There exists a measurable function $G: \RR \to \{0,1\}$ with Hermite rank $k$. Moreover, for a given stationary Gaussian process $(\tau_i)_{i\in \ZZ}$ with $\Cov(\tau_1, \tau_t) \asymp t^{2H-2}$, it holds that
	\begin{equation} \label{eq:VarG_expression}
		\Var\left(\sum_{i=1}^t G(\tau_i) \right) \asymp t \,\Var(G(\tau_1)) + \frac{2}{(a+1)(a+2)} t^{(2H-2)k +2}.
	\end{equation}
\end{proposition}
\begin{proof}
	It suffices to find a function $G': \RR \to \{-1, 1\}$ which is orthogonal to $H_1, H_2, \dots, H_{k-1}$ and not orthogonal to $H_k$ with respect to $L^2(\gamma)$ since by \citet[Remark 2.3]{bai2018instability} the function $G\coloneqq (G'+1)/2 \colon \RR\to \{0, 1\}$ will then also fulfill this property. 

Consider the set of functions $(\tilde g_i)_{i = 1, \dots, k}$ on $[-\pi, \pi]$ defined by $g_i'(x) \coloneqq g_i(\tan(x))$ for $g_i(y) = (1+y^2)(H_i(y))\gamma(y)$ and note that they are continuous and bounded on $[-\pi, \pi]$ due to the exponential decay of the Gaussian kernel $\gamma$. Then, by the Hahn-Banach separation theorem \citep[Theorem 3.5]{rudin1991functionalAnalysis}, since $g_k'$ is not in the span of $\{g_1, \dots, g_{k-1}\}$, there exists a continuous linear functional $\Psi\colon L^1([-\pi, \pi])\to \RR$ such that 
\begin{align*}
  \Psi(\tilde g_i) = 0 \quad \forall i = 1, \dots, k-1, \quad \Psi(\tilde g_k) = 1.
\end{align*}
Further, since the dual space of $L^1([-\pi, \pi])$ is isomorphic to $L^\infty([-\pi, \pi])$, it follows that there exists some bounded function $\tilde G\in L^\infty([-\pi, \pi])$ such that $\Psi(f) = \int_{-\pi}^\pi f(x) \tilde G(x) \dif x$ for all $f\in L^1([-\pi, \pi])$. Next, consider the function $\tilde G' \coloneqq \tilde G/\norm{\tilde G}_{L^\infty([-\pi, \pi])}$, and note that 
\begin{align*}
  \int_{-\pi}^\pi \tilde g_i(x) \tilde G'(x) \dif x = 0 \quad \forall i = 1, \dots, k-1, \quad \int_{-\pi}^\pi \tilde g_k(x) \tilde G'(x) \dif x = 1/\|\tilde G'\|_{L^\infty([-\pi, \pi])} > 0.
\end{align*}
Next, we consider the maximization problem, 
\begin{align*}
  &\max \int \tilde g_{k} (x) \tilde G''(x) \dif x \\ &\text{subject to } \int_{-\pi}^\pi \tilde g_i(x)  \tilde G''(x) \dif x  = 0\,  \text{ for all }  i =1, \dots, k-1, \quad \|\tilde G''\|_{L^\infty([-\pi, \pi])} \leq 1.
\end{align*}
The objective is convex and continuous with respect to the weak-$*$-topology and the constraint set is convex and weak-$*$-compact by the Banach-Alaoglu-theorem. Consequently, there exists a maximizer $\tilde G''' \in L^\infty([-\pi, \pi])$ which is an extremal point of the constraint set, thus fulfills $|\tilde G'''(x)| = 1$ for Lebesgue almost all $x \in [-\pi, \pi]$, and without loss of generality we may also assume that it fulfills this relation everywhere on $[-\pi, \pi]$. Further, it holds 
\begin{align*}
  \int_{-\pi}^\pi \tilde g_k(x) \tilde G'''(x) \dif x = 1/\|\tilde G'\|_{L^\infty([-\pi, \pi])} \geq \int_{-\pi}^\pi \tilde g_k(x) \tilde G'(x) \dif x = 1/\|\tilde G'\|_{L^\infty([-\pi, \pi])}>0.
\end{align*}
Setting $G' \coloneqq \tilde G''' \circ \arctan \colon \RR \to \{-1,1\}$, we conclude via a change of variables $x = \tan(y)$ that
$$
  \int_{\RR} G'(y) H_i(y) \gamma(y)\dif y =  \int_{\RR} G'(y) \frac{g_i(y)}{1+y^2} \dif y = \int_{-\pi}^\pi \tilde G'''(x) \tilde g_i(x) \dif x \begin{cases} = 0& \text{ if } i \in\{1, \dots, k-1\}, \\ > 0 & \text{ if } i = k, \end{cases}
$$
hence $G'$ has a Hermite rank at least $k$. %
	
	The asymptotic equivalence \eqref{eq:VarG_expression} follows by combining Lemma 2.4 in \cite{bai2018instability} and \Cref{lem:sum_Var_and_cov}.
\end{proof}

\begin{proposition}[\cite{bai2018instability}] \label{prop:gaussian_clt}
	Let $(\tau_i)_{i \in \NN}$ be a stationary sequence of Gaussian random variables such that $\Cov(\tau_t, \tau_1) \asymp t^{2H - 2}$ for some $H \in (1/2, 1)$. Let $G : \RR \to \{0,1\}$ have a Hermite rank $k \geq 1$ and suppose that $(H-1)k + 1 \in (1/2, 1)$. Then, for $t \in \NN$ we have
	$$
		\frac{1}{t^{(H-1)k + 1}} \sum_{i=1}^{t}(G(\tau_i) - \EE[G(\tau_1)]) \konvD c \, \ZC_{H,k},
	$$
	where $\konvD$ denotes convergence in distribution, $c>0$ and $\ZC_{H, k}$ is the law of a well-defined $L^2$-random variable $Z_{H, k}(1)$, where $(Z_{H, k}(t))_{t>0}$ denotes the $k$-th order Hermite process with self-similarity index $H$, see \citet[Section 2.2]{tudor2023non} for its definition.
	\end{proposition}

To provide the relevant background for the point process described in \Cref{ex:lb_beta_positive}$(ii)$, we put it into the general context of \emph{delayed} (or \emph{persistent}) renewal processes (see, e.g., Section XI.4 in \cite{fellerIntroductionProbabilityTheory1971}). 
\begin{proposition} \label{prop:renewal_mean_var}
Assume that $S_2, S_3, \dots$ is a sequence of  i.i.d.\ random variables with a common cumulative distribution function $F$ and $\EE[S_2] =1/ m_\mu > 0$ and that independent of them $S_1$ is a random variable with a cumulative distribution function $G$ given by $ \ G(t) = m_\mu \int_{0}^t (1 - F(s)) \dif s$. For $t \geq 0$, we define $N_t \coloneqq 0 \vee \max\{n \in \NN : \sum_{i=1}^n S_i \leq t\}$. Then, 
	\begin{align}
	\EE[N_t] &= m_\mu t  \quad \text{and} \label{eq:delayed_renewal_expe} \\
	\Var(N_t) &= m_\mu t + 2  m_\mu  \int_{0}^t (U^0(s) -  m_\mu  s ) \dif s, \label{eq:delayed_renewal_var} 
\end{align}
where $U^0(t)= \sum_{k=1}^\infty F^{k*}(t)$. If $F$ is such that $\overline{F}(t) := 1 - F(t) \asymp t^{-\alpha} L(t)$ for some slowly varying function $L$, bounded on every compact subset of $[0, \infty)$ and $\alpha \in (0, 2)$, then there exists $t_0 > 0$ such that, for all $t \geq t_0$,
\begin{equation*}
	\Var(N_t) = m_\mu t + 2 m_\mu^{3} t^{3-\alpha} L(t).
\end{equation*}
\end{proposition}

\begin{proof}
	Linearity of $\EE[N_t]$ in $t$ follows by a classical argument \citep{fellerIntroductionProbabilityTheory1971}. From \cite{pekalp2020plug}, we obtain an explicit expression for the variance in \eqref{eq:delayed_renewal_var}. 
	
	Now, assume that $\overline{F}(t) \asymp t^{-\alpha} L(t)$. To bound $U^0(s) -  m_\mu  s$, we use the result by \cite{teugels1968renewal}, that leads to
	$$
	U^0(t) - m_\mu t \asymp \frac{t^{2-\alpha}}{m_\mu^{-2}(\alpha - 1)(2 - \alpha)} \asymp m_\mu^2 t^{2-\alpha} L(t)
$$
as $t \to \infty$. Using Karamata's integral theorem for regularly varying functions, if $L$ is bounded on every compact subset of $[0, \infty)$, we have
$
	\int_0^t (U^0(s) - m_\mu s) \dif s \asymp \frac{t^{3-\alpha}}{3-\alpha} L(t)
$, and we infer for $t \to \infty$ that
\begin{align*}
	\Var(N_t) \asymp m_\mu t + 2 m_\mu^{3} t^{3-\alpha} L(t).&\qedhere
\end{align*}
\end{proof}
Next we show that such a process experiences a degradation of speed of convergence in distribution. 
Put $T_n = \sum_{i=1}^n S_i$ and $T_n' = S_1' + \sum_{i=2}^n S_i$ where $S_1' \eqinlaw S_2$ is independent from $S_2, S_3, \dots$. The following result is classical \citep{gnedenko1968limit}, see also \citet[Theorem 4.5.2]{whitt2002stochastic}.

\begin{theorem} \label{prop:renewal_clt}
	Let $S_1', S_2, S_2, \dots$ be a sequence of i.i.d.\ non-negative random variables with $\EE[S_1'] = 1/m_\mu$ and a common cumulative distribution function $F$ such that $\overline F(t)=t^{-\alpha} L(t)$ for some $\alpha \in (1,2)$ and slowly varying function $L$. Then,
	\begin{equation} \label{eq:th_S_n_distrlimit}
		c_n^{-1} (T'_n - n/m_\mu) \konvD \LC_\alpha
	\end{equation}
	in $\RR$ as $t \to \infty$, where $\LC_{\alpha}$ is a stable non-degenerate random law. The normalizing sequence $(c_n)_{n \in \NN}$ can be chosen such that $c_n = n^{1/\alpha} L_0(n)$ for some slowly varying function $L_0$ (generally different from $L$).
	Moreover, if $\overline F(t) \asymp A t^{-\alpha}$ when $t \to \infty$, we can choose $c_n $ as $ c n^{1/\alpha}$ for some $c>0$.
\end{theorem}

Observe now that $T_n = T_n' - S_1'+ S_1$, implying that the same convergence result holds for the delayed $T_n$.  Next,  since $N_t=0 \vee \max\{n \in \NN: T_{n} \leq t\}$ 
is a \emph{generalized inverse} of $T_n$, suitable theory can be applied to  obtain (after appropriate renormalization) a distributional limit for $N_t - m_\mu t$. By \citet[Theorem 7.3.1]{whitt2002stochastic} we infer distributional convergence 
\begin{equation}
	\frac{1}{c(t)}(N_t- m_\mu t) \konvD - m_\mu^{-(1+1/\alpha)} \LC_{\alpha}  
\end{equation}
in $\RR$ as $t \to \infty$, where $c(t) = t^{1/\alpha}L_0(t)$ (or, if $1-F(t) \asymp A t^{-\alpha}$ as $t \to \infty$, $c(t) = c \, t^{1/\alpha}$).

\subsection{Proofs for minimax lower bounds}

\begin{proof}[Proof of \Cref{prop:KR_minimax}$(i)$]
We first reduce the left-hand side in \eqref{eq:minimaxLowerBoundMeasure} to a minimax lower bound derived by \cite{singh2018minimax}. To this end, observe that 
\begin{align*}
    \inf_{\tilde \mu_t}\sup_{\mu\in\mathcal{M}(\XC,1)} \EE\left[ \KR_{p,C}(\tilde \mu_t, \mu) \right] &\geq  \inf_{\tilde \mu_t}\sup_{\mu\in\mathcal{P}(\XC)} \EE\left[ \KR_{p,C}(\tilde \mu_t, \mu) \right],
  \end{align*}   
  where in the infimum in taken over measure estimators in $\MC(\XC)$. To reduce the range of the measure estimators to $\PC(\XC)$, we utilize a $(p,C)$-KRD projection of measures. More precisely, based on \citet[Proposition 7.33]{bertsekas1996stochastic} it follows from compactness of $\PC(\XC)$ and continuity of $\KR_{p,C}$ under weak convergence (\Cref{prop:weak_conv_KRD}) that there exists a Borel measurable map $\Psi\colon \MC(\XC) \to \PC(\XC)$ such that $\Psi(\tilde \mu) \in \argmin_{\mu\in\PC(\XC)}\KR_{p,C}(\tilde \mu, \mu)$. 
  Then, it follows for any measure estimator $\tilde \mu_t\in \MCsX$ and $\mu \in \PC(\XC)$ from triangle inequality and the definition for $\Psi$ that 
  \begin{align*}
    \KR_{p,C}( \Psi(\tilde \mu_t), \mu)\leq \KR_{p,C}( \Psi(\tilde \mu_t), \tilde \mu_t) + \KR_{p,C}( \tilde \mu_t, \mu)\leq 2 \KR_{p,C}(\tilde \mu_t, \mu). 
  \end{align*}
  Hence, we can replace the infimum over measure estimators in $\tilde \mu_t\in \MC(\XC)$ by an infimum over probability measure estimators $\overline \mu_t \in \PC(\XC)$ at the cost of a factor $1/2$, 
  \begin{align*}
    \inf_{\tilde \mu_t}\sup_{\mu\in\mathcal{P}(\XC)} \EE\left[ \KR_{p,C}(\tilde \mu_t, \mu) \right] &\geq \frac{1}{2}\inf_{\overline \mu_t}\sup_{\mu\in\mathcal{P}(\XC)} \EE\left[ \KR_{p,C}(\overline \mu_t, \mu) \right].
  \end{align*}
  By conditioning on $N_t$, the right-hand side in the above display can further be rewritten to
  \begin{align*}
    &\frac{1}{2}\inf_{\overline \mu_t}\sup_{\mu\in\mathcal{P}(\XC)} \EE\left[\EE\left[ \KR_{p,C}\left(\overline \mu_t\left( \sum_{i = 1}^{N_t}\delta_{X_i}\right), \mu\right) \middle| N_t \right]\right]\\
    &=\frac{1}{2}\inf_{\overline \mu_t}\EE\left[\sup_{\mu\in\mathcal{P}(\XC)} \EE\left[ \KR_{p,C}\left(\overline \mu_t\left(\sum_{i = 1}^{N_t}\delta_{X_i}\right), \mu\right) \middle| N_t \right]\right],
\end{align*}
where the last equality follows from the fact $N_t\sim \Poi(t)$ for every $\mu \in \PC(\XC)$. In particular, conditionally on $N_t = n\in \NN_0$, it follows that $\sum_{i = 1}^{N_t}\delta_{X_i}$ has the same law as $\sum_{i = 1}^{n}\delta_{X_i'}$ for $n$ i.i.d.\ random variables $X_1', \dots, X_n'\sim \mu$. Further, note that conditionally on $N_t = n$, any measure estimator $\overline \mu_t(\sum_{i = 1}^{n}\delta_{X_i})\in \PC(\XC)$ can also be written as a measurable estimator $\tilde \mu_n(X_1, \dots, X_n)\in \PC(\XC)$, hence it follows that 
\begin{align*}
    \sup_{\mu\in \PC(\XC)}\EE\left[ \KR_{p,C}\left(\overline \mu_t\left(\sum_{i = 1}^{N_t}\delta_{X_i}\right), \mu\right) \middle| N_t = n\right]&\geq \inf_{\overline \mu_n}\sup_{\mu\in \PC(\XC)}\EE_{\mu^{\otimes n}}\left[ \KR_{p,C}\left(\overline \mu_n(X_1, \dots, X_n), \mu\right) \right]\\
   & = \inf_{\overline \mu_n}\sup_{\mu\in \PC(\XC)}\EE_{\mu^{\otimes n}}\left[ \OT_{d^p\wedge C^p}^{1/p}\left(\overline \mu_n(X_1, \dots, X_n), \mu\right) \right],
\end{align*}
where the infimum on the right-hand sides is taken over all measurable estimators $\overline \mu_n\colon \XC^n\to \PC(\XC)$ based on $n$ i.i.d.\ random variables $X_1, \dots, X_n\sim \mu$. 
For $n = 0$, i.e., when no observations are available, it holds $\sup_{\mu, \nu\in \MCsX}\KR_{p,C}(\mu, \nu) \geq(\diam(\XC)\wedge C)$, which implies that the minimax risk is lower bounded by $(\diam(\XC)\wedge C)/2$. Moreover, for $n\in \NN$, according to \citet[Theorem 2]{singh2018minimax} the right-hand side in the above display is lower bounded up to a universal constant as in terms of the packing radius\footnote{\label{fn:packingRadius}Given a metric space $(\XC, d)$ and  $k\in \NN$ the packing radius $\mathcal{R}(k, \XC, d)$ denotes the largest possible minimum distance among $k$ points in $\XC$, i.e., it is defined as $\max_{x_1, \dots, x_k\in \XC }\min_{i\neq j} d(x_i,x_j)$.}
\begin{align}\label{eq:packingLowerbound}
    \sup_{k\in \{2, \dots, 32n\}} \mathcal{R}(k,\XC,d_{p,C}) \left(\frac{k-1}{n}\right)^{1/2p}.
\end{align}
To lower bound the packing radius, we utilize \Cref{lem:impactCoveringC}.  
Choosing $ k = 2$ we see that \eqref{eq:packingLowerbound} is lower bounded by $(C\wedge \diam(\XC))n^{-1/2p}$. Choosing $k = (n+1)$ we obtain (up to a constant which depends on $\alpha, A, \diam(\XC)$) for \eqref{eq:packingLowerbound} the lower bound $C\wedge (n^{-1/\alpha})$. And finally, under $\lceil C^{-\alpha}\rceil\leq n$ we also consider $k = \lceil C^{-\alpha}\rceil+1$ which implies that $\mathcal{R}(k,\XC,d_{p,C})\gtrsim_{\alpha, A, \diam(\XC)} C \wedge (\lceil C^{-\alpha}\rceil^{-1/\alpha}) \gtrsim_\alpha C$ and hence \eqref{eq:packingLowerbound} is lower bounded up to constants which depend on $\alpha$, $A$ and $\diam(\XC)$ by $C \cdot \left(\lceil C^{-\alpha}\rceil/n \right)^{1/2p}\asymp_{p,\alpha} C^{1-\alpha/2p}n^{-1/2p}$. Combining these bounds  with \Cref{lem:PoissonRateLowerbound} yields for $t>0$ sufficiently large and any $C>0$ that 
\begin{align*}
    &\inf_{\tilde \mu_t}\EE\left[\sup_{\mu\in\mathcal{P}(\XC)} \EE\left[ \KR_{p,C}\left(\tilde \mu\left( \frac{1}{t}\sum_{i = 1}^{N_t}\delta_{X_i}\right), \mu\right) \middle| N_t \right]\right]\\
   & \gtrsim_{A, p, \alpha} \EE\left[ C\wedge \left(  CN_t^{-1/2p} + C^{1-\alpha/2p} N_t^{-1/2p} + N_t^{-1/\alpha} \right) \mathds{1}(N_t\geq 1) +  C\mathds{1}(N_t= 0)\right]  \\
   &\geq \EE\left[ C\wedge \left(  C(N_t+1)^{-1/2p} + C^{1-\alpha/2p} (N_t+1)^{-1/2p} + (N_t+1)^{-1/\alpha} \right)\right]  \\
    &\gtrsim_{p,\alpha} C\wedge \left(Ct^{-1/2p} + C^{1-\alpha/2p}t^{-1/2p} + t^{-1/\alpha}\right).\qedhere
\end{align*}
\end{proof}

\begin{proof}[Proof of \Cref{prop:KR_minimax}(ii)]
The proof is divided into three steps, we first show the lower bound $C\wedge C t^{-1/{2p}}$, then show the lower bound $C\wedge (C^{1-\alpha/2} (t\log(t))^{-1/2p})$, and finally establish the dimension dependent lower bound $C\wedge ((t\log(t))^{-1/\alpha})$. The first step is based on LeCam's two-point method (see, e.g, \citealt{tsybakov2009introduction}), the second step relies on the relation (\Cref{prop:interpolate_C}) between the $(p,C)$-KRD and the total variation distance in conjunction with minimax lower bounds for estimating the total variation distance, which utilizes prior work by  \cite{jiao2018minimax}, and the third step builds on a prior analysis by \cite{niles2019estimation} on estimating the Wasserstein distance. Notably, the last two steps rely on the method of two fuzzy hypotheses \citep[Theorem 2.14]{tsybakov2009introduction}.

\emph{Step 1}: Let $\mu \in \PC(\XC)$ and define $\nu^{\epsilon}\coloneqq (1-\epsilon)\mu\in \mathcal{M}(\XC, 1)$ for $\epsilon\in [0,1)$. Then it follows from \Cref{prop:kr_properties}$(iii)$ that $$\KR_{p,C}(\mu, \nu^\epsilon) - \KR_{p,C}(\mu, \nu^0)= \KR_{p,C}(\mu, \nu^\epsilon)  = \frac{C}{2^{1/p}}|m_\mu - m_{\nu^\epsilon}|^{1/p} \geq \frac{C}{2^{1/p}} \epsilon^{1/p}\geq \frac{C}{2} \epsilon^{1/p}. $$
Further, note that $tm_{\hat \mu_t}\sim \Poi(t)$ and $tm_{\hat \nu^\epsilon_t}\sim \Poi((1-\epsilon)t)$ and observe for the Kullback-Leibler divergence between the respective laws that \begin{align*}
  \KL{\Poi(t), \Poi((1-\epsilon)t)} = t \log(1/(1-\epsilon)) - \epsilon t \leq t \left(\frac{1}{1-\epsilon}-1-\epsilon\right) = \frac{t\epsilon^2}{1-\epsilon}.
\end{align*}
Hence, for pairs of distributions $(\mu, \nu^{0})$ and $(\mu, \nu^{\epsilon_t})$ for $\epsilon_t = t^{-1/2}/2$ we obtain by the tensorization property of the Kullback-Leibler divergence for $t\geq 1$ that \begin{align*}
   &\KL{\Poi(t)\otimes \Poi(t),\Poi(t)\otimes  \Poi((1-\epsilon_t)t)}= \KL{\Poi(t), \Poi((1-\epsilon_t)t)} \leq 1/4.
\end{align*}
This implies by \citet[Theorem 2.2]{tsybakov2009introduction} that the left-hand side in \eqref{eq:minimaxLowerBoundEstimationKRD} is lower bounded for $t\geq 1$ up to a universal constant by $C\epsilon_t^{1/p} = C t^{-1/2p} = C\wedge \left(C t^{-1/2p}\right)$.   

\emph{Step 2}: %
For the second lower bound, we aim to reduce the minimax risk for the KRD to that of estimating the $p$-th root of the total variation norm. To this end, for $C>0$ consider a $C$-separated packing of $(\XC, d)$, i.e., a collection of points $\XC_C \coloneqq \{x_1, \dots, x_m\}\in \XC$ such that $d(x_i, x_j) \geq C$ for all $i\neq j$. Then, for two measures $\mu, \nu \in \MC(\XC_C)$ it follows by \Cref{prop:interpolate_C} that $\KR_{p,C}(\mu, \nu) = \frac{C}{2^{1/p}}\cdot \textup{TV}(\mu, \nu)^{1/p}$. 

The minimax risk of the total variation distance has been derived by \citet[Theorem 3]{jiao2018minimax}.
Specifically, if $t/\log t\ge |\XC_C|$ with $|\XC_C|\geq 2$ and if there exists some $\kappa>0$ such that $\log t\ge\kappa\log|\XC_C|$, they establish that there exists $\kappa'=\kappa'(\kappa)>0$ for which the additional condition
$\sqrt{\frac{|\XC_C|}{t\log t}}\;\ge\;\kappa'\!\left(\sqrt{\frac{\log t}{t}}+\frac{\sqrt{|\XC_C|}\,\log t}{t}\right)$
implies for $\xi = \frac{1}{|\XC_C|}\sum_{x\in \XC_C}\delta_{x}y$ the uniform distribution on $\XC_C$ that
\begin{align*}
  \inf_{\widehat{\textup{TV}}_{t}}\sup_{\mu\in \PC(\XC_C)} \EE\left[ \left| \widehat{\textup{TV}}_t -\textup{TV}(\mu, \xi)\right|^2 \right]\gtrsim_{\kappa} 1 \wedge \left(\frac{|\XC_C|}{t\log(t+1)}\right),%
\end{align*}
where the infimum is taken over all measurable estimators $\widehat{\textup{TV}}_{t}$ based on $N_t \sim \Poi(t)$ i.i.d.\ random variables $X_1, \dots, X_{N_t}\sim \mu$ and the estimator may depend in particular on $\xi$. This result does not immediately enable inferring that the minimax risk for estimating $\textup{TV}^{1/p}(\mu, \xi)$ in mean absolute deviation is of order $ 1 \wedge \left(\frac{|\XC_C|}{t\log(t+1)}\right)^{1/2p}$. Nonetheless, by modifying their proof technique, i.e., by adapting the method of two fuzzy hypotheses (see Theorem 2.14 of \cite{tsybakov2009introduction} for the original method and \Cref{lem:TsybakovExtension} below for the adaptation to estimating the $p$-th root of the functional) we show in \Cref{prop:minimax_lowerbound_pth-root-TV} (below) under the same assumptions imposed on $|\XC_C|$ and $t$ that 
\begin{align}\label{eq:minimaxLowerBound_TVpth_root}
  \inf _{\widehat{\textup{TV}}_t} \sup _{\mu \in \mathcal{P}(\XC_{C})} \mathbb{E}\left[\left|\widehat{\textup{TV}}_t^{1/p}-\textup{TV}^{1/p}(\mu,\xi)\right|\right] \gtrsim_{p, \kappa} 1\wedge \left( |\XC_C|^{1/2p} t^{-1/2p}\log(t)^{-1+1/2p}\right),
\end{align}
where the infimum is once again taken over all measurable estimators $\widehat{\textup{TV}}_{t}$ based on $N_t \sim \Poi(t)$ i.i.d.\ random variables $X_1, \dots, X_{N_t}\sim \mu$ and the estimator may depend in particular on $\xi$. 

To infer minimax lower bounds for the KRD, we need to quantify $|\XC_C|$ in terms of $C$ and then translate the imposed assumption on $|\XC_C|$ and $t$ to $C$ and $t$. For the first aspect, note by Lemma \ref{lem:impactCoveringC} that $\max_{x_1, \dots, x_m\in \XC}\min_{i\neq j} d(x_i,x_j)\wedge C\geq C\wedge \left(\frac{1}{2}(a/m)^{1/\alpha}\right).$ Hence, it follows that we may choose $\XC_C$ with cardinality at least $m= \lfloor a(2C)^{-\alpha}\rfloor-1$. Consequently, the assumptions for the validity \eqref{eq:minimaxLowerBound_TVpth_root} reduce to 
\begin{enumerate}
  \item[$(i)$] $\lfloor a(2C)^{-\alpha}\rfloor-1\geq 2$, which is equivalent to $C\leq (3/a)^{-1/\alpha},$ 
  \item[$(ii)$]  $t/\log(t) \geq \lfloor  a(2C)^{-\alpha}\rfloor-1$, which is implied  using $(i)$ by
  $t/\log(t) \geq a(2C)^{-\alpha}/2$ which is in turn is equivalent to $C\geq (2 t/(a\log(t)))^{-1/\alpha}/2$,
  \item[$(iii)$] 
   $\log(t) \geq \kappa \log(|\XC_C|)$ which is implied by $t\geq a(2C)^{-\alpha\kappa}$ or equivalently $C\leq 2(t/a)^{-1/(\alpha\kappa)}$,
  \item[$(iv)$] and for some $\kappa'= \kappa'(\kappa)>0$ the inequality $\sqrt{\frac{|\XC_C|}{t\log t}}\;\ge\;\kappa'\!\left(\sqrt{\frac{\log t}{t}} + \frac{\sqrt{|\XC_C|}\log(t)}{t}\right)$. Two sufficient conditions for the validity of these inequalities are given by $|\XC_C| \geq (\kappa'/2)^2 \log(t)^2$, which is implied by $|\XC_C|\geq \log(t)^3$ for $t$ sufficiently large, and $\frac{t}{\log(t)^3}\geq (k'/2)^2$ which is met for $t$ sufficiently large. The first of these two conditions is further implied by $C \leq (2\log^3(t)/a)^{-1/\alpha}/2$.
\end{enumerate}
Based on the assumption that $C\in I(t)$, and recalling that $\KR_{p,C}(\mu, \nu) = \frac{C}{2^{1/p}}\cdot \textup{TV}(\mu, \nu)^{1/p}$ for $\mu, \nu \in \PC(\XC)$, we obtain for $t$ sufficiently large that
\begin{align*}
   \inf_{\hat \KR_t} \sup_{\mu, \nu\in \MC(\XC,1)} \EE\left[\left|\hat\KR_t - \KR_{p,C}(\mu, \nu)\right|\right] &\geq \; \inf_{\hat \KR_t} \sup_{\mu\in \PC(\XC_C)} \EE\left[\left|\hat\KR_t - \KR_{p,C}(\mu, \xi)\right|\right]\\
  &\geq \;C \inf _{\widehat{\textup{TV}}_t} \sup _{\mu \in \mathcal{P}(\XC_{C})} \mathbb{E}\left[\left|\widehat{\textup{TV}}_t^{1/p}-\textup{TV}^{1/p}(\mu,\xi)\right|\right]\\
  &\gtrsim_{a,  p, \alpha, \kappa} \;C \wedge \left(C^{1-\alpha/2p} t^{-1/2p} \log(t)^{-1+1/2p}\right),
\end{align*}
which proves the desired lower bound. 

\emph{Step 3}: %
To derive the lower bound for the regime $\alpha>2p$, we first show a lower bound for the minimax risk of estimating the KRD between probability measures on $\XC$, or equivalently the quantity $\OT_{d^p\wedge C^p}^{1/p}$ (\Cref{prop:interpolate_C}), and where we consider estimators based on $n\in \NN$ i.i.d.\ observations from both measures. To this end, we follow along the proof of Theorem~11 in \cite{niles2019estimation}, which asserts under the covering number assumption on $\XC$ a minimax lower bound for the estimation of the Wasserstein distance with metric $d$.  A key ingredient in their proof is their Proposition 9, which forms the basis for the method of two fuzzy hypotheses and which we now adapt to the setting when the original metric is replaced by the truncated metric $d\wedge C$. 

\begin{lemma}\label{lem:adaptation_Prop9_NWR19}
  Assume $\alpha>2p \geq 2$ and let $(\XC,d)$ be a metric space which fulfills $\diam(\XC) = 1$ and $\NC(\epsilon, \XC, d)\epsilon^{\alpha}\in [a,A]$ for all $\epsilon\in (0,1]$ for some $A\geq a>0$. Further, let $m\in\NN$ be large enough such that $\frac{1}{2}(a/m)^{1/\alpha}\leq C$ and denote by $u$ the uniform distribution on $\{1, \dots,m\}$. Then, there is a random function $F\colon \{1, \dots, m\}\to \XC$ such that for any probability distribution~$q$ on $\{1,\dots, m\}$ it holds 
  \begin{align*}
    \frac{1}{2}(a/m)^{1/\alpha} \textup{TV}(q,u)^{1/p}&\leq \OT_{d^p\wedge C^p}^{1/p}(F_{\#}q, F_{\#}u) %
    \lesssim_{A,p,\alpha} m^{-1/\alpha}(\chi^2(q,u))^{1/\alpha}\textup{TV}(q,u)^{1/p-2/\alpha}
  \end{align*}
  with probability at least 0.9.
\end{lemma}
\begin{proof}[Proof of Lemma \ref{lem:adaptation_Prop9_NWR19}]
  We follow along the same construction as in \citet[Proposition 9]{niles2019estimation}. To this end, we consider a maximal $m$-packing $\XC_m\coloneqq \{x_1, \dots, x_m\}$ of $\XC$, which by our assumptions on the covering number of $\XC$ and the size of $m$ by \Cref{lem:impactCoveringC} that $\min_{i\neq j} d(x_i,x_j)\wedge C \geq \frac{1}{2}(a/m)^{1/\alpha}$. Further, we select $F$ uniformly at random among the bijections between $\{1, \dots, m\}$ and $\{x_1, \dots, x_m\}$. For the lower bound, we observe since any two distinct elements in $\XC_m$ are separated with respect to $d\wedge C$ by $\frac{1}{2}(a/m)^{1/\alpha}$ that 
  \begin{align*}
    \OT_{d^p\wedge C^p}^{1/p}(F_{\#}q, F_{\#}u) \geq \frac{1}{2}(a/m)^{1/\alpha}\textup{TV}(q,u).
  \end{align*}
  This confirms the lower bound with probability 1. For the upper bound we note that $\OT_{d^p\wedge C}(\cdot, \cdot) \leq \OT_{d^p}(\cdot, \cdot)$ on $\PC(\XC_m)$, hence the assertion follows from the upper bound in \citet[Proposition 9]{niles2019estimation} for the non-truncated setting.
\end{proof}

The definition of the random function $F$ in \Cref{lem:adaptation_Prop9_NWR19} yields that \begin{align*}
  \inf_{\widehat \OT_n} \sup_{\mu, \nu\in \PC(\XC)} \EE\left[\left|\widehat\OT_n - \OT_{d^p\wedge C}^{1/p}(\mu, \nu)\right|\right] &\geq \inf_{\widehat \OT_n} \sup_{q\in \PC(\{1, \dots, m\})} \EE\left[\left|\widehat\OT_n - \OT_{d^p\wedge C}^{1/p}(F_{\#}u, F_{\#}q)\right|\right],
\end{align*}
where the infimum on the left-hand side is taken over measurable estimators based on i.i.d.\ random variables from $\mu$ and $\nu$ in $\PC(\XC)$, whereas on the right-hand side we consider i.i.d.\ random variables from $F_{\#}q$ and $F_{\#}u$ with $q,u\in \PC(\{1, \dots, m\})$ and $u$ denoting the uniform distribution. Moreover, the validity of the inequalities in \Cref{lem:adaptation_Prop9_NWR19} for the regime where $m\in \NN$ is large enough such that $\frac{1}{2}(a/m)^{1/\alpha}\leq C$ implies that lower bounds for the right-hand side in the above display can be inferred from minimax lower bounds for estimating the TV norm as long as it is bounded away from zero and the $\chi^2$-distance is suitably bounded. For this regime, as part of the proof of Theorem 11 in \cite{niles2019estimation}, the authors confirm for the choice $m = \lceil \kappa n \log(n+1)\rceil$ for some sufficiently large universal constant $\kappa>0$ and $n$ sufficiently large that
\begin{align*}
  &\inf_{\widehat \OT_n} \sup_{q\in \PC(\{1, \dots, m\})} \EE\left[\left|\widehat\OT_n - \OT_{d^p\wedge C}^{1/p}(F_{\#}u, F_{\#}q)\right|\right]\\
  \geq & \frac{1}{32}\left(\frac{a}{m}\right)^{1/\alpha} \inf_{\widehat \OT_n} \sup_{q\in \PC(\{1, \dots, m\})}   \PP\left[\left|\widehat\OT_n - \OT_{d^p\wedge C}^{1/p}(F_{\#}u, F_{\#}q)\right|\geq \left(\frac{a}{m}\right)^{1/\alpha}\right]\\
  \geq & \frac{0.8}{32}\left(\frac{a}{m}\right)^{1/\alpha} \gtrsim_{a, p, \alpha} (n \log(n+1))^{-1/\alpha}. 
\end{align*}
Note that these minimax lower bounds only find validity for the regime where $\frac{1}{2}(a/m)^{1/\alpha}\leq C$, i.e., when $a/(2C)^{\alpha}\leq m =  \lceil \kappa n \log(n)\rceil$. Whenever $n$ is too small, then we may nonetheless infer for $n^*$, the smallest possible sample size such that the above inequalities hold and for which $a/(2C)^{\alpha}\leq m =  \lceil \kappa n \log(n)\rceil$ is met, that
\begin{align*}
   &\inf_{\widehat \OT_n} \sup_{q\in \PC(\{1, \dots, m\})} \EE\left[\left|\widehat\OT_n - \OT_{d^p\wedge C}^{1/p}(F_{\#}u, F_{\#}q)\right|\right] \gtrsim_{a, A, \alpha} (n^* \log(n^*))^{-1/\alpha} \gtrsim_{a, \alpha} C ,
\end{align*}
since any estimator based on $n$ data points can also be interpreted as an estimator for $n^*$ data points but where the latter $n^*-n$ are not used. Combining all previous displays, we arrive at 
\begin{align*}
  \inf_{\widehat \OT_n} \sup_{\mu, \nu\in \PC(\XC)} \EE\left[\left|\widehat\OT_n - \OT_{d^p\wedge C}^{1/p}(\mu, \nu)\right|\right]  \gtrsim_{a, A, \alpha} C\wedge (n\log(n+1))^{-1/\alpha}, 
\end{align*}
for all $C\in (0,1]$ and all $n\in \NN$.  Since according to \Cref{prop:interpolate_C} it holds $\KR_{p,C}(\mu, \nu ) = \OT_{d^p\wedge C^p}^{1/p}(\mu, \nu)$ for all $\mu, \nu \in \PC(\XC)$, we thus also infer that 
\begin{align*}
  \inf_{\hat \KR_n} \sup_{\mu, \nu\in \PC(\XC)} \EE\left[\left|\hat\KR_n - \KR_{p,C}(\mu, \nu)\right|\right]  \gtrsim_{a, A, \alpha} C\wedge (n\log(n+1))^{-1/\alpha}. 
\end{align*}
Now, to conclude the desired bound for samples from a Poisson point process from here, first condition on the number of observations $N_t$ and $M_t$ from both Poisson point processes 
\begin{align*}
  &\inf_{\hat \KR_t} \sup_{\mu, \nu\in \MC(\XC,1)} \EE\left[\left|\hat\KR_t - \KR_{p,C}(\mu, \nu)\right|\right]\\
  &\geq \; \inf_{\hat \KR_t} \sup_{\mu, \nu\in \PC(\XC)} \EE\left[\EE\left[\left|\hat\KR_t - \KR_{p,C}(\mu, \nu)\right| \Big| N_t, M_t\right]\right]\\
  &\geq \; \EE\left[\inf_{\hat \KR_{N_t \vee M_t}} \sup_{\mu, \nu\in \PC(\XC)} \EE\left[\left|\hat\KR_{N_t\vee M_t} - \KR_{p,C}(\mu, \nu)\right| \Big| N_t, M_t\right]\right]\\
  &\gtrsim_{a,A,\alpha} \; \EE\left[ C\wedge (N_t \vee M_t)\log(N_t\vee M_t+1) \mathds{1}(N_t\vee M_t\geq 1) +  C\mathds{1}(N_t\vee M_t= 0) \right]\\
  &\geq \; \EE\left[ C\wedge ((N_t \vee M_t+1)\log(N_t\vee M_t+2))^{1/\alpha}\right]\\
  &\geq \; C\wedge (t \log(t))^{-1/\alpha},
\end{align*}
where the first inequality follows from restricting the supremum from $\MC(\XC,1)$ to $\PC(\XC)$, the second utilizes that the distribution of $N_t$ and $M_t$ does not depend on $\mu, \nu$. In the third inequality we use that any estimator based on $N_t$ and $M_t$ can also be interpreted as an estimator of two collection of $N_t\vee M_t$ i.i.d.\ samples. The fourth inequality is a consequence of the above considerations, and the final inequality follows from \Cref{lem:PoissonRateLowerbound} for $t$ sufficiently large and all $C>0$.
\end{proof}

The remainder of this appendix is concerned with proving minimax lower bounds for estimating the $p$-th root of the TV distance over a finite domain. Most lower bounds on functional estimation results, such as those presented in \cite{jiao2018minimax}, see also \cite{lepski1999estimation,cai2011testing,niles2019estimation, han2021optimal,wang2021optimal}, rely extensively on the method of two fuzzy hypotheses, see \cite[Theorem 2.14]{tsybakov2009introduction} for comprehensive treatment. The following lemma is a straightforward adaptation of this method for estimating the $p$-th root of certain functionals.

\begin{lemma}\label{lem:TsybakovExtension}
    Let $\Theta$ be a measurable space and $F\colon \Theta\to [0, \infty)$ a Borel-measurable functional. Denote by $\mu_0, \mu_1$ two probability distributions on $\Theta$. Further, denote by $\{P_\theta\}_{\theta\in \Theta}$ a collection of probability measures on a measurable space $(\XC, \mathcal{A})$. Assume the following. \begin{enumerate}
        \item[$(i)$] There exist $b, c, s\geq 0$, and $0\leq \beta_0, \beta_1<1$ such that 
        \begin{align*}
            &\mu_0(\theta \colon F(\theta)\leq c) \geq 1-\beta_0,\\
            &\mu_1(\theta \colon F(\theta)\geq c + 2s, F(\theta) \leq b ) \geq 1-\beta_1.
        \end{align*}    
    \item[$(ii)$]For each measurable set $A\in \mathcal{A}$ the map $\theta\mapsto P_\theta(A)$ is measurable.  Further, for the posterior probability measures $\mathbb{P}_0$ and $\mathbb{P}_1$ on $(\XC, \mathcal{A})$, defined for $A\in\mathcal{A}$ and $j \in \{0,1\}$ as $\mathbb{P}_j(A) \coloneqq \int P_\theta(S) \dif\mu_j(\theta)$, there exist $\tau>0$ and $0<\alpha<1$ such that 
    \begin{align*}
        \mathbb{P}_1\left(\frac{\dif\mathbb{P}_0^a}{\dif\mathbb{P}_1}\geq \tau \right)\geq 1-\alpha,
    \end{align*}
    where $\mathbb{P}_0^a$ is the absolutely continuous component of $\mathbb{P}_0$ with respect to $\mathbb{P}_1$. 
\end{enumerate}
Then, for any $p\geq 1$ and a measurable estimator $\hat F_p\colon \XC \to [0,\infty)$, 
    \begin{align*}
        \sup_{\theta\in \Theta}P_\theta\left( |\hat F_p - F^{1/p}(\theta)|\geq s  b^{\frac{1-p}{p}}p^{-1}\right) \geq \frac{\tau(1-\alpha - \beta_1)-\beta_0}{1+\tau}.
    \end{align*}
\end{lemma}
\begin{remark}
    Assumption $(i)$ in \Cref{lem:TsybakovExtension} necessitates that $b\geq c+2s$, which overall implies that $F(\theta)$ is bounded with high probability by $b$ under $\mu_0$ and $\mu_1$. Under this assumption, the pair of distributions $\mu_0, \mu_1$ designed for a minimax lower bound on $F(\theta)$ immediately implies a related minimax lower bound for $F^{1/p}(\theta)$.  In particular, if $b = \kappa s$ for $\kappa\geq 2$, then the lower bound simplifies to $p^{-1} \kappa^{(1-p)/p} s^{1/2p} \asymp_{p,\kappa} s^{1/2p}$. 
\end{remark}

\begin{proof} For $p = 1$ the proof is stated, without the constraint $F(\theta)\leq b$ in condition $(i)$, in \citet[Theorem 2.14]{tsybakov2009introduction}.
For $p>1$ we take the functional $F^{1/p}$ and note that
\begin{align*}
    \mu_0(\theta \colon F^{1/p}(\theta)\geq C^{1/p}) = \mu_0(\theta \colon F(\theta)\geq c) \geq 1 - \beta_0.
\end{align*}
Further, using the inequality $a^p - b^p \leq p(a-b)a^{p-1}$ for $a\geq b\geq 0$ it follows under $\mu_1$ with probability at least $1-\beta_0$ that 
\begin{align*}
    F^{1/p}(\theta) - C^{1/p} \geq \frac{F(\theta) - c}{pF^{\frac{p-1}{p}}(\theta)}\geq 2s  b^{\frac{1-p}{p}}p^{-1},
\end{align*}
which implies the assertion by invoking \citet[Theorem 2.14]{tsybakov2009introduction}.
\end{proof}

Based on our adaptation in \Cref{lem:TsybakovExtension} we now use the hypotheses considered in the proof for the minimax lower bound by \cite{jiao2018minimax} on estimating the TV-distance on a finite domain to infer a related statement for the $p$-th root functional. 

\begin{proposition}\label{prop:minimax_lowerbound_pth-root-TV}
	Let $t\in [e,\infty), S\in \NN$ and define $\XC_S = \{1, \dots, S\}$ and the uniform distribution $U \coloneqq S^{-1}\sum_{i = 1}^{S} \delta_{i}$ on $\XC_S$. Assume that $t/\log t \geq S$ and that there exists a constant $\kappa>0$ such that  $\log t \geq \kappa \log S$ with $S \geq 2$. Then, there exists a constant $\kappa^{\prime}>0$ that only depends on $\kappa$ such that if $\sqrt{\frac{S}{t \log t}}\geq \kappa^{\prime}\left(\sqrt{\frac{\log t}{t}}+\frac{\sqrt{S} \log t}{t}\right)$, then  for $p\geq 1$ it holds for $t>0$ %
\begin{equation}\label{eq:minimax_lowerbound_pth-root-TV}
\inf _{\widehat{\textup{TV}}_t} \sup _{P \in \mathcal{P}(\XC_{S})} \mathbb{E}\left[\left|\widehat{\textup{TV}}_t^{1/p}-\textup{TV}^{1/p}(P,U)\right|\right] \gtrsim_{p, \kappa}1 \wedge \left(S^{1/2p} t^{-1/2p}\log(t)^{-1+1/2p}\right),%
\end{equation}
where the infimum is taken over all possible estimators based on realizations of a Poisson point process with intensity measure $t \cdot P$.
\end{proposition}

\begin{proof}
The proof of the minimax lower bound by \cite{jiao2018minimax} for estimating the TV-distance between probability measures relies on a variant of the method of two fuzzy hypotheses \citep[Theorem 2.15]{tsybakov2009introduction} (but where  Assumption $(i)$ in \Cref{lem:TsybakovExtension} with $c$ and $c+2s$ are parametrized as $\zeta-s$ and $\zeta+s$ respectively, for suitable $\zeta\geq s>0$). Their two fuzzy hypotheses are denoted by $\mu_0, \mu_1$. For our assertion, we utilize the same construction and utilize \Cref{lem:TsybakovExtension}. To this end, it suffices to find $b\geq \zeta+s>0$ such that $\textup{TV}(P,U)\leq b$ for $(\mu_0+\mu_1)$-almost every $P$. %

To define the fuzzy hypothesis $\mu_0$ and $\mu_1$, considered by \cite{jiao2018minimax}, first define auxiliary measures
 $\tilde \mu_0$ and $\tilde \mu_1$, 
\begin{align*}
  \tilde \mu_0 \coloneqq  \left(\bigotimes_{j = 1}^{S-1} \tilde \mu_0^{S}\right) \otimes \delta_{1/S}\quad \text{ and } \quad 
   \tilde \mu_1 \coloneqq \left(\bigotimes_{j = 1}^{S-1} \tilde \mu_1^{S}\right) \otimes \delta_{1/S},
\end{align*}
for certain probability measures $\tilde \mu_0^S, \tilde \mu_1^S$ on $\RR$ with mean $S^{-1}$. Further, by assumption we have $S^{-1} \geq  \log(t)/t >  \log(t)/2t$. Therefore, case 2 of the construction in \cite{jiao2018minimax} for the probability measures $\tilde \mu_0^S$ and $\tilde \mu_1^S$ yields measures defined on the interval $[q-\sqrt{\log(t)/(2St)}, q+\sqrt{\log(t)/(2St)}]$. Hence, for $(\tilde \mu_0+\tilde \mu_1)$-almost every $P$ we obtain \begin{align}\label{eq:boundTV-norm-fuzzy-hypothesis}
  \textup{TV}(P,U) = \sum_{i =1}^{S-1} |p_i-1/S| \leq \sqrt{\frac{S\log(t)}{2t}}\eqqcolon b.
\end{align}
According to Equation (170) of \cite{jiao2018minimax} the two fuzzy hypotheses $\mu_0$ and $\mu_1$ are now defined by conditioning onto suitable sets $E_0$ and $E_1$ (defined in Equation (158)), respectively, of non-negative measures which nearly sum up to one. In particular, the bound in \eqref{eq:boundTV-norm-fuzzy-hypothesis} is still valid for $(\mu_0+\mu_1)$-almost every $P$ since the set $E_i$ admits positive mass under $\tilde \mu_i$ for $i\in \{0,1\}$. 

In addition, according Equations (151), (154), (157), and (171--172) from \cite{jiao2018minimax} it follows for $c = \zeta -s $ and $c +2s = \xi +s$ with  \begin{align*}
  \zeta &\coloneqq \EE_{\mu_0}[\textup{TV}(P,U)] + \chi/2, \quad s \coloneqq \chi/4,\quad \chi \coloneqq \EE_{\mu_1}[\textup{TV}(P,U)] - \EE_{\mu_0}[\textup{TV}(P,U)] \gtrsim_{\kappa} \sqrt{\frac{S}{t \log(t)}} 
\end{align*}
that Assumption $(i)$ of \Cref{lem:TsybakovExtension} is met for $\mu_0$ and $\mu_1$ with $\beta_0=\beta_1 = 0$. 
Moreover, Assumption $(ii)$ of \Cref{lem:TsybakovExtension} follows from Equations (173--179) of \cite{jiao2018minimax} by invoking \citet[Theorem 2.15(i)]{tsybakov2009introduction} with $\alpha = o(1)$ for $t\to \infty$. Hence, by \Cref{lem:TsybakovExtension} we infer, \begin{align*}
  \inf _{\widehat{\textup{TV}}_t} \sup _{P \in \mathcal{P}(\XC_{S})} \mathbb{P}\left(\left|\widehat{\textup{TV}}_t^{1/p}-\textup{TV}^{1/p}(P,U)\right| \geq  sb^{\frac{1-p}{p}} p^{-1}\right) &\gtrsim 1-\alpha = 1-o(1),
\end{align*}
where the last equality is to be understood for the regime $t \to \infty$. Hence, by Markov's inequality, we obtain 
\begin{align*}
  \inf _{\widehat{\textup{TV}}_t} \sup _{P \in \mathcal{P}(\XC_{S})} \mathbb{E}\left[\left|\widehat{\textup{TV}}_t^{1/p}-\textup{TV}^{1/p}(P,U)\right|\right]&\geq (1-\alpha)  sb^{\frac{1-p}{p}} p^{-1} \\
  &\gtrsim_{\kappa} (1-o(1)) \left( \frac{S \log(t)}{t }\right)^{\frac{1-p}{2p}} \left(\frac{S}{t\log(t)}\right)^{1/2} \\
  &=(1-o(1)) S^{1/2p} t^{-1/2p} \log(t)^{-1+1/2p}.\qedhere
\end{align*}

\end{proof}

\section{Technical results} \label{sec:aux_results}

\begin{lemma}[Impact of $C$ on covering and packing numbers]\label{lem:impactCoveringC}
    For a totally bounded metric space  $(\XC,d)$ the following assertions hold. 
    \begin{enumerate}
        \item[$(i)$] For $C>0$ it follows $\NC(\epsilon, \XC, d\wedge C)  = \mathds{1}(\epsilon\geq C) + \mathds{1}(\epsilon<C)\NC(\epsilon, \XC, d)$.
        \item[$(ii)$] Assume there exists some $\alpha, \, C>0$ such that $\NC(\epsilon, \XC, d)\geq c \, \epsilon^{-\alpha}$ for  all $\epsilon \in (0,\diam(\XC)]$. Then, for any integer $k\in \NN$ the packing radius (recall its definition in \Cref{fn:packingRadius}) is lower bounded by $\mathcal{R}(k, \XC, d\wedge C)\geq C\wedge \left(\frac{1}{2}(c/k)^{1/\alpha}\right)$, and it holds $\mathcal{R}(2, \XC, d\wedge C) = C \wedge \diam(\XC)$. 
    \end{enumerate}
    \end{lemma}
    \begin{proof}
    For Assertion $(i)$ we note by definition of covering numbers for $\epsilon \geq C$ that since $d(x,y)\wedge C\leq C\leq \epsilon$ for every $x,y\in \XC$, hence any single point defines an $\epsilon$-covering, so $\NC(\epsilon, \XC, d\wedge C) = 1$. Moreover, for $\epsilon< C$ consider a minimal $\epsilon$-covering $\{x_1, \dots, x_m\}$ of $\XC$ with respect to $d$. Then $\sup_{x\in \XC}\min_{i = 1, \dots, m}d(x,x_i)\leq \epsilon$ and hence $\sup_{x\in \XC}\min_{i = 1, \dots, m}d(x,x_i)\wedge C\leq \epsilon\wedge C = \epsilon$, which yields $\NC(\epsilon, \XC, d\wedge C)\leq \NC(\epsilon, \XC, d)$. Conversely, note that a  minimal $\epsilon$-covering $\{x_1, \dots, x_m\}$ of $\XC$ with respect to $d\wedge C$ also defines an $\epsilon$-covering of $\XC$ with respect to $d$. Thus, for $\epsilon< C$ we have $\NC(\epsilon, \XC, d\wedge C) = \NC(\epsilon, \XC, d)$. 
    
    To show Assertion $(ii)$ we first note by \citet[Lemma 6]{niles2019estimation} that there exist by assumption $k$ elements $x_1, \dots, x_k\in \XC$ such that $\min_{i\neq j} d(x_i,x_j)\geq \frac{1}{2}(c/k)^{1/\alpha}$. This implies  $\min_{i\neq j} d(x_i,x_j)\wedge C\geq C\wedge \left(\frac{1}{2}(c/k)^{1/\alpha}\right)$ and yields the claim. The special case for $k = 2$ follows from $\max_{x,y\in \XC} d(x,y) = \diam(\XC)$. 
    \end{proof}
    
\begin{lemma} \label{lem:sum_Var_and_cov}
	For a sequence $T_1, T_2, \dots$ of bounded, identically distributed random variables, such that for $t\in \NN$ it holds $\Cov(T_1, T_t) = t^a$ with $a > -1$, it holds 
	$$
		\Var\left(\sum_{i=1}^t T_i\right) = t \, \Var(T_1) + 2 \sum_{m=1}^{t-1} (t-m) m^a \asymp t \, \Var(T_1) + \frac{2}{(a+1)(a+2)} t^{a + 2}
	$$
	for large $t$.
\end{lemma}
\begin{proof}
	Observe that
	$$
		\Var\left(\sum_{i=1}^t T_i\right) = \sum_{i=1}^t \Var\left( T_i\right)  + 2 \sum_{1 \leq n < m \leq t} \Cov(T_n, T_m) = t \,\Var\left( T_1\right) + 2 \sum_{1 \leq n < m \leq t} (m-n)^a.
	$$
	For each $N \coloneqq n-m$ there are exactly $t-N$ ordered pairs $(n,m)$ with this gap, so we can rewrite the sum as
	$$
		\sum_{1 \leq n < m \leq t} (m-n+1)^a = \sum_{N=1}^{t-1} (t - N)N^a = t \sum_{N=1}^{t-1} N^a - \sum_{N=1}^{t-1} N^{a+1},
	$$
	and, using the integral test, we conclude that
	$$
		\Var\left(\sum_{i=1}^t T_i\right) = t \Var\left( T_1\right) + \frac{2}{(a+1)(a+2)} t^{a+2} + o(t^{a+2})
	$$
	as $t \to \infty$ as desired.
\end{proof}

\begin{lemma}\label{lem:PoissonRateLowerbound}
Let $N_t, M_t\sim \Poi(t)$ be independent Poisson random variables with parameter $t\geq 1$. Then, for $\alpha>0$ and $\beta\geq 0$ it follows for $t$ sufficiently large and every $C>0$ that 
\begin{align*}
    &\quad \EE\left[ C\wedge \left((N_t+1)^{-\alpha}\log(N_t+2)^{-\beta}\right) \right]\\
   &\geq \EE\left[ C\wedge \left(([N_t\vee M_t]+1)^{-\alpha}\log([N_t\vee M_t]+2)^{-\beta}\right) \right]\\
   &\gtrsim_{\alpha, \beta} C\wedge \left(t^{-\alpha}\log(t)^{-\beta}\right).
\end{align*}
\end{lemma}
\begin{proof} 
    The first inequality follows since $N_t \vee M_t \geq N_t$ (deterministically) and the function $t \in [0,\infty)\mapsto C\wedge (t+1)^{-\alpha}\log(t+2)^{-\beta}$ is non-increasing. For the second inequality, we show in the following for $t$ sufficiently large that \begin{align}\label{eq:asympBehaviorN}
  ([N_t\vee M_t]+1)^{-\alpha}\log([N_t\vee M_t]+2)^{-\beta} = t^{-\alpha}\log(t)^{-\beta} + \mathcal{O}_p(t^{-1/2 -\alpha}\log(t)^{-\beta}).
\end{align}
To this end, note by the central limit theorem for Poisson random variables in conjunction with Slutzky's lemma for $t \to \infty$ and $\eta\geq 0$ fixed that $$t^{1/2}\left(\frac{(N_t+ \eta ,M_t+ \eta)^\intercal}{t}- 1\right) %
\konvD \mathcal{N}(0,\textup{Diag}(1,1)).$$ By the delta method for the directionally differentiable maximum function $(x,y) \in \RR^2 \mapsto (x \vee y)$ we thus obtain for $t\to \infty$ that
\begin{align}\label{eq:WeakLimitPoisson}
    t^{1/2}\left(\frac{[N_t\vee M_t]+ \eta}{t}- 1\right) %
\konvD  \mathcal{N}_2(0,1),
\end{align}
where $\mathcal{N}_2(0,\sigma^2)$ represents the distribution of the maximum of two independent centered normal random variables with variance $\sigma^2>0$. Invoking the delta method for \eqref{eq:WeakLimitPoisson} with $\eta = 1$ and the function $x\in (0,\infty)\mapsto x^{-\alpha}$, it follows that
\begin{align}\label{eq:PowerLimitNonZero}
	t^{1/2}\left(([N_t\vee M_t]/t + 1/t)^{-\alpha} - 1\right)\konvD \mathcal{N}_2(0,\alpha^2). 
\end{align}
Utilizing the delta method once again on \eqref{eq:WeakLimitPoisson} with $\eta=2$ for the map $x\in (0,\infty)\mapsto \log(x)$ it follows $t^{-1/2}\log(([N_t\vee M_t] + 2)/t) \to \mathcal{N}_2(0,1)$, and by another application of the delta method for $x\mapsto (1+x)^{-\beta}$ it follows for $\beta\geq 0$ and $t\to \infty$ that 
\begin{align*}
    & t^{1/2}\log(t)^{1+\beta} \left(\log([N_t\vee M_t]+2)^{-\beta}-\log(t)^{-\beta}\right)\\=\;  &t^{1/2}\log(t)\left(\left(1+ \frac{\log\left(\frac{[N_t\vee M_t]+2}{t}\right)}{\log(t)}\right)^{-\beta} - 1\right) \konvD \mathcal{N}_2(0,\beta^2).
\end{align*}
Slutzky's lemma thus implies for $t\to \infty$ convergence in probability \begin{align}\label{eq:LogLimitZero}
  t^{1/2}\left(\left(1+ \frac{\log\left(\frac{[N_t\vee M_t]+2}{t}\right)}{\log(t+1)}\right)^{-\beta} - 1\right) \konvP 0.
\end{align}
Combining \eqref{eq:PowerLimitNonZero} and \eqref{eq:LogLimitZero} with the delta method for the map $(x,y) \in \RR^2\mapsto x\cdot y$ we obtain  
\begin{align*}
  & t^{1/2 + \alpha}\log(t+1)^{\beta}\left( ([N_t\vee M_t]+1)^{-\alpha}\log([N_t\vee M_t]+2)^{-\beta} - t^{-\alpha}\log(t)^{-\beta} \right) \\
  = \;& t^{1/2}\left(\left(\frac{[N_t\vee M_t] +1}{t}\right)^{-\alpha}\left(1+ \frac{\log\left(\frac{[N_t\vee M_t]+2}{t}\right)}{\log(t)}\right)^{-\beta} - 1\right) \konvD \mathcal{N}_2(0,\alpha^2), 
\end{align*}
which confirms the validity of \eqref{eq:asympBehaviorN}. In consequence, there exist some positive constants $\kappa= \kappa(\alpha, \beta)\in (0,1)$ and $T= T(\alpha, \beta)\in [1,\infty)$ such that for all $t\geq T$ it holds with probability at least $\kappa$ that 
\begin{align*}
 ([N_t\vee M_t]+1)^{-\alpha}\log([N_t\vee M_t]+2)^{-\beta} &\geq t^{-\alpha}\log(t)^{-\beta} -t^{-1/2 -\alpha}\log(t)^{-\beta}/2\\
 &=  t^{-\alpha}\log(t)^{-\beta}(1- t^{-1/2}/2)\\
 &\geq 2^{-1} t^{-\alpha}\log(t)^{-\beta}. 
\end{align*}
Hence, for every $C>0$ and $t\geq T$ we infer that with probability at least $\kappa$ it holds \begin{align*}
 C\wedge  \left(([N_t\vee M_t]+1)^{-\alpha}\log([N_t\vee M_t]+2)^{-\beta}\right) &\geq C\wedge \left( 2^{-1} t^{-\alpha}\log(t)^{-\beta}\right)\\
 &\gtrsim C\wedge \left( t^{-\alpha}\log(t)^{-\beta}\right),
\end{align*}
The assertion now follows at once from Markov's inequality since for $t\geq T$ we have 
\begin{align*}
  \EE\left[ C\wedge \left(([N_t\vee M_t]+1)^{-\alpha}\log([N_t\vee M_t]+2)^{-\beta}\right) \right] &\geq \kappa C\wedge \left( t^{-\alpha}\log(t)^{-\beta}\right) \\ &\gtrsim_{\alpha, \beta}C\wedge \left( t^{-\alpha}\log(t)^{-\beta}\right). \qedhere
\end{align*}
\end{proof}

 \newpage
\addcontentsline{toc}{section}{References}
\small

\end{document}